\documentclass[a4paper, reqno, twoside]{amsart}
\usepackage{etex}
\usepackage{amsfonts,amssymb,graphics,amsthm,amsmath}
\usepackage{latexsym}
\usepackage[2emode]{psfrag}
\usepackage{mathrsfs}
\usepackage[all]{xy}
\usepackage{enumerate}
\usepackage[latin1]{inputenc}
\usepackage{lscape}
\usepackage{a4wide}
\usepackage{txfonts}
\usepackage{multicol}
\usepackage[bookmarks=false]{hyperref} 
\usepackage{wasysym}
\usepackage{enumitem,cleveref}
\usepackage{anyfontsize}
\usepackage{version}
\usepackage{anyfontsize}

\usepackage[all]{pstricks}
\usepackage{pst-text,pst-node,pst-coil}
\usepackage{tikz}

\usepackage{cancel}
\usepackage[normalem]{ulem}
\usetikzlibrary{arrows,shapes,backgrounds}

\usepgflibrary{arrows}
\usetikzlibrary{topaths}
\usetikzlibrary{er}

\makeatletter
\DeclareMathSizes{10}{10}{6}{5}
\makeatother

\newtheorem{theorem}{\sc Theorem}[section]
\newtheorem{proposition}[theorem]{\sc Proposition}

\newtheorem{lemma}[theorem]{\sc Lemma}
\newtheorem{corollary}[theorem]{\sc Corollary}
\theoremstyle{definition}
\newtheorem{definition}[theorem]{\sc Definition}

\newtheorem{example}[theorem]{\sc Example}

\theoremstyle{remark}
\newtheorem{remark}[theorem]{\sc Remark}

\newenvironment{invisible}{{\noindent\sc \underline{\color{blue}Invisible (To be hidden)}:\quad}\color{red}}{\medskip}
\IfFileExists{tcilatex.tex}{\input{tcilatex}}{}

\excludeversion{invisible}

\newcommand{\tensor}[1]{\otimes_{\Sscript{#1}}}

\newcommand{\Sf}[1]{\mathsf{#1}}
\newcommand{\fk}[1]{\mathfrak{#1}}
\newcommand{\cat}[1]{\mathcal{#1}}
\newcommand{\rmod}[1]{\Sf{Mod}_{\Sscript{#1}}}
\newcommand{\lmod}[1]{{}_{\Sscript{#1}}\Sf{Mod}}
\newcommand{\bimod}[1]{{}_{\Sscript{#1}}{\Sf{ Mod}}_{\Sscript{#1}}}

\renewcommand{\hom}[3]{\mathrm{Hom}_{\Sscript{#1}}\left(#2,\,#3\right)}

\newcommand{\td}[1]{\widetilde{#1}}

\newcommand{\bara}[1]{\overline{#1}}

\newcommand{\op}[1]{{#1}^{\Sscript{\mathrm{op}}}}
\newcommand{\cop}[1]{{#1}^{\Sscript{\mathrm{cop}}}}
\newcommand{\lr}[1]{\left(\underset{}{} #1 \right)}

\newcommand{\ring}[1]{#1\text{-}\mathsf{Rings}}
\newcommand{\End}[2]{\mathrm{End}_{\Sscript{#1}}(#2)}

\newcommand{\coring}[1]{\mathfrak{#1}}
\newcommand{\lcomod}[1]{ {}_{#1}\mathsf{Comod}}
\newcommand{\rcomod}[1]{ \mathsf{Comod}{}_{#1}}

\newcommand{\can}[1]{\mathsf{can}_{#1}}

\newcommand{\peque}[1]{\scriptscriptstyle{#1}}

\newcommand{\Ae}{A^{\Sscript{e}}}
\newcommand{\LR}[1]{\left\{\underset{}{} #1 \right\}}
\newcommand{\bcirc}{{\boldsymbol{\circ}}}

 \newcommand{\id}[1]{id_{\Sscript{#1}}}

\newcommand{\lfun}[5]{{#1} \colon {#2} \longrightarrow {#3};\quad  \Big({#4} \longmapsto {#5}\Big)}

\newcommand{\Hom}[3]{\mathsf{Hom}_{\Sscript{#1}}\left({#2},{#3}\right)}

\newcommand{\Nat}[1]{\mathsf{Nat}\left({#1}\right)}

\newcommand{\Calg}{\mathrm{CAlg}_{\K }}
\newcommand{\proj}[1]{\mathsf{proj}(#1)}

\renewcommand{\ker}[1]{\mathrm{Ker}\left({#1}\right)}
\newcommand{\what}[1]{\widehat{#1}}

\definecolor{bostonuniversityred}{rgb}{0.8, 0.0, 0.0}

\newcommand{\ie}{i.e.~}
\newcommand{\eg}{e.g.~}

\newcommand{\prim}{{\mathrm{Prim}}}
\newcommand{\indec}[1]{{\mathrm{Q}}({#1})}
\newcommand{\Endk}[1]{\mathrm{End}_{\Sscript{\K}}(#1)}
\newcommand{\rend}[2]{\mathrm{End}(#1_{#2})}

\newcommand{\rcomatrix}[2]{#2^* \tensor{#1} #2}
\newcommand{\infcomatrix}[2]{#2^{\Sscript{\dag}}\tensor{\Sscript{\bara{#1}}} #2}

\newcommand{\coanillos}[1]{#1\text{-}\mathsf{Corings}}

\newcommand{\ldot}{\,.\,}

\newcommand{\Aug}{\mathsf{Alg}_{\K}^{+}}
\newcommand{\Coaug}{\mathsf{Coalg}_{\K}^{+}}
\newcommand{\Coalgk}{\mathsf{Coalg}_{\K}}
\newcommand{\Algk}{\mathsf{Alg}_{\K}}
\newcommand{\veck}{\mathsf{Vec}_{\K}}

\newcommand{\N}{\mathbb{N}}
\newcommand{\K}{\Bbbk}
\newcommand{\C}{\mathbb{C}}
\newcommand{\R}{\mathbb{R}}

\newcommand{\cC}{\mathscr{C}}
\newcommand{\dD}{\mathscr{D}}

\newcommand{\hH}{\mathscr{H}}
\newcommand{\iI}{\mathscr{I}}
\newcommand{\jJ}{\mathscr{J}}

\newcommand{\lL}{\mathscr{L}}

\newcommand{\oO}{\mathscr{O}}

\newcommand{\rR}{\mathscr{R}}
\newcommand{\sS}{\mathscr{S}}

\newcommand{\xX}{\mathscr{X}}
\newcommand{\yY}{\mathscr{Y}}

\newcommand{\cA}{{\mathcal A}}

\newcommand{\cD}{{\mathcal D}}
\newcommand{\cE}{{\mathcal E}}
\newcommand{\cF}{{\mathcal F}}
\newcommand{\cG}{{\mathcal G}}
\newcommand{\cH}{{\mathcal H}}
\newcommand{\cI}{{\mathcal I}}
\newcommand{\cJ}{{\mathcal J}}
\newcommand{\cK}{{\mathcal K}}
\newcommand{\cL}{{\mathcal L}}
\newcommand{\cM}{{\mathcal M}}

\newcommand{\cO}{{\mathcal O}}
\newcommand{\cP}{{\mathcal P}}

\newcommand{\cR}{{\mathcal R}}
\newcommand{\cS}{{\mathcal S}}
\newcommand{\cT}{{\mathcal T}}
\newcommand{\cU}{{\mathcal U}}
\newcommand{\cV}{{\mathcal V}}

\newcommand{\cX}{{\mathcal X}}
\newcommand{\cY}{{\mathcal Y}}

\newcommand{\Sscript}[1]{{#1}}
\newcommand{\script}[1]{\scriptstyle{#1}}
\newcommand{\due}[3]{{}_{{#2 }} {#1}_{{ #3}}\,}

\newcommand{\HAlgd}{\mathsf{CHAlgd}_{\Sscript{\K}}}
\newcommand{\Halgd}[1]{\mathsf{CHAlgd}_{\Sscript{#1}}}
\newcommand{\CHalgd}[1]{\mathsf{CCHAlgd}_{\Sscript{#1}}}
\newcommand{\GHalgd}[1]{\mathsf{GCHAlgd}_{\Sscript{#1}}}
\newcommand{\Ms}{M_{\Sscript{s}}}
\newcommand{\Mt}{M_{\Sscript{t}}}

\newcommand{\As}{A_{\Sscript{s}}}
\newcommand{\At}{A_{\Sscript{t}}}
\newcommand{\Ht}{\cH_{\Sscript{t}}}
\newcommand{\Hs}{\cH_{\Sscript{s}}}
\newcommand{\Hp}{\cH_{\Sscript{p}}}
\newcommand{\Hq}{\cH_{\Sscript{q}}}
\newcommand{\Mp}{M_{\Sscript{p}}}
\newcommand{\Mq}{M_{\Sscript{q}}}

\newcommand{\Der}[3]{\mathrm{Der}_{\Sscript{#1}}(#2, \, #3)}
\newcommand{\Derk}[3]{\mathrm{Der}_{\Sscript{\K}}{}^{\Sscript{#1}}(#2, \, #3)}
\newcommand{\pis}{\pi^{\Sscript{s}}}
\newcommand{\pit}{\pi^{\Sscript{t}}}
\newcommand{\pip}{\pi^{\Sscript{p}}}

\newcommand{\Hst}{{}_{\Sscript{s}}\cH_{\Sscript{t}}}

\newcommand{\psis}{\psi^{\Sscript{s}}}
\newcommand{\psit}{\psi^{\Sscript{t}}}
\newcommand{\sH}{{}_{\Sscript{s}}\cH}
\newcommand{\tH}{{}_{\Sscript{t}}\cH}
\newcommand{\sHt}{{}_{\Sscript{s}}\cH_{\Sscript{t}}}
\newcommand{\tHs}{{}_{\Sscript{t}}\cH_{\Sscript{s}}}
\newcommand{\DerHs}{\mathrm{Der}_{\Sscript{\cH}}{}^{\Sscript{s}}(\cH, \, \cH)}
\newcommand{\Aep}{{A}_{\Sscript{\varepsilon}}}
\newcommand{\Bep}{{B}_{\Sscript{\varepsilon}}}
\newcommand{\Ders}[2]{\mathrm{Der}_{\Sscript{#1}}(#2)}
\newcommand{\LieR}[1]{\mathsf{LieRin}_{\Sscript{#1}}}
\newcommand{\Scr}[1]{\mathscr{#1}}
\newcommand{\fomega}{\boldsymbol{\omegaup}}

\newcommand{\Acom}[1]{\cA^{\Sscript{#1}}}
\newcommand{\Amod}[1]{\cA_{\Sscript{#1}}}
\newcommand{\Circ}[1]{{#1}^{\circ}}
\newcommand{\Sigmad}{\Sigma^{\Sscript{\dag}}}
\newcommand{\VL}{\cV_{\Sscript{A}}(L)}

\newcommand{\Coac}{\mathsf{Coac}}
\newcommand{\Coring}{\mathsf{Coring}}

\newcommand{\vi}{v_{\Sscript{(0)}}}
\newcommand{\vii}{v_{\Sscript{(1)}}}
\newcommand{\uo}{u_{\Sscript{(1)}}}
\newcommand{\uoo}{u_{\Sscript{(2)}}}

\newcommand{\wi}{w_{\Sscript{(0)}}}
\newcommand{\wii}{w_{\Sscript{(1)}}}
\newcommand{\vdos}{v_{\Sscript{(2)}}}

\newcommand{\udos}{u_{\Sscript{(2)}}}
\newcommand{\utres}{u_{\Sscript{(3)}}}

\newcommand{\ao}{a_{\Sscript{(0)}}}
\newcommand{\aoo}{a_{\Sscript{(1)}}}

\newcommand{\Grpd}{{\rm Grpds}}
\newcommand{\Rxe}{R_{x\varepsilon}}
\newcommand{\Rg}{R_{g}}

\newcommand{\Sets}{\mathrm{Sets}}
\newcommand{\Pir}{\pi_{\Sscript{R}}}
\newcommand{\Pis}{\pi_{\Sscript{S}}}

\newcommand{\hHo}{\hH_{\Sscript{0}}}
\newcommand{\hHa}{\hH_{\Sscript{1}}}

\newcommand{\Go}{\cG_{\Sscript{0}}}
\newcommand{\Ga}{\cG_{\Sscript{1}}}

\newcommand{\ida}{id_{\Sscript{A}}}

\newcommand{\taua}{\tau_{\Sscript{A}}}
\newcommand{\taur}{\tau_{\Sscript{R}}}
\newcommand{\taus}{\tau_{\Sscript{S}}}
\newcommand{\tauh}{\tau_{\Sscript{\cH}}}
\newcommand{\ptauh}{\tau'_{\Sscript{\cH}}}

\newcommand{\ptaur}{\tau'_{\Sscript{R}}}

\newcommand{\alphar}{\alpha_{\Sscript{R}}}

\newcommand{\Sigtr}{\Sigma^{\Sscript{\tau}}_{\Sscript{R}}}

\newcommand{\Gamax}{\Gamma(\xX)}
\newcommand{\Gamay}{\Gamma(\yY)}

\newcommand{\rsaft}[1]{{#1}^\bullet}
\newcommand{\ldual}[1]{{}^*#1}
\newcommand{\rdual}[1]{#1{}^{*}}

\newcommand{\rcirc}[1]{#1{}^{\circ}}
\newcommand{\mb}[1]{\boldsymbol{#1}}

\newcommand{\Cinfty}[1]{\mathcal{C^{\infty}}(#1)}
\newcommand{\Rr}{\mathbb{R}}
\newcommand{\flecha}[1]{\overset{\to}{#1}}

\newcommand{\cXR}{\cX\otimes R}

\begin{document}
\allowdisplaybreaks

\title[Towards differentiation and integration between  Hopf algebroids and Lie algebroids.]{Towards differentiation and integration between Hopf algebroids and Lie algebroids.}
\author{Alessandro Ardizzoni}
\address{University of Turin, Department of Mathematics ``Giuseppe Peano'', via Carlo Alberto 10, I-10123 Torino, Italy}  \email{alessandro.ardizzoni@unito.it}
\urladdr{https://sites.google.com/site/aleardizzonihome/home}
\author{Laiachi El Kaoutit}
\address{Universidad de Granada, Departamento de \'{A}lgebra and IEMath-Granada. Facultad de Ciencias. Fuente Nueva s/n. E18071, \indent Granada, Spain }
\email{kaoutit@ugr.es}
\urladdr{http://www.ugr.es/~kaoutit}
\author{Paolo Saracco}
\address{D\'epartement de Math\'ematique, Universit\'e Libre de Bruxelles, Boulevard du Triomphe, B-1050 Brussels, Belgium.}
\email{paolo.saracco@ulb.ac.be}
\urladdr{sites.google.com/site/paolosaracco}
\date{\today}
\thanks{Research supported by   the Spanish Ministerio de Econom\'ia y Competitividad and the European Union, grant  MTM2016-77033-P. This paper was written while the first and third authors were members of the ``National Group for Algebraic and Geometric Structures and their Applications'' (GNSAGA-INdAM). The third author is also grateful to FNRS for support and his postdoc fellowship through the MIS grant ``Antipode'' (application number 31223212).}

\begin{abstract}
In this paper we set up the  foundations around the notions of formal differentiation and formal integration in the context of commutative Hopf algebroids and Lie-Rinehart algebras. Specifically, we construct a contravariant functor from the category of  commutative Hopf algebroids with a fixed base algebra to that of Lie-Rinehart algebras over the same algebra, \emph{the differentiation functor}, which can be seen as an algebraic counterpart to the differentiation process from  Lie groupoids to Lie algebroids. The other way around,  we provide two interrelated contravariant functors form the category of Lie-Rinehart algebras to that  of commutative Hopf algebroids, \emph{the integration functors}. One of them yields a contravariant adjunction together with the differentiation functor. Under mild conditions, essentially on the base algebra, the other integration functor only induces an adjunction at the level of Galois Hopf algebroids. By employing the differentiation functor, we also analyse the geometric \emph{separability} of a given morphism of Hopf algebroids. Several examples and applications are presented along the exposition.
\vspace{-1.2cm}
\end{abstract}

\maketitle

\begin{footnotesize}
\tableofcontents
\end{footnotesize}

\pagestyle{headings}

\section{Introduction}
We will describe the motivations behind the ideas of this work and give an algebraic overview on the classical theory of differentiation and integration in the context of both algebraic and differential geometry. Thereafter, we will  briefly discuss the main results of this paper in sufficient details, aiming to make this  summary, as far as possible, self-contained.

\subsection{Motivations and overviews}
The main motivation behind the research of this paper is to provide foundational tools  for the formal development of the \emph{differentiation} and \emph{integration} in the context of Hopf algebroids and Lie-Rinehart algebras, both over the same base algebra. Thus, we  hereby propose to establish, in terms of contravariant adjunctions, a relation between these two latter classes of objects, hoping to leave a paved path for the study of  integration problems  in this context.  Our main results are presented  as  Theorems \ref{thm:TeoremoneA}  and \ref{thm:AA} of this introduction, together with Theorem \ref{thm:AB} as an application.  The exposition includes also two Appendices, where we  offer alternative approaches and/or clarifications  on some topics we have  discussed before in the text.

In the framework of  Lie algebras and Lie groups, that is, in the domain of  differential geometry, the notions of  ``differentiation'' and ``integration'' are involved in the outstanding Lie's third theorem. Classically, differentiation means to assign a finite-dimensional Lie algebra to each Lie group (namely, its tangent vector space at the identity point). Conversely, integration constructs a Lie group out of a given finite-dimensional Lie algebra (in fact, a connected and simply connected Lie group).

For affine group schemes, that is, in the context of algebraic geometry, both  notions   are introduced in a similar way. Specifically, starting with an affine  group scheme $G$, one assigns to it the Lie algebra of all derivations from the  associated Hopf algebra $\cO(G)$ to the base field, taking as a point the counit of the Hopf algebra structure of this ring (the identity point). This assignment is functorial and (by abuse of terminology) can be termed the ``\emph{differentiation functor}''.  Conversely, if a  Lie algebra is given, then the \emph{finite dual} of its universal enveloping algebra acquires  a structure of commutative Hopf algebra and so it leads in a functorial way to an affine group scheme. This procedure might be called the (formal) ``\emph{integration functor}''.  

In a more general ``algebraic way'', these two functors  induce a contravariant adjunction between the category of Lie algebras and that of commutative Hopf algebras. More precisely, if $\K$ denotes a ground base field, $\mathsf{Lie}_{\Sscript{\K}}$ and $\mathsf{CHopf}_{\Sscript{\K}}$ denote, respectively, the categories of Lie $\K$-algebras and of commutative  Hopf $\K$-algebras, then we have a contravariant adjunction
\begin{equation}\label{Eq:hola}
\xymatrix@C=45pt{ \cI: \mathsf{Lie}_{\Sscript{\K}}  \ar@<0.6ex>@{->}|(.4){}[r]  &  \ar@<0.6ex>@{->}|(.4){}[l]    \mathsf{CHopf}_{\Sscript{\K}}  : \cL }
\end{equation}
explicitly given as follows. For every Lie algebra $L$ and Hopf algebra $H$, we have $\cI(L)=U(L)^{\circ}$ (the finite, or Sweedler's, dual Hopf algebra of the universal enveloping algebra\footnote{This is  also the commutative Hopf algebra constructed as the coend of the fiber functor attached to the symmetric monoidal category of finite-dimensional $L$-representations. It is called the algebra of \emph{representative functions} on $U(L)$ in  \cite[\S2]{Hochschild:1959}.}) and $\cL(H)={\rm Der}_{\Sscript{\K}}(H, \K_{\Sscript{\varepsilon}})$ (the vector space of derivations with coefficients in the $H$-module $\K$ via the counit $\varepsilon$ of $H$). 

Thus, the unit and counit of adjunction  \eqref{Eq:hola} provide us  with a more conceptual way of how to  relate Lie algebras with commutative Hopf algebras (playing here the role of ``groups'' associated with them). Specifically, let us denote by $\Theta_{\Sscript{L}}:  L \to \cL(\cI(L))$ the unit at a Lie algebra $L$ and by $\Psi_{\Sscript{H}}: H  \to \cI(\cL(H))$ the counit at a Hopf algebra $H$. Then it is known from the literature that, in characteristic zero,  $\Theta_{\Sscript{L}}$ is injective for any finite-dimensional Lie algebra $L$ (see Remark \ref{rem:Smooth} for  a proof), while, for an  affine algebraic group $G$,  $\Psi_{\Sscript{\cO(G)}}$ is injective if and only if  $G$ is connected (see, e.g., \cite[0.3.1(g)]{Takeuchi:75}).   It is noteworthy to mention that $\Theta_{\Sscript{L}}$ is not an  isomorphism even for  some trivial finite-dimensional Lie algebras. For example, in the case of one dimensional abelian Lie $\mathbb{C}$-algebra $\fk{a}$,  the Hopf algebra $\cI(\fk{a})$ splits as a tensor product of two Hopf algebras (\cite[Example 9.1.7]{Montgomery:1993}) in such a way that it possesses at least two  linearly independent derivations with values in $\mathbb{C}$; whence $\Theta_{\Sscript{\fk{a}}}$ is not  surjective. 
However,  over an algebraically closed field of characteristic zero,  if the given finite-dimensional Lie $\Bbbk$-algebra $L$ coincides with its derived Lie algebra (i.e., $L=[L, L]$, e.g.~when $L$ is semisimple), then $\Theta_{\Sscript{L}}$ is surjective by \cite[Theorem 6.1(3)]{Hochschild:1970}, and so an isomorphism. As a consequence, the restriction $\cI'$ of the functor $\cI: \mathsf{Lie}_{\Sscript{\K}}  \to \op{\left(\mathsf{CHopf}_{\Sscript{\K}}\right)}$ to the full-subcategory of all those finite-dimensional Lie algebras $L$ such that $L = [L,L]$, is fully faithful. In view of \cite[II, \S6, n$\textsuperscript{o}$ 2, Corollary 2.8, page 263]{DemazureGabriel} and \cite[Theorem 3.1]{Hochschild:1970}, any object $L$ of that subcategory is an algebraic Lie algebra, that is, $L=\mathsf{Lie}(G)$  the Lie algebra of a connected and simply connected affine algebraic  group $G$. It turns out that $\cO(G)$ is a finitely generated Hopf algebra, it is an integral domain, it has no proper affine unramified extensions, $\cL(\cO(G)) \cong L$ and, moreover, it can be identified with $\cI(L)$ (see \cite[top of page 57 and Theorem 4.1]{Hochschild:1970}). Therefore, if we corestrict $\cI'$ to its essential image (i.e., the full subcategory of all those finitely generated Hopf algebras which are integral domains and have no proper affine unramified extension and such that $\cL(H) = [\cL(H),\cL(H)]$), it induces an anti-equivalence of categories.
Not less important is the fact that the adjunction \eqref{Eq:hola}, when restricted to a certain class of real Hopf algebras (see \cite[Corollary 3.4.4, page 162]{Abe}), can be seen as a categorical reformulation of Lie groups differentiation and integration.

Now, if we want to extend these constructions to a category wider than that of groups (respectively commutative Hopf algebras), for example that of groupoids (resp.~commutative Hopf algebroids), then several obstructions show up, specially in the construction of the integration functor (or functors). For instance,  it  is well-known (see \cite[\S 3.5]{Mackenzie}) that to each Lie groupoid, one can attach ``in a functorial way'' a Lie algebroid (for the reader's sake, we included some details in Appendix \ref{ssec:LA-LG}), but  there are Lie algebroids which do not integrate to Lie groupoids. However, we point out that there are conditions which guarantee the integrability (see e.g. \cite{Crainic/Fernandes:2003} and  \cite{Fernandes}).

In the same lines as before, if we want to think of a Hopf algebra, instead of a (Lie) group, then the  closest algebraic
prototype of a (Lie) groupoid is a commutative Hopf algebroid\footnote{Note that Morita theory of Lie groupoids behaves in a similar way as for commutative Hopf algebroids, see \cite{Kaoutit/Kowalzig:14} for details.}. However, in contrast with the case of Lie groups\footnote{In this case for every Lie group we have, in contravariant functorial way, the commutative real Hopf algebra of smooth representative functions.}, as far as we know, there is no functorial way to go directly  from the category of Lie groupoids to that of commutative Hopf algebroids.  Nevertheless, there is a well-defined functor from the category of Lie algebroids (overs a fixed connected smooth real manifold $\cM$) to the category of complete topological and commutative Hopf algebroids (with  $\Cinfty{\cM}$ as a base algebra), that is,  formal affine groupoid schemes (see \cite{LaiachiPaolo2} for the precise definition of these algebroids).
It is noteworthy that this functor passes through three constructions: The first one uses the smooth global sections functor from Lie algebroids to  Lie-Rinehart algebras, the second  resorts to the well-known universal enveloping algebroid functor that assigns to any Lie-Rinehart algebra (see \S \ref{ssec:LR} for the definition) its universal (right) cocommutative Hopf algebroid, and the third construction utilizes the notion of \emph{convolution Hopf algebroid} \cite{LaiachiPaolo2}.  In this way, a notable observation due to Kapranov \cite{Kapranov:2007} says  that the module of smooth global sections of a given Lie algebroid (as above), can be recovered as the subspace of continuous derivations (killing the source map)  of the attached convolution algebroid. In other words, formal affine groupoid schemes give rise to an algebraic approach to Lie algebroids integration problem.

Finally, as implicitly suggested above, Lie-Rinehart algebras  present themselves as the algebraic counterpart of Lie algebroids and so they become a natural substitute for Lie algebras (in subsection \ref{ssec:Umemura}, we give new examples of these objects).
Moreover, by the foregoing, it is reasonable to expect that Lie-Rinehart algebras and affine groupoid schemes are closely related, although no adjunction connecting them and extending the one stated in \eqref{Eq:hola} is known in the literature. It is then natural to look for an adjunction between the category of  Lie-Rinehart algebras (or Lie algebroids) and that of commutative Hopf algebroids (or affine groupoid schemes), which could set up the bases of the formal differentiation and integration processes in this context. The main  achievement of this paper is to solve this question in the affirmative.  As we will see, similar difficulties as those mentioned above show up in this setting.

\subsection{Description of main results} We now give a detailed description of our main results.
Let $A$ be a commutative algebra  over a ground field $\Bbbk$ (usually of zero characteristic). Set $\Halgd{A}$ to be the category of commutative Hopf algebroids with base algebra $A$ and consider its full subcategory $\GHalgd{A}$ whose objects are Galois\footnote{The terminology ``Galois'' is  motivated by the fact that it extends Galois theory of commutative  Hopf algebras, which in turn extends the classical Galois theory.} (see \S \ref{ssec:CHAlgds} and \S \ref{ssec:Galois}).  The  category of (right) co-commutative Hopf algebroids with base algebra $A$ is denoted by  $\CHalgd{A}$ (see \S \ref{ssec:cocom}).

The first task in order to establish the notion of differentiation and integration in this context and  in the previous sense  is to construct a contravariant functor from  $\CHalgd{A}$ to  $\Halgd{A}$. There are two interrelated ways to construct such a  functor. The first one uses what is known in the literature  as \emph{Tannaka reconstruction process}, applied to a certain symmetric monoidal category of modules (this was mainly achieved in \cite{LaiachiGomez} and recalled in \S \ref{ssec:circ} for the reader's sake). The second way uses the Special  Adjoint Functor Theorem (SAFT) applied to the category of $A$-rings. The structure maps of the constructed commutative Hopf algebroid (via SAFT) out of a co-commutative one, as well as its universal property, are explicitly given in \S \ref{ssec:rsaft}. The construction of these contravariant functors is of independent interest and it constitutes our first main result, stated below as a combination of Proposition \ref{prop:Fdual} and Theorem \ref{thm:Teoremone}:

{\renewcommand{\thetheorem}{{\bf A}}
\begin{theorem}\label{thm:TeoremoneA}
 Let $A$ be a commutative algebra. Then there are two contravariant  functors
$$
(-)^{\circ}: \CHalgd{A} \longrightarrow \Halgd{A}, \quad \rsaft{(-)}:\CHalgd{A} \longrightarrow \Halgd{A}.
$$
Explicitly, take a (right) cocommutative Hopf algebroid $(A,\cU)$ and consider its convolution algebra $(A,\rdual{\cU})$. There are two commutative Hopf algebroids $(A,\cU^{\circ})$ and $(A,\rsaft{\cU})$, which fit into a commutative diagram of $(A\tensor{}A)$-algebras:
\begin{equation}\label{Eq:ZetaXi}
\begin{gathered}
\xymatrix@R=15pt{ \cU^{\circ}  \ar@{->}^-{{\zeta}}[rr] \ar@{->}_-{\hat{\zeta}}[rd] & & \rdual{\cU} \\ & \rsaft{\cU},  \ar@{->}_-{\xi}[ru] & }
\end{gathered}
\end{equation}
where  $\hat{\zeta}$ is a morphism of commutative Hopf algebroids.  Furthermore, the map $\hat{\zeta}$ is an isomorphism either when $\cU$ is a Hopf algebra (i.e., when $A=\Bbbk$), or when it has a finitely generated and projective underlying (right) $A$-module.
 \end{theorem}
}

In contrast to the classical situation, in diagram \eqref{Eq:ZetaXi} neither $\zeta$ nor $\xi$ are necessarily injective. Its seems that this injectivity forms part of the structure of the involved Hopf algebroids. For instance, $\zeta$ is injective for any pair $(A, \cU)$ where $A$ is a Dedekind domain, $\xi$ is injective if and only if its kernel is a coideal,  and $\hat{\zeta}$ is an isomorphism  if and only if $(A, \rsaft{\cU})$ is a Galois Hopf algebroid. These and other properties are explored with full details in \S \ref{ssec:FDGeneral}.

Now denote by  $\LieR{A}$ the category of all Lie-Rinehart algebras over $A$.\footnote{When $A= \Cinfty{\cM}$ is the real algebra of smooth functions on a smooth real manifold, then the category of Lie algebroids over $\cM$ can be realised, via the global smooth sections functor, as a subcategory of $\LieR{A}$. If, furthermore, $\cM$ is compact then, by using Serre-Swan theorem, one can shows that the full subcategory of  Lie-Rinehart algebra over $A$ whose underlying modules are finitely generated and projective of constant rank is equivalent to that of Lie algebroids over $\cM$.}  It is well known from the literature that there is a (covariant) functor $\cV_{\Sscript{A}}(-): \LieR{A} \to \CHalgd{A}$ which assigns to any Lie-Rinehart algebra its universal enveloping Hopf algebroid (details are expounded in \S \ref{ssec:LR}).

Our first main goal is to show, by employing  Theorem \ref{thm:TeoremoneA},  that there are  functors:
\begin{equation}\label{Eq:Functors}
\lL:\Halgd{A}^{\text{op}}\longrightarrow \LieR{A}; \qquad     \iI,\, \iI':  \LieR{A}  \longrightarrow \Halgd{A}^{\text{op}},
\end{equation}
which are termed the \emph{differentiation} and \emph{integration functors}, respectively, and to establish two adjunctions involving these functors. In the notation of \S \ref{sec:LI} below, we have that
$$
\lL(\cH) = \Derk{s}{\cH}{\Aep}=\Big\{  \delta: \cH \to A\;\;  \Bbbk\text{-linear map }|\, \, \delta \circ s= 0, \, \delta(uv)= \varepsilon(u)\delta(v)+\delta(u)\varepsilon(v),\, \forall u,v \, \in \, \cH \Big\},
$$
this is referred to as the Lie-Rinehart algebra of a given commutative Hopf algebroid $(A,\cH)$ and its structure maps are explicitly expounded in Lemma \ref{lema:LR} and Proposition \ref{prop:LR}.

Mimicking \cite[II \S 4]{DemazureGabriel}, we give an alternative construction of the differentiation functor (Proposition \ref{prop:Gamma} in Appendix  \ref{ssec:FA}), which can be seen as an algebraic counterpart of the differentiation of Lie groupoids, and we examine the case of an operation of an affine group scheme on an affine scheme, providing several illustrating examples (see Appendix \ref{sec:A1} for more details).  More examples are also expounded in \S \ref{ssec:Umemura}, where we provide with  full details the computation of the Lie-Rinehart algebra of a certain Malgrange's Hopf algebroid, that arise  from differential Galois theory over the affine complex line. Besides, we show that there is a canonical morphism of Lie-Rinehart algebras between the latter and the one given by the global sections of the Lie algebroid of the associated invertible jet groupoid. In  analogy with Lie groupoid theory, when the affine scheme attached to $A$ admits $\K$-points, then we are able to recognise the isotropy Lie algebras  underlying the Lie algebroid $\lL(\cH)$ as the Lie algebras of the affine isotropy group schemes of the affine groupoid scheme attached to $(A,\cH)$. This is achieved in \S \ref{ssec:isotropy}.

The fact that there are two integrations functors $\iI$ and $\iI'$, which in the classical case of commutative Hopf algebras and Lie algebras coincide, is mainly due to the existence of two different and interrelated approaches in constructing the finite dual contravariant functor on non necessarily commutative rings hereby explored. More precisely, the first  functor $\iI$ is the composition of two functors $\iI= (-)^{\circ} \circ \cV_{\Sscript{A}}(-)$ and the second integration functor $\iI'$ decomposes as $\iI'= \rsaft{(-)} \circ  \cV_{\Sscript{A}}(-)$, where $(-)^{\circ}$ and $\rsaft{(-)}$ are the functors stated in Theorem \ref{thm:TeoremoneA}. According to this Theorem, both integrations functors are shown to fit into a commutative diagram:
$$
\xymatrix@R=18pt{ \iI(-)  \ar@{->}^-{\mb{\zeta}_{\Sscript{\cV_A(-)}}}[rr] \ar@{->}_-{\mb{\hat{\zeta}}_{\Sscript{\cV_A(-)}}}[rd] & & (\cV_{\Sscript{A}}(-))^* \\ & \iI'(-) \ar@{->}_-{\mb{\xi}_{\Sscript{\cV_A(-)}}}[ru] & }
$$
where, for every Lie-Rinehart algebra $(A,L)$,  the algebra $(\cV_{\Sscript{A}}(L))^*$ is the convolution algebra of $\cV_{\Sscript{A}}(L)$ endowed with its topological commutative Hopf algebroid structure (see \cite{LaiachiPaolo2} for the precise notion). In the above diagram,  the natural transformation $\mb{\zeta}$ is the one defined in  Eq. \eqref{Eq:zeta}, $\mb{\hat{\zeta}}$ is the lifting of $\mb{\zeta}$ by the universal property \eqref{Eq:zetahat}, and $\mb{\xi}$ is the natural transformation described in Lemma \ref{lem:maggico}, where we also characterize the injectivity of this map.

The second main result of the paper is the following theorem, which is presented here as a combination of Theorems \ref{thm:Ap} and \ref{thm:A} stated below.

{\renewcommand{\thetheorem}{{\bf B}}
\begin{theorem}\label{thm:AA}
Let $A$ be a commutative algebra. Then there is a natural isomorphism
$$
\xymatrix@R=15pt{  \Hom{\Halgd{A}}{\cH}{\iI'(L)}  \ar@{->}^{\cong}[rr] & &  \Hom{\LieR{A}}{L}{\lL(\cH)}, }
$$
for any commutative Hopf algebroid $(A,\cH)$ and Lie-Rinehart algebra $(A,L)$. That is, the integration functor $\iI'$ is left adjoint to the differentiation functor $\lL$.

Assume now that  the map $\mb{\zeta}_R$  of Eq. \eqref{Eq:zeta} is injective for every $A$-ring $R$ (e.g.,~when $A$ is a Dedekind domain\footnote{This the case when $A$ is the coordinate algebra of an irreducible smooth curve over an algebraically closed field.}). Then there is a natural isomorphism
$$
\xymatrix{  \Hom{\GHalgd{A}}{\cH}{\iI(L)}  \ar@{->}^{\cong}[rr] & &  \Hom{\LieR{A}}{L}{\lL(\cH)}, }
$$
for any commutative Galois Hopf algebroid  $(A,\cH)$ and Lie-Rinehart algebra $(A,L)$. That is, the integration functor $\iI$ is left adjoint to the restriction of the differentiation functor $\lL$ to the full subcategory of Galois Hopf algebroids.
\end{theorem}
}

The unit and the counit of the second adjunction are detailed  in Appendix \ref{ssec:UAdj}.  Given a Lie-Rinehart algebra $(A,L)$, it is of particular interest to consider the following commutative diagram involving both units and  stated in Proposition \ref{prop:Theta} below:
\begin{equation}\label{Eq:LTriangle}
\begin{gathered}
\xymatrix@R=15pt{   L  \ar@{->}^-{\Theta_{L}}[rr]   \ar@{->}_-{\Theta'_{L}}[rrd]  &&   \lL(\VL^{\circ})  \\ &&  \lL(\rsaft{\VL}).  \ar@{->}_-{\lL(\hat{\zeta})}[u] }
\end{gathered}
\end{equation}
As a consequence of Theorem \ref{thm:AA}, the commutative Hopf algebroid $(A,\lL(\rsaft{\VL}))$  (hence its associate presheaf of groupoids) can be  thought of as the universal groupoid of the given Lie-Rinehart algebra $(A,L)$ with a universal  morphism $\Theta'_{L}: L \to  \lL(\rsaft{\VL})$.
In our opinion, the question if either $\Theta_L$ or $\Theta'_L$ is an isomorphism for a specific $(A,L)$ can be regarded as a first step towards the study of the integrability of Lie-Rinehart algebras (i.e., the problem of integrating Lie-Rinehart algebra\footnote{This problem can be rephrased as follows: Given a Lie-Rinehart algebra $(A,L)$ where $L$ is a finitely generated and projective $A$-module, under which conditions  is there a commutative Hopf algebroid $(A,\cH)$ such that $L \cong \lL(\cH)$ as Lie-Rinehart algebras? See Remark \ref{rem:Integrable}, for more discussions.}). Another  question that Theorem \ref{thm:AA}  introduces, is to seek for full subcategories of $\LieR{A}$ and $\Halgd{A}$ for which the previous adjunction restrict to an anti-equivalence of categories.

Let $(A,\cH)$ be a commutative Hopf algebroid, set $\cI={\rm Ker}(\varepsilon)$ for the kernel of its counit and consider its quotient $A$-bimodule $Q(\cH):=\cI/ \cI^2$. Then the K\"ahler module $\Omega_{A}^{s}(\cH)$ of $(A,\cH)$ with respect to the source map is shown to be given, up to a canonical isomorphism,  by:
$$
\Omega_{A}^{s}(\cH) \, \cong \, \Hst \tensor{A} {{{}_{\Sscript{s}}}}Q(\cH), \quad \Big( \psis: \Hs \longrightarrow  \Omega_{A}^{s}(\cH), \; \big[  u \longmapsto u_{1}\tensor{A} \pis(u_{2}) \big] \Big)
$$
where $\psis$ is the morphism that plays the role of the universal derivation and  $\pis: \sH \to {{{}_{\Sscript{s}}}}Q(\cH)$ is the left $A$-modules morphism which sends $u \mapsto
(u - s(\varepsilon(u))) + \cI^{2}$.

The subsequent one is the third aforementioned main result, which deals with the notion of \emph{separable morphism} between commutative Hopf algebroids with same base algebra:
{\renewcommand{\thetheorem}{{\bf C}}
\begin{theorem}\label{thm:AB}
Let $(\id,\, \phi):(A,\cK)\to (A,\cH)$ be a morphism of commutative Hopf algebroids. Assume that $\indec{\cH}$ and $\indec{\cK}$ are finitely generated and projective $A$-modules. The following assertions are equivalent
\begin{enumerate}[label=({\alph*}),ref=\emph{({\alph*})}]
\item $\indec{\phi}$ is split-injective.
\item $\lL(\phi): \lL(\cH) \to \lL(\cK)$ is surjective.
\item $\Derk{s}{\phi}{-}:\Derk{s}{\cH}{-}\to \Derk{s}{\cK}{\phi_{\Sscript{*}}(-)}$ is surjective on each component.
\item $\Derk{s}{\phi}{\cH}:\Derk{s}{\cH}{\cH}\to \Derk{s}{\cK}{\cH}$ is surjective.
\item $\cH\tensor{\cK}\Omega_{A}^s(\cK)\to \Omega_{A}^s(\cH):\,h\tensor{\cK}w\mapsto h\Omega_{A}^s(\phi)(w)$ is split-injective.
\end{enumerate}
\end{theorem}
}
The assumption made in this theorem  are, of course, fulfilled whenever the total algebras $\cH$ and $\cK$ are regular\footnote{For instance, regular functions of an algebraic smooth variety.}.
In analogy with the affine algebraic groups \cite[page 196]{Abe}, a morphism of Hopf algebroids is called \emph{separable} if it satisfies one of the equivalent conditions in Theorem \ref{thm:AB}.

Lastly, we would like to mention that the construction of the finite dual for commutative Hopf algebroids, which are at least flat over the base algebra,   is also possible in principle. Thus, the construction  of a  contravariant functor from a certain full subcategory $\Halgd{A}$ to $\CHalgd{A}$ is feasible in theory.   Pushing further the investigation in this direction, one can be tempted to construct, for instance, a certain analogue of the \emph{hyperalgera} (or \emph{hyperalgebroid}) for an affine algebraic $\Bbbk$-groupoid and subsequently establish results similar to \cite[Theorems 4.3.13, 4.3.14]{Abe} for a flat commutative Hopf algebroid. We will not go on this topic here as, in our opinion, this deserves a separate research project.

\subsection{Notation and basic notions}\label{ssec:Notations}
Given a (Hom-set) category  $\cC$, the notation $C \in \cC$ stands for: $C$ is an object of $\cC$.  Given two objects $C,C' \in\cC$, we sometimes denote by $\hom{\cC}{C}{C'}$ the set of all morphisms from  $C$ to $C'$. We work over a  commutative base field $\Bbbk$ (possibly of characteristic zero). All algebras are $\Bbbk$-algebras  and the unadorned tensor product $\tensor{}$ stands for the tensor product over $\Bbbk$, $\tensor{\Bbbk}$. Given an algebra $A$, we denote by $\Ae=A\tensor{}A^{\Sscript{op}}$ its enveloping algebra. Bimodules over algebras are  understood to have a central underlying $\Bbbk$-vector space structure. As usual the notations $\lmod{A}, \rmod{A}$ and $\bimod{A}$ stand for the categories of left $A$-modules, right $A$-modules and $A$-bimodules, respectively.

Given two algebras $R,S$ and two bimodules ${}_{\Sscript{R}}M_{\Sscript{S}}$ and ${}_{\Sscript{R}}N_{\Sscript{S}}$, for simplicity, we denote by $\hom{R-}{M}{N}$, $\hom{-S}{M}{N}$ and $\hom{R-S}{M}{N}$ the $\Bbbk$-vector spaces of all left $R$-module, right $S$-module and $(R,S)$-bimodule morphisms from $M$ to $N$, respectively. The left and right \emph{duals}  of ${}_{\Sscript{R}}M_{\Sscript{S}}$ are denoted  by ${}^{*}M:=\hom{R-}{M}{R}$ and $M^{*}:=\hom{-S}{M}{S}$, respectively. These are $( S, R)$-bimodules and the actions are given as follows. For every $r \in R$, $s \in S$, $f \in {}^*M$ and $g \in M^*$, we have
\begin{equation}\label{Eq:actions}
sfr: M \longrightarrow R, \quad \Big(m \longmapsto f(ms)r  \Big); \qquad sgr: M \longrightarrow S, \quad \Big(m \longmapsto sg(rm)  \Big).
\end{equation}

For two morphisms   $p, q: A \to B$ of algebras, we shall denote by ${}_{\Sscript{p}}B,  B_{\Sscript{q}}$ and ${}_{\Sscript{p}}B_{\Sscript{q}}$, the left $A$-module, the right $A$-module and the $A$-bimodule structure on $B$, respectively. In case that only one algebra morphism is involved, \ie~ when $p=q$, for simplicity, we use the obvious notation: ${}_{\Sscript{A}}B, B_{\Sscript{A}}$ and ${}_{\Sscript{A}}B_{\Sscript{A}}$.

For an algebra $A$, a left (or right) \emph{$A$-linear map} stands for a morphism of left (right) $A$-modules, while an \emph{$A$-bilinear map} refers to a morphism between $A$-bimodules. For such an algebra $A$, an \emph{$A$-ring} is an algebra extension $A \to R$, or equivalently a monoid in the monoidal category  $(\bimod{A},\tensor{A},A)$. Given an $A$-ring $R$,  we will denote by $\Amod{R}$ the full subcategory of right $R$-modules whose underlying right $A$-modules are finitely generated and projective.

The dual notion of $A$-ring is that of $A$-coring. Thus, an \emph{$A$-coring} is a co-monoid in the monoidal category $\left(\bimod{A},\tensor{A},A\right)$ of $A$-bimodules. That is, an $A$-bimodule $\coring{C}$ with two $A$-bilinear maps $\Delta: \coring{C} \to \coring{C} \tensor{A} \coring{C}$ (\emph{the comultiplication}, sending $x$  to $x_{\Sscript{1}}\tensor{A}x_{\Sscript{2}}$ with summation understood) and $\varepsilon: \coring{C} \to A$ (\emph{the counit}) subject to the co-associativity and co-unitarity constraints. A \emph{right $\coring{C}$-comodule} is a pair $(M , \varrho_{\Sscript{M}})$, where $M$ is a right $A$-module and $\varrho_{\Sscript{M}}: M \to M \tensor{A} \coring{C}$ is a right $A$-linear map which is compatible with $\Delta$ and $\varepsilon$ in a natural way (\ie $\left(\varrho_{\Sscript{M}}\tensor{A}\coring{C}\right)\circ \varrho_{\Sscript{M}}=\left(M\tensor{A}\Delta\right)\circ \varrho_{\Sscript{M}}$ and $\left(M\tensor{A}\varepsilon\right)\circ \varrho_{\Sscript{M}} = \id{\Sscript{M}}$). There is an adjunction between right $A$-modules and right $\coring{C}$-comodules given on the one side by the forgetful functor $\oO: \rcomod{\coring{C}} \to \rmod{A}$ and  on the other one by the functor $-\tensor{A}\coring{C}: \rmod{A} \to \rcomod{\coring{C}}$ (see \eg \cite[\S18.10]{BrzezinskiWisbauer}).  For a given $A$-coring $\coring{C}$ we denote by $\cat{A}^{\Sscript{\coring{C}}}$ the full subcategory of right $\coring{C}$-comodules $(M, \varrho_{\Sscript{M}})$ such that $\oO\big(M,\varrho_{\Sscript{M}}\big)$ is a finitely generated and projective right $A$-module.  For a  given $A$-coring $(\fk{C},\Delta,\varepsilon)$ we have an $A$-ring structure on $\ldual{\fk{C}}$ called \emph{the left convolution algebra of $\fk{C}$}. This structure is given by
\begin{equation}\label{eq:convolution}
(f*g)(x)=g\big(x_{\Sscript{1}}f(x_{\Sscript{2}})\big), \qquad 1_{\ldual{\fk{C}}}=\varepsilon \qquad \text{and} \qquad (afb)(x)=f(xa)b
\end{equation}
 for all $f,g\in\ldual{\fk{C}}$, $a,b\in A$ and $x\in \fk{C}$.  Analogously, one can introduce the \emph{right convolution algebra} $\fk{C}^*$ of $\fk{C}$.

\begin{remark}\label{rem:circdot}
Recall that given two $A$-corings $\coring{C}$ and $\coring{D}$ we can consider the new $A$-coring
\begin{equation*}
\coring{C} \odot \coring{D}:=\frac{\coring{C} \otimes \coring{D}}{\mathsf{Span}_\K\left\{acb\otimes d-c\otimes adb\mid a,b\in A, c\in \coring{C},d\in \coring{D}\right\}}
\end{equation*}
which is a coring with respect to the following structures
\begin{gather*}
a(c\odot d)b = c\odot adb, \qquad \Delta(c\odot d)=\left(c_1\odot d_1\right)\tensor{A}\left(c_2\odot d_2\right), \qquad \varepsilon(c\odot d)=\varepsilon_{\coring{C}}(c)\varepsilon_{\coring{D}}(d),
\end{gather*}
where the notation is the obvious one. We point out that $\coring{C}\odot \coring{D}$ has been obtained by applying \cite[Theorem 3.10]{Takeuchi} to $\coring{C}$ and $\coring{D}$ endowed with the $T|S$ and the $S|R$-coring structures respectively whose underlying multi-module structures are given by
$
(t\otimes s)c(t'\otimes s')=tsct's' $ and $(s\otimes r)d(s'\otimes r')=s'rdsr',
$
where $R=S=T=A$ and $r,r'\in R$, $s,s' \in S$, $t,t'\in T$, $c\in \coring{C}$ and $d\in \coring{D}$.
\end{remark}

\begin{remark}\label{rem:cop}
Notice that given an $A$-coring $\coring{C}$, we may consider the $A$-coring $\cop{\coring{C}}$ with structures given by
\begin{equation}\label{eq:Ccop}
\Delta(\cop{c})=\cop{(c_2)}\tensor{A}\cop{(c_1)},\qquad \varepsilon(\cop{c})=\varepsilon(c) \qquad \text{and} \qquad b\cop{c}a=\cop{\left(ac b\right)},
\end{equation}
where $\cop{c}$ denotes $c\in \coring{C}$ as seen in $\cop{\coring{C}}$.
\end{remark}

Let $A$ be a commutative algebra, we denote by $\proj{A}$ the full subcategory of the category of (one sided, preferably right) $A$-modules whose objects are finitely generated and projective.  For a given morphism of commutative algebras $\phi:A \to B$ we denote by $\phi_{\Sscript{*}}: \rmod{B} \to \rmod{A}$ the restriction functor between the categories of right modules.


\section{Hopf algebroids and Lie-Rinehart algebras: Definitions and examples}\label{sec:HAlgds}

A Hopf algebroid can be naively thought as a Hopf algebra over a non-commutative ring. In the present paper we are going to focus on the distinguished classes of commutative and cocommutative Hopf algebroids (\ie those that have a closer connection with algebraic and differential geometry), instead of dealing with them in the full generality. Therefore, and for the sake of the unaccustomed reader, we will recall in the present section the definitions of these objects together with some significant examples, that is to say, the universal enveloping Hopf algebroids of Lie-Rinehart algebras.

\subsection{Commutative Hopf algebroids}\label{ssec:CHAlgds}
We recall here from \cite[Appendix A1]{ravenel} the definition of commutative Hopf algebroid. We also expound some examples which will be needed in the forthcoming sections.

A \emph{commutative Hopf algebroid over $\Bbbk$} is a \emph{cogroupoid} object in the category $\mathrm{CAlg}_{\K }$  of commutative $\Bbbk$-algebras, or equivalently, a groupoid in the category of affine schemes. Thus, a Hopf algebroid consists of a  pair of commutative algebras $\left( A,\mathcal{H}\right) $, where $A$ is the \emph{base algebra} and $\cH$ is the \emph{total algebra}   with a diagram of algebra maps:
\begin{equation}\label{Eq:loop}
\xymatrix@C=45pt{ A \ar@<1ex>@{->}|(.4){\scriptstyle{s}}[r] \ar@<-1ex>@{->}|(.4){\scriptstyle{t}}[r] & \ar@{->}|(.4){ \scriptstyle{\varepsilon}}[l] \ar@(ul,ur)^{\scriptstyle{\cS}} \cH \ar@{->}^-{\scriptstyle{\Delta}}[r] & \cH \tensor{A}\cH , }
\end{equation}
where to perform the tensor product over $A$, the algebra $\cH$ is considered as an $A$-bimodule of the form ${}_{\scriptstyle{s}}\cH_{\scriptstyle{t}}$, i.e., $A$ acts on the left through $s$ while it acts on the right through $t$.
The maps $s,t:A\rightarrow \mathcal{H}$ are called the \emph{source} and \emph{target} respectively, and $\etaup:=s\tensor{}t: A\tensor{}A \to \cH$ , $a\tensor{}a'\mapsto s(a)t(a')$ is the \emph{unit},  $\varepsilon :\mathcal{H}\rightarrow A$ the \emph{counit},  $\Delta :\mathcal{H}\rightarrow \mathcal{H}\tensor{A} \mathcal{H}$ the \emph{comultiplication} and $\cS:\mathcal{H}\rightarrow \mathcal{H}$ the \emph{antipode}. These  have to satisfy the following compatibility conditions.
\begin{itemize}
\item The datum $( \sHt, \Delta, \varepsilon)$ has to be a coassociative and counital coalgebra in the category of $A$-bimodules, i.e., an $A$-coring. At the level of groupoids, this encodes a unitary and associative composition law between morphisms.
\item The antipode has to satisfy $\cS \circ s=t$, $\cS \circ t = s$ and $\cS^2=\id{\cH}$, which encode the fact that the inverse of a morphism interchanges source and target and that the inverse of the inverse is the original morphism.
\item The antipode has to satisfy also $\cS(h_1)h_2=(t\circ \varepsilon)(h)$ and $h_1\cS(h_2)=(s\circ \varepsilon)(h)$, which encode the fact that the composition of a morphism with its inverse on either side gives an identity morphism (the notation $h_1\otimes h_2$ is a variation of the Sweedler's Sigma notation, with the summation symbol understood, and it stands for $\Delta(h)$).
\end{itemize}

\begin{remark}\label{rem:Santicocomm} Let us make the following observations on the previous definition:
\begin{enumerate}[label=(\arabic*), ref=(\arabic*)]
\item Note that there is no need to require that $\varepsilon \circ s= \id{A}= \varepsilon \circ t$, as  it is implied by the first condition.
\item\label{item:Santicocomm} Since the inverse of a composition of morphisms is the reverse composition of the inverses, the antipode $\cS$ of a commutative Hopf algebroid is an anti-cocommutative map. This means that,  $\tau\Delta(\cS(u))=(\cS\tensor{A}\cS)(\Delta(u))$ in $\Hs\tensor{A}\, \tH$, explicitly, $\cS(u_1)\tensor{A}\cS(u_2) = \cS(u)_2\tensor{A}\cS(u)_1$ for all $u\in\cH$. Thus, $\cS: \sHt \to \tHs$ is an isomorphism of $A$-corings.
\end{enumerate}
\end{remark}

A \emph{morphism of commutative Hopf algebroids} is a pair of algebra maps
$\left( \phi _{\Sscript{0}},\phi _{\Sscript{1}}\right) :\left( A,\mathcal{H}\right)
\rightarrow \left( B,\mathcal{K}\right)$
such that
\begin{eqnarray*}
\phi _{1}\circ s =s\circ \phi _{0},&\qquad& \phi _{1}\circ t=t\circ \phi
_{0}, \\
\Delta \circ \phi _{1} =\chi \circ \left( \phi _{1}\otimes _{A}\phi
_{1}\right) \circ \Delta ,&\qquad& \varepsilon \circ \phi _{1}=\phi _{0}\circ
\varepsilon , \\
\cS\circ \phi _{1} =\phi _{1}\circ \cS &&
\end{eqnarray*}
where $\chi :\mathcal{K}\tensor{A}\mathcal{K\rightarrow }\mathcal{K} \tensor{B}\mathcal{K}$ is the obvious map induced by $\phi_0$, that is $\chi \left(h\tensor{A}k\right) =h\tensor{B}k$. The category  obtained in this way is denoted by $\HAlgd$, and if the base algebra $A$ is fixed, then  the resulting category will be denoted by $\Halgd{A}$.

\begin{example}\label{exm:Halgd}
Here there are some common examples of Hopf algebroids (see  also \cite{ElKaoutit:2015}):
\begin{enumerate}
\item Let $A$ be an algebra. Then the pair $(A, A\otimes_{}A)$ admits a Hopf algebroid structure given by $s(a)=a\otimes_{}1$, $t(a)=1\otimes_{}a$, $\cS(a\otimes_{}a')=a'\otimes_{}a$, $\varepsilon(a\tensor{}a')=aa'$ and $\Delta(a\otimes_{}a')= (a\otimes_{}1) \tensor{A} (1\otimes_{}a')$, for any $a, a' \in A$.
\item Let $(B,\Delta, \varepsilon, \sS)$ be a Hopf algebra and $A$ a right $B$-comodule algebra with coaction $A \to A\otimes_{}B$, $a \mapsto a_{\Sscript{(0)}} \otimes_{} a_{\Sscript{(1)}}$. This means that $A$ is right $B$-comodule and the coaction is an algebra map (see e.g.~\cite[\S 4]{Montgomery:1993}).   Consider  the algebra  $\cH= A\otimes_{}B$ with  algebra extension $ \etaup: A\otimes_{}A \to  \cH$, $a'\otimes_{}a \mapsto a'a_{\Sscript{(0)}}\otimes_{}a_{\Sscript{(1)}}$. Then $(A,\cH)$ has  a structure of Hopf algebroid, known as  \emph{split Hopf algebroid}:
$$
\Delta(a\otimes_{}b) = (a\otimes_{}b_{\Sscript{1}}) \otimes_{A} (1_{\Sscript{A}}\otimes_{}b_{\Sscript{2}}), \;\;\varepsilon(a\otimes_{}b)=a\varepsilon(b),\;\; \cS(a\otimes_{}b)= a_{\Sscript{(0)}}\otimes_{}  a_{\Sscript{(1)}}\sS(b).
$$
\item Let $B$ be as in part $(2)$ and $A$ any algebra. Then $(A, A\otimes_{}B\otimes_{}A)$ admits in a canonical way a structure of Hopf algebroid.  For $a,a'\in A$ and $b\in B$, its structure maps are given as follows
\begin{gather*}
s(a)=a\otimes 1_{\Sscript{B}}\otimes 1_{\Sscript{A}}, \quad t(a)=1_{\Sscript{A}}\otimes 1_{\Sscript{B}}\otimes a, \quad \varepsilon(a\otimes b\otimes a')=aa'\varepsilon(b), \\
\Delta(a\otimes b\otimes a')= \big(a\otimes b_1\otimes 1_{\Sscript{A}}\big) \tensor{A} \big(1_{\Sscript{A}}\otimes b_2\otimes a'\big), \quad \cS(a\otimes b\otimes a')=a'\otimes \sS(b)\otimes a.
\end{gather*}
 \end{enumerate}
Notice that (1) may be recovered from (3) by considering $B=\Bbbk$ as Hopf $\Bbbk$-algebra with trivial structure.
\end{example}

\subsection{Co-commutative Hopf algebroids}\label{ssec:cocom}
Next, we recall the definition of a cocommutative Hopf algebroid. It can be considered as a revised (right-handed and cocommutative) version of the notion of a $\times_A$-Hopf algebra as it appears in \cite[Theorem and Definition 3.5]{Schau:DADOQGHA}. However, to define the underlying right bialgebroid structure we preferred to mimic \cite{Lu} as presented in \cite[Definition 2.2]{BrzezinskiMilitaru} (in light of \cite[Theorem 3.1]{BrzezinskiMilitaru}, this is something we may do). See also \cite[A.3.6]{Kapranov:2007} and compare with \cite[Definition 2.5.1]{Kowalzig} and \cite[\S4.1]{Szlachanyi} as well.

A \emph{(right) co-commutative Hopf algebroid} over a commutative algebra is the datum of a commutative algebra $A$,  a possibly  noncommutative algebra $\cU$ and an algebra map $s=t: A\to \cU$  landing not necessarily in the center of $\cU$, with the following additional structure maps:
\begin{itemize}
\item A  morphism of right $A$-modules $\varepsilon: \cU \to A$ which satisfies
\begin{equation}\label{eq:doubleepsilon}
\varepsilon(uv)=\varepsilon(\varepsilon(u)v),
\end{equation}
for all $u, v \in \cU$;
\item An $A$-ring map $\Delta: \cU \to \cU \times_{\Sscript{A}} \cU$, where the module
\begin{equation}\label{takeuchi}
\cU\times_{\Sscript{A}} \cU:=\left\{ \sum_{\Sscript{i}} u_{\Sscript{i}}\tensor{A} v_{\Sscript{i}} \in \cU_{\Sscript{A}}~ \tensor{A} ~\cU_{\Sscript{A}} \mid~ \sum_{\Sscript{i}} au_{\Sscript{i}}\tensor{A} v_{\Sscript{i}}= \sum_{\Sscript{i}} u_{\Sscript{i}}\tensor{A} av_{\Sscript{i}}  \right\}
\end{equation}
is endowed with the algebra structure
\begin{equation*}
\sum_{\Sscript{i}} u_{\Sscript{i}}\times_{\Sscript{A}} v_{\Sscript{i}} ~ .~ \sum_{\Sscript{j}} u'_{\Sscript{j}}\times_{\Sscript{A}}v'_{\Sscript{j}} = \sum_{\Sscript{i,j}} u_{\Sscript{i}}u'_{\Sscript{j}} \times_{\Sscript{A}} v_{\Sscript{i}}v'_{\Sscript{j}}, \qquad 1_{\cU\times_{\Sscript{A}} \cU}=1_{\cU}\tensor{A}1_{\cU}
\end{equation*}
and the $A$-ring structure given by the algebra map $1: A \rightarrow \cU \times_{\Sscript{A}}\cU,~ \Big(  a \mapsto a \times_{\Sscript{A}} 1_{\Sscript{\cU}}=   1_{\Sscript{\cU}}  \times_{\Sscript{A}} a \Big)$;
\end{itemize}
subject to the conditions
\begin{itemize}
\item $\Delta$ is coassociative, co-commutative in a suitable sense and has $\varepsilon$ as a right and left counit;
\item the canonical map
\begin{equation*}
\lfun{\beta}{\cU_{\Sscript{A}} ~\tensor{A}~ {}_{\Sscript{A}}\cU}{\cU_{\Sscript{A}}~\tensor{A}~ \cU_{\Sscript{A}}}{u\tensor{A}v}{uv_{\Sscript{1}}\tensor{A}v_{\Sscript{2}}}
\end{equation*}
is bijective, where we denoted $\Delta(v)=v_{\Sscript{1}}\otimes_{A}v_{\Sscript{2}}$ (summations understood). As a matter of terminology, the map $\beta^{-1}(1\tensor{A}-):\cU\to \cU_{\Sscript{A}}\tensor{A}{_{\Sscript{A}}\cU}$ is the so-called \emph{translation map}.
\end{itemize}

The first three conditions say that the category of all right $\cU$-modules is in fact a symmetric monoidal category with tensor product given by $-\tensor{A}-$ (see the details below), and the forget full functor to the category of $A$-bimodules is strict monoidal. The last condition says that this forgetful functor also preserves right inner homs-functors.  The pair  $(A,\cU)$ is then referred to as  a \emph{right co-commutative Hopf algebroid over $\Bbbk$}.  From now on the terminology co-commutative Hopf algebroid stands for right ones.

The aforementioned monoidal structure is detailed as follows:  Given  a  co-commutative Hopf algebroid $(A,\cU)$, the identity object is the base algebra $A$, with right $\cU$-action given by $a \centerdot  u =\varepsilon(au)$.
The tensor product of two right $\cU$-modules $M$ and $N$ is the $A$-module $M_A\tensor{A}N_A$ endowed with the following right $\cU$-action:
\begin{equation}\label{Eq:mono}
(m\tensor{A}n) \centerdot u = (m \centerdot u_{\Sscript{1}}) \tensor{A} (n \centerdot u_{\Sscript{2}}).
\end{equation}
The symmetry is provided by the one in $A$-modules, that is to say, the flip $M\tensor{A}N \to N\tensor{A}M$ is a natural isomorphism of right $\cU$-modules.
The dual object of a right $\cU$-module $M$ whose underlying $A$-module is finitely generated and projective, is the $A$-module $M^{*}=\hom{-A}{M}{A}$ with the right $\cU$-action
\begin{equation}\label{Eq:centerdot}
\varphi \centerdot u : M \longrightarrow A, \quad \Big( m \longmapsto \varphi(m\centerdot u_{\Sscript{-}}) \centerdot u_{\Sscript{+}}  \Big),
\end{equation}
where $u_{\Sscript{-}}\tensor{A}u_{\Sscript{+}} = \beta^{-1}(1\tensor{A}u)$ (summation understood). It is easily checked that, for every $a\in A$ and $u,v\in\cU$, one has
\begin{gather}
(au)_{\Sscript{-}}\tensor{A}(au)_{\Sscript{+}}=u_{\Sscript{-}}\tensor{A}au_{\Sscript{+}}\label{form:beta1},\\
au_{\Sscript{-}}\tensor{A}u_{\Sscript{+}} = u_{\Sscript{-}}\tensor{A}u_{\Sscript{+}}a\label{form:beta2},\\
v_{\Sscript{-}}u_{\Sscript{-}}\tensor{A}u_{\Sscript{+}}v_{\Sscript{+}}= (uv)_{\Sscript{-}}\tensor{A}(uv)_{\Sscript{+}}\label{form:beta3},\\
(1_\cU)_{\Sscript{-}}\tensor{A}(1_\cU)_{\Sscript{+}}=1_\cU \tensor{A} 1_\cU \label{form:beta4},\\
(u_{\Sscript{-}})_1\tensor{A}(u_{\Sscript{-}})_2\tensor{A}u_{\Sscript{+}}=(u_{\Sscript{+}})_{\Sscript{-}}\tensor{A}u_{\Sscript{-}}\tensor{A} (u_{\Sscript{+}})_{\Sscript{+}}\label{form:beta5},\\
u_{\Sscript{-}}\tensor{A}(u_{\Sscript{+}})_1\tensor{A}(u_{\Sscript{+}})_2=(u_1)_{\Sscript{-}}\tensor{A}(u_1)_{\Sscript{+}}\tensor{A} u_2\label{form:beta6},\\
u_{\Sscript{-}}u_{\Sscript{+}}=\varepsilon(u)1_{\cU}\label{form:beta7},\\
(u_{\Sscript{-}})_{\Sscript{-}}\tensor{A}(u_{\Sscript{-}})_{\Sscript{+}}u_{\Sscript{+}}=u\tensor{A} 1_{\cU}\label{form:beta8},\\
u_1(u_2)_{\Sscript{-}}\tensor{A}(u_2)_{\Sscript{+}}=1_{\cU} \tensor{A}u \label{form:beta9}.
\end{gather}

Morphisms between co-commutative Hopf algebroids over the same algebra $A$ are canonically defined, and the resulting category is denoted by $\CHalgd{A}$.

\subsection{Lie-Rinehart algebras and the universal enveloping algebroid}\label{ssec:LR}
Let $A$ be a commutative algebra over a field $\K$ of characteristic $0$ and denote by $\mathrm{Der}_{\Sscript{\K}}(A)$ the Lie algebra of all linear derivations of $A$.  Consider a Lie algebra $L$ which is also an $A$-module and let $\omega: L \to \mathrm{Der}_{\Sscript{\K}}(A)$ be an $A$-linear morphism of Lie algebras.
In honour of Rinehart \cite{Rin:DFOGCA}, the pair $(A,L)$ is called a \emph{Lie-Rinehart algebra} with \emph{anchor} map $\omega$ provided that
\begin{equation}\label{Eq:Aseti}
{[ X, aY ]} = a{[X,Y]}+X(a)Y,
\end{equation}
for all $X, Y \in L$ and $a, b \in A$, where $X(a)$ stands for $\omega(X)(a)$.

Apart from the natural examples $(A,\mathrm{Der}_{\Sscript{\K}}(A))$ (with anchor the identity map), another  basic source of examples are the smooth global sections of a given Lie algebroid over a smooth manifold.

\begin{example}\label{exam:VayaCon}
A \emph{Lie algebroid}  is a  vector bundle $\mathcal{L} \to \mathcal{M}$ over a smooth manifold, together with a map $\omega: \mathcal{L} \to T\mathcal{M}$ of vector bundles and a Lie structure $[-,-]$ on the  vector space  $\Gamma(\cL)$  of global smooth sections of $\mathcal{L}$, such that the induced map $\Gamma(\omega): \Gamma(\cL) \to \Gamma(T\mathcal{M})$ is a Lie algebra homomorphism, and for all $X, Y \in \Gamma(\cL)$ and any  smooth function  $f \in \mathcal{C}^{\infty}(\mathcal{M})$  one has
\begin{equation}\label{eq:LieAlgd}
[X,fY]\,=\, f[X,Y]+ \Gamma(\omega)(X)(f)Y.
\end{equation}
Then the pair $(\mathcal{C}^{\infty}(\mathcal{M}), \Gamma(\cL))$ is obviously a Lie-Rinehart algebra. In the Appendix \ref{ssec:LA-LG}, we give a detailed description, using elementary algebraic arguments,  of the Lie-Rinehart algebra attached to the Lie algebroid of a given Lie groupoid.
\end{example}

\begin{remark}
The fact that the map $\Gamma(\omega): \Gamma(\cL) \to \Gamma(T\mathcal{M})$ in Example \ref{exam:VayaCon} is a Lie algebra homomorphism is a consequence of the Jacobi identity and of Relation \eqref{eq:LieAlgd} (see \eg \cite{Grabowski, Herz, Magri}). Therefore, it should be omitted from the definition of a Lie algebroid. Nevertheless, we decided to keep the somewhat redundant definition above to make it easier for the unaccustomed reader to see the parallel with Lie-Rinehart algebras.
\end{remark}

As in the classical case of (co-commutative) Hopf algebras, primitive elements of a (co-commutative) Hopf algebroid\footnote{In fact, the claim is true in general for bialgebroids over a commutative base algebra, but we are interested mainly in the particular case of co-commutative Hopf algebroids.} form a Lie-Rinehart algebra, see \cite{Kowalzig, MoerdijkLie}\footnote{In fact, in \cite{MoerdijkLie} the terminology used is $R/k$-bialgebra (as in \cite{Nichols}). Nevertheless, as we will see, the universal enveloping algebra of a Lie-Rinehart algebra inherits actually a Hopf algebroid structure in the sense of \cite{Schau:DADOQGHA}.}.

\begin{example}[{Primitive elements as Lie-Rinehart algebra}]\label{exam:PU}
Let $(A,\cU)$ be a co-commutative Hopf algebroid. An element $X \in \cU$ is said to be \emph{primitive}, if  it satisfies
$$
\Delta(X)\,=\, 1\tensor{A}X+ X \tensor{A}1,\quad \text{and }\quad \varepsilon(X)=0.
$$
{Notice that the second equality is a consequence of the first one and the counitality property.}
The vector space of all primitive elements $\prim(\cU)$ inherits simultaneously a structure  of $A$-module and Lie algebra, where the $A$-action descends from the right $A$-module structure of $\cU$. In fact, the pair $(A,\prim(\cU))$ is a Lie-Rinehart algebra with anchor map:
$$
\omega: \prim(\cU) \longrightarrow \Ders{\K}{A}, \quad \Big( X \longmapsto \big[a \mapsto -\varepsilon(t(a) X) \big] \Big).
$$
Indeed,  $\omega$ is a Lie algebra and $A$-linear map, since we have
\begin{align*}
\omega(XY-YX)(a) & =-\varepsilon(t(a)(XY-YX))=\varepsilon(t(a)YX)-\varepsilon(t(a)XY)=\varepsilon(\varepsilon(t(a)Y)X)-\varepsilon(\varepsilon(t(a)X)Y) \\
 & = \omega(X)(\omega(Y)(a))-\omega(Y)(\omega(X)(a))=\left[\omega(X),\omega(Y)\right](a),\\
\omega\big(Xt(b)\big)(a) & =-\varepsilon\big(t(a)Xt(b)\big)=-\varepsilon(t(a)X)b=\omega(X)(a)b.
\end{align*}
Equation \eqref{Eq:Aseti} is derived from the following computation
\begin{align*}
\left[X,Yt(a)\right]-\left[X,Y\right]t(a) & =XYt(a)-Yt(a)X-XYt(a)+YXt(a)=Y(Xt(a)-t(a)X)\\
&=Yt(-\varepsilon(t(a)X))=Yt(\omega(X)(a)),
\end{align*}
where the  third equality follows from the fact that $\Delta(t(a)X)=Xt(a)\tensor{A}1+ 1\tensor{A}t(a)X$ which in turns leads to the equality
$$
\varepsilon(t(a)X)1_{\cU}= t(a)X-Xt(a), \quad \text{for every }\; a \in A \, \text{ and }\, X \in \prim(\cU).
$$
\end{example}

A \emph{morphism of Lie-Rinehart algebras} $f: (A, L) \to (A,K)$ is an $A$-linear and Lie algebra map $f: L \to K$ which is compatible with the anchors. That is,  if the following diagram is commutative
$$
\xymatrix{  L  \ar@{->}^-{f}[rr]  \ar@{->}_-{\omega}[rd]  & & K \ar@{->}^-{\omega'}[ld]  \\ &  \Ders{\K}{A}  &}
$$
The category so  constructed well be denoted by $\LieR{A}$

Next we give our main example of co-commutative Hopf algebroids. The \emph{(right) universal enveloping Hopf algebroid} of a given Lie-Rinehart algebra $(A,L)$ is an algebra $ \cV_{\Sscript{A}}\left( L\right) $ endowed with a morphism $\iota _{\Sscript{A}}:A\rightarrow  \cV_{\Sscript{A}}\left( L\right) $ of algebras and a  Lie algebra morphism $\iota _{\Sscript{L}}:L\rightarrow  \cV_{\Sscript{A}}\left( L\right) $ such that
\begin{equation}\label{eq:compLRalg}
\iota_L\left(aX\right) = \iota_L(X)\iota_A(a) \quad \text{and} \quad \iota _{\Sscript{L}}\left(X\right) \iota _{\Sscript{A}}\left( a\right) -\iota _{\Sscript{A}}\left( a\right) \iota _{\Sscript{L}}\left( X\right) =\iota _{\Sscript{A}}\left( X
\left( a\right) \right)
\end{equation}
for all $a\in A$ and $X\in L$, which is universal with respect to this property. In details, this means that if $\left( W,\phi _{\Sscript{A}},\phi _{\Sscript{L}}\right) $ is another algebra with a morphism $\phi _{\Sscript{A}}:A\rightarrow W$ of algebras and a morphism $\phi _{\Sscript{L}}:L\rightarrow W$ of Lie algebras  such that
\begin{equation*}
\phi_L\left(aX\right) = \phi_L(X)\phi_A(a) \quad \text{and} \quad \phi _{\Sscript{L}}\left( X\right)\phi _{\Sscript{A}}\left( a\right) -\phi _{\Sscript{A}}\left( a\right) \phi _{\Sscript{L}}\left( X\right) =\phi _{\Sscript{A}}\left( X \left(a\right) \right) ,
\end{equation*}
then there exists a unique algebra morphism $\Phi : \cV_{\Sscript{A}}
\left( L\right) \rightarrow W$ such that $\Phi \, \iota _{\Sscript{A}}=\phi _{\Sscript{A}}$ and $
\Phi \, \iota _{\Sscript{L}}=\phi _{\Sscript{L}}$.

Apart from the well-known constructions of \cite{Rin:DFOGCA} and \cite{MoerdijkLie}, the universal enveloping Hopf algebroid of a Lie-Rinehart algebra $(A,L)$ admits several other equivalent realizations. For instance, one can use  the smash product (right) $A$-bialgebroid $A\#U_{\K}(L)$, as introduced  by Sweedler in \cite{sweedler}, and quotient this algebra by a proper ideal, in order to perform the universal enveloping of $(A,L)$.
In this paper we opted for the following construction which comes from \cite{LaiachiPaolo2}. Set $\eta :L\rightarrow A\otimes L;\, X\longmapsto 1_{\Sscript{A}}\otimes X$ and consider the tensor $A$-ring $T_{\Sscript{A}}\left( A\otimes L\right) $ of the $A$-bimodule $A\otimes L$. It can be shown that
\begin{equation*}
 \cV_{\Sscript{A}}\left( L\right) \cong \frac{T_{\Sscript{A}}\left( A\otimes L\right) }{\cJ}
\end{equation*}
where the two sided ideal $\cJ$ is generated by the set
\begin{equation*}
\cJ:=\left\langle \left.
\begin{array}{c}
\eta \left( X\right) \otimes _{\Sscript{A}}\eta \left( Y\right) -\eta \left( Y\right)
\otimes _{\Sscript{A}}\eta \left( X\right) -\eta \left( \left[ X,Y\right] \right) , \\
 \eta \left( X\right) \cdot a -a\cdot
\eta \left( X\right)-\omega \left( X\right)
\left( a\right)
\end{array}
\right| \; X,Y\in L, \; a\in A\right\rangle.
\end{equation*}
We have the algebra morphism $\iota _{\Sscript{A}}:A\rightarrow
\cV_{\Sscript{A}}\left( L\right);\, a\longmapsto a+\cJ$ and the  Lie algebra map $\iota _{\Sscript{L}}:L\rightarrow \cV_{\Sscript{A}}\left( L\right);\, X\longmapsto \eta \left( X\right) +\cJ$ that
satisfy the compatibility condition \eqref{eq:compLRalg}. It turns out that $\cV_{\Sscript{A}}(L)$ is a co-commutative right Hopf algebroid over $A$ with structure maps induced by the assignments
\begin{gather*}
\varepsilon \left( \iota _{\Sscript{A}}\left( a\right) \right)   =a, \qquad \varepsilon \left( \iota _{\Sscript{L}}\left( X\right) \right) =0, \\
\Delta \left( \iota _{\Sscript{A}}\left( a\right) \right)  =\iota _{\Sscript{A}}\left( a\right) \times _{\Sscript{A}}1_{\Sscript{ \cV_{\Sscript{A}}\left( L\right)}
}=1_{ \Sscript{\cV_A( L)} }\times _{\Sscript{A}}\iota _{\Sscript{A}}\left( a\right) , \\
\Delta \left( \iota _{\Sscript{L}}\left( X\right) \right)  =\iota _{\Sscript{L}}\left( X\right) \times _{\Sscript{A}}1_{ \Sscript{\cV_A\left( L\right)}
}+1_{\Sscript{\cV_A\left( L\right)} }\times _{\Sscript{A}}\iota _{\Sscript{L}}\left( X\right) , \\
\beta^{-1}\left( 1_{\Sscript{ \cV_A\left( L\right) }} \tensor{A } \iota_{\Sscript{A} }\left( a \right)\right)   = \iota _{\Sscript{A}}\left( a\right) \tensor{A}1_{\Sscript{ \cV_A\left( L\right)} }=1_{\Sscript{ \cV_A\left( L\right) }}\tensor{A}\iota _{\Sscript{A}}\left( a\right) ,  \\
\beta^{-1}\left( 1_{ \Sscript{\cV_A\left( L\right) }} \tensor{A } \iota_{\Sscript{L }}\left( X \right)\right)   = 1_{ \Sscript{\cV_A\left( L\right) }}\tensor{A}\iota _{\Sscript{L}}\left( X\right) -\iota
_{\Sscript{L}}\left( X\right) \tensor{A}1_{\Sscript{ \cV_A\left( L\right)} }.
\end{gather*}

\begin{remark}\label{rem:primfunctor}
The primitive functor $\prim:\CHalgd{A}\to \LieR{A}$, assigning to a  co-commutative Hopf algebroid $(A,\cU)$ the space $\prim(\cU)$ and to a morphism $f:(A,\cU)\to(A,\cV)$ its restriction to the primitive elements, admits as a left adjoint the functor $\cV_A:\LieR{A}\to \CHalgd{A}$, which assigns to a Lie-Rinehart algebra $(A,L)$ its universal enveloping Hopf algebroid $\VL$ and to a morphism of Lie-Rinehart algebras $f:(A,L)\to (A,K)$ the morphism of co-commutative Hopf algebroids $\cV_A(f)$ induced by the universal property of $\VL$. The unit $L\to \prim(\VL)$ of the adjunction is given by the corestriction of the map $\iota_L$, while the counit $\cV_A(\prim(\cU))\to \cU$ is given by the universal property of its domain applied to the inclusion of $\prim(\cU)$ in $\cU$. The verification is straightforward. For the analogue in the case of left bialgebroids we refer to \cite[Theorem 3.1]{MoerdijkLie} or \cite[Proposition 4.2.3]{Kowalzig}.
\end{remark}

\begin{remark}\label{rem:LUEA-mod}
Given a Lie-Rinehart algebra $(A,L)$, there exists a notion of left  $(A,L)$-module, see \cite[\S1]{Huebschmann:Poisson}. As it happens for the universal enveloping algebra of an ordinary Lie algebra, the definition of the universal enveloping algebroid $\left(U(A,L),\jmath_A,\jmath_L\right)$ (as introduced \eg in \cite[page 63]{Huebschmann:Poisson}) is designed in such a way that left $(A,L)$-modules bijectively correspond to left $U(A,L)$-modules in a natural way, as claimed in \cite[page 65]{Huebschmann:Poisson}. In fact, this correspondence turns out to be an isomorphism of categories. In the present paper, working with right co-commutative Hopf algebroids, we are interested in dealing with \emph{right} modules over the universal enveloping algebroid associated to a Lie-Rinehart algebra. As a consequence, we define right $(A,L)$-modules to be left modules over $(A,L^{\Sscript{\text{op}}},-\omega)$, where $(A,L^{\Sscript{\text{op}}})$ is the Lie-Rinehart algebra with same underlying $A$-module $L$, with opposite bracket and opposite anchor map with respect to $(A,L)$ (equivalently, $A$-modules $M$ with a morphism of Lie-Rinehart algebras from $L^{\text{op}}$ to the Atiyah algebra of $M$). They are in one-to-one correspondence with right $\cV_A(L)$-modules. Moreover,
\begin{enumerate}[label=(\emph{\alph*})]
\item in general we have $\cV_A(L)\cong U(A,L^{\Sscript{\text{op}}})^{\Sscript{\text{op}}}$ (see \cite[Proposition 2.1.12]{ChemlaGavarini})
\item\label{item:2remLUEA-mod} in the particular case of ${}_AL$ free, \ie $L=\bigoplus_{i}AX_i$, we have that $U(A,L)$ with $\jmath_A$ and $\jmath'_L$ given by $\jmath'_L(\sum_ia_iX_i):=\sum_i\jmath_L(X_i)\jmath_A(a_i)$ is the right universal enveloping algebra of $(A,L)$ (symmetrically for $\cV_A(L)$ on the other side).
\end{enumerate}

It is worthy to point out however that our definition of a right representation differs slightly from the one given in \cite[page 430]{Huebschmann-LR}. The reason to introduce this new one is threefold: first of all this is more symmetric, secondly it ensures that $A$ is a right representation as much naturally as it is a left one, that is to say, via the anchor map $\omega$, and thirdly because with this definition right representations correspond to right modules over the right universal enveloping algebra in a natural way.
\end{remark}


\section{A dual for cocommutative Hopf algebroids}\label{sec:duality}

It is well-known that, for Hopf algebras, the functor $\Der{\K}{-}{\K}:\mathsf{CHAlg}^{\mathrm{op}}_{\Sscript{\Bbbk}}\to \mathsf{Lie}_{\Sscript{\Bbbk}}$ is right adjoint to the functor $\rcirc{(U(-))}:\mathsf{Lie}_{\Sscript{\Bbbk}}\to \mathsf{CHAlg}^{\mathrm{op}}_{\Sscript{\Bbbk}}$, where $\mathsf{CHAlg}_{\Sscript{\Bbbk}}$ and $\mathsf{Lie}_{\Sscript{\Bbbk}}$ denote the categories of commutative Hopf $\Bbbk$-algebras and that of Lie $\Bbbk$-algebras, respectively. Indeed, this can be seen as the composition of the two adjunctions $(U,\mathsf{Prim})$ and $(\rcirc{(-)},\rcirc{(-)})$, where $U:\mathsf{Lie}_{\Sscript{\Bbbk}}\to\mathsf{CCHAlg}_{\Sscript{\Bbbk}}$ is the universal enveloping functor, $\mathsf{Prim}:\mathsf{CCHAlg}_{\Sscript{\Bbbk}}\to \mathsf{Lie}_{\Sscript{\Bbbk}}$ is the functor of primitive elements and $\rcirc{(-)}$ denotes the finite (or Sweedler) dual. Since we plan to extend this construction to the Hopf algebroid framework, we first need an analogue of the finite dual. This section and the next one are devoted to this construction. In fact, by following two different but equally valid approaches, we will provide even two possible such analogues.

\subsection{Tannaka reconstruction process}\label{ssec:Tannaka}
Let  $A$ be a commutative algebra and $\fomega: \cat{A} \rightarrow \proj{A}$  be a  faithful $\K$-linear functor (referred to as  a \emph{fiber functor}), where $\cat{A}$ is a $\K$-linear (essentially) small category. The image ${^{\fomega} P}$ of an object $P$ of $\cat{A}$ under $\fomega$ will be denoted by $P$ itself when no confusion may be expected. Given $P,Q\in \cat{A}$, we denote by $T_{PQ}=\hom{\cat{A}}{P}{Q}$ the $\K$-module of all morphisms in $\cat{A}$ from $P$ to $Q$. The symbol $T_P$ is reserved to the ring (in fact, algebra) of endomorphisms of $P$.  Clearly, $S_P = \rend{P}{A}$ is a ring extension of $T_P$ via $\fomega$. In this way, every image ${}^{\fomega}P$ of an object $P \in \cat{A}$, becomes canonically a $(T_P,A)$-bimodule.

Now consider the following direct sum of $A$-corings
\begin{eqnarray*}
  \fk{B}\left(\cat{A}\right) &=& \bigoplus_{P  \, \in \,   \cat{A}} P^*\tensor{T_P}P
\end{eqnarray*}
and its $A$-sub-bimodule $\fk{J}_{\cat{A}}$ generated by the set
\begin{equation}\label{Eq:JA}
\LR{\underset{}{} q^* \tensor{T_Q} tp - q^* t \tensor{T_P}p |\, \, q^* \in Q^*,\, p \in P,\, t \in T_{PQ},\, P,Q \in \cat{A} },
\end{equation}
where $q^* t = q^* \circ \fomega(t)$ and  $tp=\fomega(t)(p)$. By \cite[Lemma 4.2]{ElKaoutit/Gomez:2004b}, $\fk{J}_{\cat{A}}$ is a coideal of the $A$-coring $\fk{B}(\cat{A})$. Therefore, we can consider the quotient $A$-coring
\begin{eqnarray}\label{comatInf}
  \Scr{R}\left( \cat{A}\right) :\,=\, \fk{B}(\cat{A})/\fk{J}_{\cat{A}} \,=\, \left(\underset{P\,\in\,\cat{A}}{\bigoplus} \rcomatrix{T_P}{P}\right)/ \fk{J}_{\cat{A}}
\end{eqnarray}
and this is the \emph{infinite comatrix $A$-coring} associated to the fiber functor $\fomega: \cat{A} \rightarrow \proj{A}$.  Furthermore, it is clear that any object $P \in \cat{A}$ admits (via the functor $\fomega$) the structure of a right $\Scr{R}(\cat{A})$-comodule, which leads to a well-defined  functor $\chi: \cat{A} \to \Acom{{\Scr{R}(\cA)}}$ (see \S\ref{ssec:Notations} for the notation), and that $\fomega$ factors through the forgetful functor $\Scr{O}: \Acom{\Scr{R}(\cat{A})} \to \rmod{A}$ via $\chi$, that is, $\fomega=\Scr{O}\circ \chi$.

The image of an element $p^{*}\tensor{T_{P}}p$ in the quotient $\Scr{R}(\cA)$, after identifying $p^{*}\tensor{T_{P}}p$ with its image in the direct sum, will be denoted by $\bara{p^{*}\tensor{T_{P}}p}$. These are generic elements in $\Scr{R}(\cA)$.  In fact, we have
\begin{equation}\label{Eq:Galway}
\sum_{i}^{n} \bara{p_{i}^{*}\tensor{T_{P_{i}}}p_{i}} \,\, =\,\, \bara{q^{*} \tensor{T_{Q}} q},
\end{equation}
where $Q=\oplus_{i=1}^{n}P_{i}$,  $q=p_{1}\dotplus \cdots \dotplus p_{n} \in Q$ and $q^{*}=\sum_{i} p_{i}^{*} \pi_{i} \in Q^{*}$, $\pi_{i}:Q \to P_{i}$ are the canonical projections.

\begin{remark}\label{rem:Adjunction}
The typical examples of the pairs $(\cA, \fomega)$ which we will deal with here are either the category $\Amod{R}$ of right $R$-modules, for a given $A$-ring $R$, which are finitely generated and projective as $A$-modules with  $\fomega$  the forgetful functor, or the category $\Acom{C}$ of right $C$-comodules, for a given $A$-coring $C$, which are finitely generated and projective as $A$-modules and $\fomega$ is the forgetful functor as well. In the first case  we obtain a functor $(-)^{\circ}: \ring{A} \to (\coanillos{A})^{\Sscript{op}}$, which was named \emph{the finite dual functor} in {\cite[\S2.1]{LaiachiGomez}}. It is noteworthy to mention that  from its own construction it is not clear whether the functor $(-)^{\circ}$ is left adjoint to the functor ${}^*(-) :(\coanillos{A})^{\Sscript{op}}  \to \ring{A}$ which sends any $A$-coring $C$ to its right convolution algebra ${}^*C$.  In the next section we will provide, using the Special Adjoint Functor Theorem (SAFT), a left adjoint of ${}^*(-)$ and study some of its properties.
\end{remark}

Assume now that $\cA$ is a symmetric rigid monoidal category and $\fomega$ is a symmetric strict monoidal functor.  Then
one can endow  the associated infinite comatrix $A$-coring $\rR(\cat{A})$  of equation \eqref{comatInf} with a structure of commutative $(A\tensor{\K}A)$-algebra.  The multiplication is given as follows:
\begin{equation}\label{Eq:mo}
(\bara{p^*\tensor{T_P}p})\,.\, (\bara{q^*\tensor{T_Q}q}) \,\, =\,\, \bara{(q^* \star p^*) \tensor{T_{Q\tensor{\peque{A}}P}} (q\tensor{\peque{A}}p)},
\end{equation}
where
\begin{equation}\label{Eq:star}
 (\varphi \tensor{A} \psi) : P\tensor{A}Q \longrightarrow A, \quad \lr{ x\tensor{\peque{A}}y \longmapsto   \varphi(x) \, \psi(y) }.
\end{equation}

The unit is the algebra map $A\tensor{\K}A \to \Scr{R}(\cA)$ which sends $a\tensor{}a' \to \bara{ l_{a}\tensor{T_{A}}a'}$, where $l_{a}$ is the image of $a$ by the isomorphism $A\cong A^{*}$ and as above we identify the identity object of $\cA$ with its image $A$. Notice that $T_{A}$ is a subring of $A$ and does not necessarily coincide with the base field $\Bbbk$.

It turns out that $(A,\rR(\cat{A}))$  with this algebra structure is actually a commutative Hopf algebroid. The antipode is given by the map
\begin{equation}\label{Eq:antipode}
\cS: \Scr{R}( \cA) \longrightarrow  \Scr{R}( \cA), \quad \Big(  \bara{p^{*}\tensor{T_{P}}p} \longmapsto \bara{ev_{p}\tensor{T_{P^{*}}}p^*} \Big),
\end{equation}
where $ev_{p}$ is the image of $p$ under the isomorphism of $A$-modules $P\cong (P^{*})^{*}$.

The previous construction, which we may call \emph{Tannaka's reconstruction process}, is in fact functorial. That is,  if $\cat{F}: \cat{A} \to \cat{A}'$ is a given symmetric monoidal $\K$-linear functor  such that
\begin{equation}\label{Eq:triangle}
\begin{gathered}
\xymatrix@R=15pt@C=40pt{  \cat{A} \ar@{->}^-{\cat{F}}[rr]  \ar@{->}_-{\fomega}[rd] & & \cat{A}'  \ar@{->}^-{\fomega'}[ld] \\ & \proj{A} & }
\end{gathered}
\end{equation}
is a commutative diagram,  then there is a morphism of Hopf algebroids $\Scr{R}(\cat{F}) : \Scr{R}(\cat{A}) \to \Scr{R}(\cat{A}')$ which renders commutative the following diagram:
\begin{equation}\label{Eq:diag}
\begin{gathered}
\xymatrix@R=15pt{   & \Acom{\rR(\cat{A})} \ar@{->}^-{\Scr{R}(\cF)_*}[rrr] \ar@/_1pc/@{->}^-{\oO}[rddd]  & & & (\cat{A}')^{\Sscript{\rR(\cat{A}')}} \ar@/^1.5pc/@{->}^-{\oO'}[llddd]  \\  \cat{A} \ar@{->}^-{\cat{F}}[rrr] \ar@{->}^-{\chi}[ru] \ar@/_1pc/@{->}^-{\fomega}[rrdd] & & & \cat{A}' \ar@{->}^-{\chi'}[ru] \ar@/^1pc/@{->}_-{\fomega'}[ldd] & \\ & & & &   \\ & & \proj{A} & & }
\end{gathered}
\end{equation}
where $\Scr{R}(\cF)_{*}$ is the restriction of the induced functor $\Scr{R}(\cF)_{*}: \rcomod{\rR(\cat{A})} \to \rcomod{\rR(\cat{A}')}$ sending any right $\rR(\cat{A})$-comodule $(M, \varrho_M)$ to the right $\rR(\cat{A}')$-comodule
$$
\xymatrix@C=40pt{M \ar@{->}^-{\varrho_M}[r]& M\tensor{A}\rR(\cat{A}) \ar@{->}^-{M\tensor{A}\Scr{R}(\cF)}[r]  & M \tensor{A}\rR(\cat{A}'), }
$$
and acting obviously on morphisms.
Explicitly, we have
\begin{equation}\label{eq:Rmap}
\Scr{R}(\cat{F}):\overline{p^* \tensor{T_P} p} \longmapsto \overline{p^* \tensor{T_{\cat{F}(P)}} p}.
\end{equation}

\begin{remark}\label{rem:Del}
It is noteworthy to mention that the underlying category $\cA$ is not assumed to be abelian nor the subalgebra $T_{A}$ of $A$ coincides with the base field $\K$. Thus we are not assuming that the pair $(\cA,\fomega)$ is a Tannakian category in the sense of \cite{Deligne:1990}. The obtained Hopf algebroids have then less properties then one constructed from the Tannakian categories. One of these missing properties is, for instance, that the functor $\chi: \cA \to \Acom{\rR(\cA)}$ is not necessarily an equivalence of categories, and that the skeleton of the full subcategory $\Acom{\rR(\cA)}$ does not necessarily form a set of small generators in the whole category of $\rR(\cA)$-comodules. Nevertheless, the conditions which we are taking on the pairs $(\cA, \fomega)$ are sufficient  to build up the construction of \S\ref{ssec:circ} below.
\end{remark}

Next we will give another description of the  Hopf algebroid $(A, \rR(\cA))$ by using rings with enough orthogonal idempotents and unital modules, which will be helpful in the sequel.  Let $(\cA,\fomega)$ as above, and consider the Gabriel's ring $\bara{\cA}$ attached to $\cA$  introduced in \cite{Gabriel:1962}. That is,  using the above notation, we have that  $\bara{\cA}:=\oplus_{P,\, Q\, \in \, \cA}T_{PQ}$ is an algebra with enough orthogonal idempotents, and  where  the multiplication of two composable morphisms is  their composition, otherwise is zero. Set $\Sigma=\oplus_{P\, \in \cA}P$ and $\Sigmad=\oplus_{P\, \in \cA}P^{*}$ direct sums of $A$-modules, and identify any element in $P$ (resp.~in $Q^{*}$) with its image in $\Sigma$ (resp.~in $\Sigmad$). It turns out that $\Sigma$ is an unital $(\bara{\cA}, A)$-bimodule while $\Sigmad$ is an unital $(A,\bara{\cA})$-bimodule. Therefore one can perform the $A$-bimodule $\infcomatrix{\cA}{\Sigma}$. The  pair $(A,\infcomatrix{\cA}{\Sigma})$ is also a commutative Hopf algebroid (its structure maps are identical to those exhibited in equations \eqref{Eq:mo} and \eqref{Eq:antipode}), and by the universal property of $\rR(\cA)$ we have  that the map
\begin{equation}\label{Eq:coringsiso}
(A,\rR(\cA) ) \longrightarrow (A, \infcomatrix{\cA}{\Sigma}), \quad \Big( \bara{p^{*}\tensor{T_{P}}p} \longmapsto p^{*}\tensor{\bara{\cA}}p \Big)
\end{equation}
establishes an isomorphism of Hopf algebroids, as it was shown in \cite{ElKaoutit/Gomez:2004b}.  Moreover this isomorphism is  natural with respect to the pairs $(\cA, \fomega)$. In this way, for a given functor $\cF: (\cA,\fomega) \to (\cA',\fomega')$  satisfying \eqref{Eq:triangle}  its image is  given by:
\begin{equation}\label{Eq:RF}
\rR(\cF): \infcomatrix{\cA}{\Sigma} \longrightarrow \infcomatrix{\cA'}{\Sigma}, \quad \Big( p^{*}\tensor{\bara{\cA}}p \longmapsto p^{*}\tensor{\bara{\cA'}}p \Big).
\end{equation}

\subsection{The zeta map and Galois corings}\label{ssec:Galoiscoring}
Let $(A, R)$ be a ring over $A$ and consider its final dual $(A,\Circ{R})$ constructed as in \S\ref{ssec:Tannaka} from the pair $(\Amod{R}, \fomega)$, where $\fomega$ is the forgetful functor, see also  Remark \ref{rem:Adjunction}. Then there is an $(A,A)$-bimodule map

\begin{equation}\label{Eq:zeta}
\mb{\zeta}_{\Sscript{R}}: \Circ{R} \longrightarrow  R^{*}, \quad \Big(  p^{*}\tensor{\bara{\Amod{R}}} p \longmapsto \Big[  r \mapsto p^{*}(p \, r) \Big] \Big)
\end{equation}
where the latter is the right $A$-linear dual of $R$ endowed with  its canonical $A$-bimodule structure.

\begin{remark}\label{rem:zetanat}
Notice that $\zeta$ should be more properly denoted by $\mb{\zeta}_R$ if we want to stress the dependence on $R$. Moreover, if $f:S\to R$ is an $A$-ring map, then
$\rdual{f}\circ \mb{\zeta}_R = \mb{\zeta}_S\circ \rcirc{f}.$
Indeed,
\begin{equation*}
\mb{\zeta}_S\left(\rcirc{f}\left(\overline{\varphi\tensor{T_N^R}n}\right)\right)(s) \stackrel{\eqref{eq:Rmap}}{=} \mb{\zeta}_S\left(\overline{\varphi\tensor{T_{\Amod{f}(N)}^S}n}\right)(s) = \varphi\left(nf(s)\right) = \rdual{f}\left(\mb{\zeta}_R\left(\overline{\varphi\tensor{T_N^R}n}\right)\right)(s)
\end{equation*}
for all $s\in S, \overline{\varphi\tensor{T_N^R}n}\in\rcirc{R}$ and where $T_N^R=\End{R}{N}$.
\end{remark}

For the reader's sake, we include here the subsequent result.

\begin{lemma}[{\cite[3.4]{LaiachiGomez}}]\label{lema:zeta}
The map $\mb{\zeta}$ of \eqref{Eq:zeta} fulfils the following equalities for every $z\in R^{\circ}$, $x,y\in R$
\begin{equation}\label{eq:zetaprop}
\zeta(z)\big(xy)=\zeta(z_1)\Big(\zeta(z_2)(x)y\Big), \qquad \zeta(z)(1_{R})=\varepsilon(z) \qquad \text{and} \qquad \zeta(azb)(u)=a\zeta(z)(bu).
\end{equation}
\end{lemma}

In contrast with the classical case of algebras over fields, the map $\mb{\zeta}$ is not known to be injective,  unless some condition are imposed on the base algebra $A$. For instance, if $A$ is a Dedekind domain then $\mb{\zeta}$ is always injective. Strong consequences of the injectivity of $\mb{\zeta}$ were discussed in \cite{LaiachiGomez}, some of them can be seen as follows.  In general, it is known that the functor $ \cL: \Acom{\Circ{R}} \to \Amod{R}$ induced by the obvious functor $\Acom{\Circ{R}} \to \Amod{\ldual{\left(\Circ{R}\right)}}$ (see \eg \cite[\S19.1]{BrzezinskiWisbauer}) and by the canonical map $\eta_R:R\to \ldual{\left(\Circ{R}\right)}$ (where $\eta_R(r)(p^*\tensor{\bara{\cA_{R}}}p)=\rdual{p}(pr)$ for every $R$-module $P$ and all $r\in R, p\in P, \rdual{p}\in\rdual{P}$) has a right inverse functor  $\chi: \Amod{R} \to \Acom{\Circ{R}}$ which sends each right $R$-module $P \in \Amod{R}$ to the right $\Circ{R}$-comodule $\cL(P)$ with underlying $A$-module $P$ and coaction
$$
\varrho:P \longrightarrow P\tensor{A}\Circ{R}, \quad \Big( p \longmapsto \sum_{i}e_{i}\tensor{A}\big( e_{i}^{*} \tensor{\bara{\Amod{R}}} p\big)\Big),
$$
where $\{e_{i},e_{i}^{*}\}_{i}$ is any dual basis for $P$. If $\zeta$ is assumed to be injective then $\chi$ and $\cL$ are mutually inverses and so  $\Amod{R}$ is isomorphic to $\Acom{\Circ{R}}$ (see Remark \ref{rem:drwho}). Now we give  the notion of \emph{Galois corings}.

\begin{definition}\label{def:Galois}
Let $(A, C)$ be a coring. Then $(A, C)$ is said to be \emph{Galois} (or \emph{$\Acom{C}$-Galois}), if it can be reconstructed from  the category $\Acom{C}$, that is, provided that the canonical map
\begin{equation*}
\can{}: \infcomatrix{\Acom{C}}{\Sigma} \longrightarrow C, \quad \Big(  p^{*}\tensor{\bara{\Acom{C}}}p \longmapsto p^{*}(p_{\Sscript{(0)}}) \,  p_{\Sscript{(1)}} \Big),
\end{equation*}
is an isomorphism of $A$-corings, where $\varrho_{\Sscript{P}}(p)=p_{\Sscript{(0)}} \tensor{A} p_{\Sscript{(1)}}$ is the $C$-coaction on $p \in P$.
\end{definition}

\subsection{The finite dual of co-commutative Hopf algebroid via Tannaka reconstruction}\label{ssec:circ}
Next we want to apply the Tannaka reconstruction process to a certain full subcategory of the category of right modules  over a co-commutative  Hopf algebroid. So take $(A, \cU)$ to be such a Hopf algebroid. Following the notation of \S\ref{ssec:Notations}, we denote by $\Amod{\cU}$ the full subcategory of right $\cU$-modules whose underlying $A$-module is finitely generated and projective,  and by $\fomega: \Amod{\cU} \to \proj{A}$ the associated forgetful functor.  Joining together the results from \S\ref{ssec:cocom} and \S\ref{ssec:Tannaka}, we get that the pair $(\Amod{\cU}, \fomega)$ satisfies the necessary assumptions such that the  algebra $(A,\rR(\Amod{\cU}))$ resulting from the Tannaka reconstruction process is a commutative Hopf algebroid. It is this Hopf algebroid which we refer to as the \emph{finite dual} of $(A,\cU)$ and we denote it by $(A, \Circ{\cU})$. The subsequent result is contained in \cite[Theorem 4.2.2]{LaiachiGomez}. We give here the main steps of its proof.

\begin{invisible}
We just point out that in the reference it is claimed that $(A,\rcirc{U})$ is a commutative Hopf algebroid \emph{over $T_A$}. In light of the terminology introduced there in the footnote 2 at page 5 this means that  the pair $(A,\rcirc{U})$ forms a Hopf algebroid even when we consider both $A$ and $\rcirc{U}$ as algebras over $T_A$, which is an extension of $\K$. In particular, it forms a commutative Hopf algebroid over $\K$.
\end{invisible}
\begin{proposition}\label{prop:Fdual}
Let $A$ be a commutative algebra. Then the finite dual establishes a contravariant functor
$$
(-)^{\circ}: \CHalgd{A} \longrightarrow \Halgd{A}
$$
from the category of co-commutative Hopf algebroids to the category of commutative ones.
\end{proposition}

\begin{proof}
Given a morphism $\phiup: \cU \to \cU'$ of co-commutative Hopf algebroids,  the  restriction of scalars leads to  a $\K$-linear functor $\cF_{\Sscript{\phi}}: \Amod{\cU'} \to \Amod{\cU}$ which commutes with the  forgetful  functor, that is, such that $\fomega \circ \cF_{\Sscript{\phi}}=\fomega'$. Using the monoidal structure described in  \eqref{Eq:mono}, it is easily checked that  $\cF_{\Sscript{\phi}}$ is a symmetric strict monoidal  functor. Therefore (see \S\ref{ssec:Tannaka}) we have a morphism $\phiup^{\circ}: \cU'{}^{\circ} \to \Circ{\cU}$ of Hopf algebroids. The compatibility of  $(-)^{\circ}$ with the composition law and the identity morphisms is obvious.
\end{proof}

\subsection{The zeta map and Galois Hopf algebroids}\label{ssec:Galois}
Let $(A, \cU)$ be a co-commutative Hopf algebroid and consider
its right $A$-linear dual $\rdual{\cU}$, regarded as an $(A\otimes A)$-algebra with the convolution product induced by the comultiplication $\Delta:\cU_A\to \cU_A\tensor{A}\cU_A$, that is to say,
\begin{equation*}
(f*g)(u)=f( u_1)g(u_2),\quad  \text{for every }\, f,g \in \rdual{\cU}, u \in \cU .
\end{equation*}
The canonical $A$-bilinear map from \S\ref{ssec:Galoiscoring}
\begin{equation}\label{Eq:zetaH}
\zeta = \boldsymbol{\zeta}_{\cU}: \Circ{\cU} \longrightarrow  \cU^{*}, \quad \Big(  p^{*}\tensor{\bara{\Amod{\cU}}} p \longmapsto \Big[  u \mapsto p^{*}(p \, u) \Big] \Big)
\end{equation}
is an $(A\tensor{}A)$-algebra map and it fulfils \eqref{eq:zetaprop} for $R=\cU$. If $\zeta$ is injective, then there is an isomorphism of symmetric rigid monoidal categories $\Acom{\Circ{\cU}}\cong \Amod{\cU}$ (see \cite[Theorem 4.2.2]{LaiachiGomez}).
The subsequent definition is a particular instance of Definition \ref{def:Galois}.
\begin{definition}\label{def:GaloisH}
A commutative Hopf algebroid $(A, \cH)$ is called \emph{Galois} (or \emph{$\Acom{\cH}$-Galois}), if  its underlying $A$-coring is Galois in the sense of Definition \ref{def:Galois}, \ie if the canonical map
\begin{equation*}
\can{}: \infcomatrix{\Acom{\cH}}{\Sigma} \longrightarrow \cH, \quad \Big(  p^{*}\tensor{\bara{\Acom{\cH}}}p \longmapsto s(p^{*}(p_{\Sscript{(0)}})) p_{\Sscript{(1)}} \Big),
\end{equation*}
is an isomorphism of Hopf algebroids, where $\varrho_{\Sscript{P}}(p)=p_{\Sscript{(0)}} \tensor{A} p_{\Sscript{(1)}}$ is the $\cH$-coaction on $p \in P$. The full subcategory of Galois commutative Hopf algebroids with base algebra $A$ is denoted by $\Sf{GCHAlgd}_{\Sscript{A}}$.
\end{definition}

\begin{remark}\label{rem:circGalois}
Let $(A,\cU)$ be a co-commutative Hopf algebroid. When the canonical map $\zeta:\Circ{\cU}\to\rdual{\cU}$ is injective, the reconstructed object $\Circ{\cU}$ is Galois (see \cite[Proposition 3.3.3]{LaiachiGomez}). The inverse of the canonical map $\can{}$ is provided by the assignment $\infcomatrix{\Amod{\cU}}{\Sigma}\to \infcomatrix{\Acom{\Circ{\cU}}}{\Sigma}, \, p^{*}\tensor{\bara{\Amod{\cU}}}p \mapsto p^{*}\tensor{\bara{\Acom{\Circ{\cU}}}}p$, employing the canonical isomorphism $\Acom{\Circ{\cU}}\cong \Amod{\cU}$.
Later on, we will recover the same isomorphism under an apparently weaker condition. We point out also that this condition makes of $\Circ{\cU}$ a Galois coring, even if we replace $\cU$ simply by an $A$-ring $R$ (see \eg Remark \ref{rem:circGalois2}).
\end{remark}

\begin{example}\label{exam:Galois}
Several well-known Hopf algebroids are Galois as the following list of examples shows.
\begin{enumerate}
\item Any commutative Hopf algebra  over a field (i.e., a Hopf algebroid with source equal target with base algebra is a  field) is Galois Hopf algebroid.
\item Let $B \to A$ be a faithfully flat extension of commutative algebras. Then $(A, A\tensor{B}A)$ is a Galois Hopf algebroid.
\item Any Hopf algebroid $(A,\cH)$ whose unit map $\etaup: A\tensor{}A \to \cH$ is a faithfully flat extension of algebras is actually Galois. In other words, any geometrically transitive Hopf algebroid is Galois,  see \cite{ElKaoutit:2015} for more details.
\item The \emph{Adams Hopf algebroids} as defined in \cite{Hovey:2004} and studied in \cite{Schappi:2012} are Galois.
\end{enumerate}
We point out that the first three cases are in fact a particular instance of a more general result \cite[Theorem 5.7]{ElKaoutit/Gomez:2004b}, which asserts that  any flat Hopf algebroid whose category of comodules $\rcomod{\cH}$ admits $\Acom{\cH}$ as a set of small generators, is a Galois Hopf algebroid.
\end{example}


\section{An alternative dual via SAFT}\label{ssec:saft}
In this section we propose a different candidate for the finite dual of a given co-commutative Hopf algebroid. Its construction is based upon the well-known  Special Adjoint Functor Theorem. We also establish a natural transformation between this new contravariant functor and the one already recalled in Subsection \ref{ssec:circ}. As before, we start by the general setting of rings.

\subsection{Finite dual using SAFT: The general case of $A$-rings.}\label{ssec:FDGeneral}
Let $A$ be a commutative algebra. Consider the category $\bimod{A}$ of $A$-bimodules. Then the functor $\rdual{(-)}:\bimod{A}\to (\bimod{A})^{\Sscript{\text{op}}}$ admits a right adjoint $\ldual{(-)}:(\bimod{A})^{\Sscript{\text{op}}}\to \bimod{A}$, where $\rdual{M}=\hom{-\, A}{M}{A}$ with structure of $A$-bimodules as in \eqref{Eq:actions}. The latter functor induces a functor
\begin{equation}\label{eq:dual}
\ldual{(-)}:(\coanillos{A})^{\Sscript{\text{op}}}\longrightarrow \ring{A},
\end{equation}
where the category $\ring{A}$ stands for $\Bbbk$-algebras $R$ with an algebra map $A \to R$ (whose image is not necessarily in the centre of $R$). The functor of \eqref{eq:dual} is explicitly given as follows: For a  given an $A$-coring $(C,\Delta,\varepsilon)$ we have that the $A$-ring structure on $\ldual{C}$ is given as in \eqref{eq:convolution}.
As a consequence of the Special Adjoint Functor Theorem, the functor of equation \eqref{eq:dual} admits a left adjoint
 \begin{equation}\label{eq:dag}
 \rsaft{(-)}:\ring{A}\to (\coanillos{A})^{\Sscript{\text{op}}},
 \end{equation}
 see \cite[Corollary 9]{Porst-Street}. For future reference, let us retrieve explicitly the $A$-ring morphism
\begin{equation}\label{eq:unitbullet}
\eta'_R :R\longrightarrow \ldual{\left( R^{\bullet }\right)},\qquad \Big(r\longmapsto \left[ z\longmapsto \xi \left( z\right) \left( r\right) \right]\Big)
\end{equation}
(i.e., unit of the previous adjunction).

\begin{remark}\label{rem:rsaft}
Given an $A$-ring $R$, the $A$-coring $\rsaft{R}$ is uniquely determined by the following universal property: it comes endowed with an $A$-bimodule morphism $\xi:\rsaft{R}\to \rdual{R}$ which satisfies the analogous of the relations \eqref{eq:zetaprop} and if $C$ is an $A$-coring endowed with a $A$-bimodule map $f:C\to \rdual{R}$ satisfying the same relations, then there is a unique $A$-coring map $\what{f}:C\to \rsaft{R}$ such that $\xi\circ \what{f}=f$. Conversely, notice that given an $A$-coring map $g:C\to \rsaft{R}$, the composition $\xi\circ g$ satisfies the relations in \eqref{eq:zetaprop}. As a consequence, if $g,g':C\to \rsaft{R}$ are coring maps such that $\xi\circ g=\xi \circ g'$, then $g=g'$.
\end{remark}

\begin{remark}\label{rem:univpropbullet}
For the reader sake, we show how the adjunction follows from this universal property.  Let $R$ be an $A$-ring, let $C$ be an $A$-coring and $h:R\to \ldual{C}$ be a $\K$-linear map. Denote by $f:C\to \rdual{R}$ the map defined by $f(c)(r)=h(r)(c)$ for all $r\in R$ and $c\in C$. We compute
 \begin{align*}
  h(bxa)(c) & = f(c)(bxa) = f(c)(bx)a = af(c)(bx) =(af(c)b)(x) , \\
 (bh(x)a)(c) & \stackrel{\eqref{eq:convolution}}{=} h(x)(cb)a = ah(x)(cb) = h(x)(acb) = f(acb)(x), \\
 h(xy)(c) & = f(c)(xy), \\
 (h(x)*h(y))(c) & \stackrel{\eqref{eq:convolution}}{=} h(y)(c_1h(x)(c_2)) = \big(h(x)(c_2)h(y)\big)(c_1) = h(h(x)(c_2)y)(c_1) = f(c_1)(f(c_2)(x)y), \\
 h(1_R)(c) & =  f(c)(1_R), \\
  1_{\rdual{C}}(c) & \stackrel{\eqref{eq:convolution}}{=} \varepsilon(c).
 \end{align*}
 Consequently, we see that $h$ from $R$ to $\ldual{C}$ is an $A$-ring morphism if and only if $f$ corresponds to $h$ via the adjunction $(\rdual{(-)},\ldual{(-)})$ and satisfies the conditions in \eqref{eq:zetaprop}. Since there is a $1$--$1$ correspondence between these $f$'s and the $\what{f}\,$'s as above, we are done. Note also that given an $A$-ring map $h:R_1\to R_2$, we can consider the $A$-bimodule map $\rdual{h}:\rdual{R_2}\to \rdual{R_1}$. If we pre-compose $\rdual{h}$ with $\xi_2:\rsaft{R_2}\to \rdual{R_2}$, the map $f:=\rdual{h}\circ \xi_2$ satisfies conditions \eqref{eq:zetaprop} since $f(z)(r)=\xi_2(z)(h(r))$ for all $z\in\rsaft{R_2}$, $r\in R_1$ and $h$ is multiplicative, unital and $A$-bilinear. As a consequence, the universal property of $\rsaft{R_2}$ yields a unique $A$-coring map $\rsaft{h}:=\what{f}:\rsaft{R_2}\to \rsaft{R_1}$ such that $\xi_1\circ \rsaft{h} = \rdual{h} \circ \xi_2$.
\end{remark}

\begin{example}[the map zeta-hat]\label{exam:zetahat}
Let $R$ and $\Circ{R}$ as in \S\ref{ssec:Galoiscoring} together with the  $A$-bimodules morphism $\zeta$ of equation \eqref{Eq:zeta}.  By Lemma \ref{lema:zeta} and the universal property of $\rsaft{R}$,  there is an $A$-corings morphism
\begin{equation}\label{Eq:zetahat}
\what{\zeta}: R^{\circ}  \longrightarrow \rsaft{R},
\end{equation}
such that $\xi \circ \what{\zeta}= \zeta$. In light of Remark \ref{rem:zetanat}, this induces a natural transformation
$
\mb{\what{\zeta}}:\rcirc{(-)} \rightarrow  \rsaft{(-)}.
$
\end{example}

 \begin{lemma}\label{lem:maggico}
 Given an $A$-ring $R$ and the canonical map $\mb{\xi}_{\Sscript{R}}:=\xi:\rsaft{R}\to \rdual{R}$, we have that $\ker{\xi}$ contains no non-zero coideals of $\rsaft{R}$ (\ie $\xi$ is \emph{cogenerating} in the sense of {\cite[Definition 1.13]{Michaelis-PBW}}). In particular, $\xi$ is injective if and only if $\ker{\xi}$ is a coideal of $\rsaft{R}$.
 \end{lemma}
 \begin{proof}
 By definition, a coideal $\cJ$ of $\rsaft{R}$ is an $A$-subbimodule such that the quotient $A$-bimodule $C:=\rsaft{R}/\cJ$ is an $A$-coring and the canonical projection $\pi:\rsaft{R}\to C$ is an $A$-coring map. If $\cJ\subseteq \ker{\xi}$, then $\xi$ factors through a map $\bar{\xi}:C\to \rdual{R}$ such that $\bar{\xi}\circ \pi=\xi$. Any $c\in C$ is of the form $\pi(x)$ for some $x\in \rsaft{R}$, so that
 \begin{align*}
 \bar{\xi}(c_1)\big(\bar{\xi}(c_2)(r)r'\big) & = \bar{\xi}(\pi(x)_1)\big(\bar{\xi}(\pi(x)_2)(r)r'\big) = {\xi}(x_1)\big({\xi}(x_2)(r)r'\big) \stackrel{\eqref{eq:zetaprop}}{=} \xi(x)(rr') = \bar{\xi}(c)(rr'), \\
  \bar{\xi}(c)(1_R) & = \xi(x)(1_R) \stackrel{\eqref{eq:zetaprop}}{=}  \varepsilon_{\rsaft{R}}(x) = \varepsilon_C(c), \\
   \bar{\xi}(acb)(r) & = \bar{\xi}(a\pi(x)b)(r)=\bar{\xi}(\pi(axb))(r)=\xi(axb)(r) = a\xi(x)(br)= a\bar{\xi}(c)(br).
 \end{align*}
 As a consequence of the universal property of $\rsaft{R}$, there exists a unique $A$-coring map $\sigma:C\to \rsaft{R}$ such that $\xi\circ \sigma = \bar{\xi}$. Now, $\xi\circ \sigma\circ \pi= \bar{\xi}\circ \pi = \xi$, so that the uniqueness in the universal property entails that $\sigma\circ \pi = \id{\rsaft{R}}$. Since $\pi$ is surjective, this forces $\pi$ to be invertible whence $\cJ=0$.
 \end{proof}

Next, we want to relate the two categories $\Amod{R}$ and $\Acom{\rsaft{R}}$ (see \S \ref{ssec:Notations} for definition), but before we recall the following general construction that has been and will be used more or less implicitly along the paper. As a matter of notation, if ${}_{B}M_{A}$ is a $\left(B,A\right) $-bimodule such that $M_{A}$ is finitely generated and projective with dual basis $\left\{ e_{i},e_{i}^{\ast }\right\} _{i}$, then we are going to set
$$
db_{M} :B\rightarrow M\tensor{A}M^{\ast },\quad \left(b\longmapsto \sum_{i}be_{i}\tensor{A}e_{i}^{\ast }\right) \qquad \text{and} \qquad
ev_{M} :M^{\ast }\tensor{B}M\rightarrow A,\quad \Big(f\tensor{B}m\longmapsto f\left( m\right) \Big).
$$
Notice that $db$ is $B$-bilinear while $ev$ is $A$-bilinear and we have the following isomorphism
\begin{align}
\beta :\hom{D-B}{M}{N\tensor{C}P} \rightarrow \hom{C-B}{N^{\ast }\tensor{D}M}{P}, \quad \Big(g\longmapsto \left( ev_{N}\tensor{C}P\right) \circ \left( N^{\ast }\tensor{D}g\right)\Big) \label{eq:beta}
\end{align}
for $B,C,D$ algebras and ${}_{D}M_{B}, {}_{D}N_{C}, {}_{C}P_{B}$ bimodules such that $N_{C}$ is finitely generated and projective.

\begin{invisible}
The following direct computations
\begin{align*}
\sum_{i}ae_{i}\tensor{A}e_{i}^{\ast } &=\sum_{i,j}e_{j}e_{j}^{\ast }\left( ae_{i}\right) \tensor{A}e_{i}^{\ast }=\sum_{j}e_{j}\otimes_{A}\sum_{i}e_{j}^{\ast }\left( ae_{i}\right) e_{i}^{\ast } =\sum_{j}e_{j}\tensor{A}e_{j}^{\ast }a,\\
\left( fa\right) \left( n\right) &=f\left( an\right) ,\\
\left( cf\right) \left( nc^{\prime }\right) &=cf\left( n\right) c^{\prime },
\end{align*}
imply that $db$ is well-defined and $B$-bilinear while $ev$ is well-defined and $A$-bilinear.
\end{invisible}

For every $(B,A)$-bimodule $N$ we set
$${_B\Coac_A}(N,N\tensor{A}C) := \left\{\rho \in \hom{B-A}{N}{N\tensor{A}C}\mid (N,\rho)\in\Acom{C}\right\}.$$

\begin{lemma}\label{lemma:coactions}
For every $(B,A)$-bimodule $N$ such that $N_A$ is finitely generated and projective, the assignment
$
\beta_C :\hom{B-A}{N}{N\tensor{A}C} \rightarrow \hom{A-A}{N^{\ast }\tensor{B}N}{C}
$
of Equation \eqref{eq:beta} induces an isomorphism
\begin{equation}\label{eq:betabar}
\bar{\beta}_C : {_B\Coac_A}(N,N\tensor{A}C) \to \Coring_A(N^*\tensor{B}N,C)
\end{equation}
natural in $C$.
\end{lemma}

\begin{proof}
By adapting \cite[Proposition 2.7]{ComatrixGalois}, one proves that $\beta_C$ induces $\bar{\beta}_C$. A direct computation shows that $\beta_C^{-1}$ restricts to $\bar{\beta}_C^{-1}:\Coring_A(N^*\tensor{B}N,C)\to {_B\Coac_A}(N,N\tensor{A}C)$, providing an inverse for $\bar{\beta}_C$.
\begin{invisible}
To this aim, consider a coring morphism $\varphi\in \Coring_A(N^*\tensor{B}N,C)$ and set $\gamma:=\beta_C^{-1}(\varphi)$. Let $\{e_i,e_i^*\}_i$ be a dual basis for $N$. Then the following direct computation
\begin{align*}
(N\tensor{A} \varepsilon)\gamma(n) & = \sum_{i}e_i\varepsilon\varphi(e_i^*\tensor{B}n) = n \\
(N\tensor{A} \Delta)\gamma(n) & = \sum_{i}e_i \tensor{A} \Delta\varphi(e_i^*\tensor{B}n) = \sum_{i,j}e_i \tensor{A} \varphi(e_i^*\tensor{B}e_j) \tensor{A} \varphi(e_j^*\tensor{B}n)  \\
 & = \sum_{j}\gamma(e_j) \tensor{A} \varphi(e_j^*\tensor{B}n) = (\gamma \tensor{A} C) \gamma(n)
\end{align*}
shows that $\gamma\in {_B\Coac_A}(N,N\tensor{A}C)$. Thus, $\beta_C^{-1}$ restricts to $\bar{\beta}_C^{-1}:\Coring_A(N^*\tensor{B}N,C)\to {_B\Coac_A}(N,N\tensor{A}C)$, providing an inverse for $\bar{\beta}_C$.
\end{invisible}
\end{proof}

\begin{lemma}\label{lemma:fmono}
Let $C$ be an $A$-coring, $M$ an $A$-bimodule and $f:C\to M$ an $A$-bilinear map. The following are equivalent
\begin{enumerate}[label=({\alph*}),ref=\emph{(\alph*)}]
\item\label{item:fmono2} $(N\tensor{A}f)\circ\rho = (N\tensor{A}f)\circ\rho'$ implies $\rho=\rho'$ for every $\rho,\rho'\in{\Coac_A}(N,N\tensor{A} C)$ and for every $N\in \proj{A}$,
\item\label{item:fmono3} $f\circ\alpha = f\circ\beta$ implies $\alpha=\beta$ for every $\alpha,\beta: E\to C$ coring maps and for every $A$-coring $E$ with $\can{E}$ (split) epimorphism of corings,
\item\label{item:fmono1} $f\circ\alpha = f\circ\beta$ implies $\alpha=\beta$ for every $\alpha,\beta: N^*\tensor{B} N\to C$ coring maps, for every algebra $B$ and every bimodule ${_BN_A}$ such that $N_A\in \proj{A}$.
\end{enumerate}
\end{lemma}

\begin{proof}
First of all, observe that \ref{item:fmono2} is equivalent to the same statement but with $\rho,\rho'\in{_B\Coac_A}(N,N\tensor{A} C)$, for every algebra $B$ and every bimodule ${_BN_A}$ such that $N_A\in \proj{A}$.
To prove that \ref{item:fmono1} is equivalent to \ref{item:fmono2} consider the commutative diagram, for every $N\in\proj{A}$,
\begin{equation*}
\xymatrix{
{{_B\Coac_A}(N,N\tensor{A}C)} \ar[r]^{\bar{\beta}_{C}} \ar@{^(->}[d] & \Coring_A(N^*\tensor{B} N,C) \ar@{^(->}[d] \\
\hom{B-A}{N}{N\tensor{A}C} \ar[r]^{{\beta}_{C}} \ar[d]_-{\hom{B-A}{N}{N\tensor{A}f}} & \hom{A-A}{N^*\tensor{B} N}{C} \ar[d]^-{\hom{A-A}{N^*\otimes N}{f}} \\
\hom{B-A}{N}{N\tensor{A}M} \ar[r]^{{\beta}_{\rdual{R}}} & \hom{A-A}{N^*\tensor{B} N}{M}
}
\end{equation*}
Since the horizontal arrows are isomorphisms, the vertical composition on the right is injective (\ie \ref{item:fmono1} holds) if and only if the vertical composition on the left is (\ie \ref{item:fmono2} holds).

To prove the remaining implications, let us show first that $\rdual{N}\tensor{B}N$ is a coring with $\can{\rdual{N}\tensor{B}N}$ (split) epimorphism of corings, for every algebra $B$ and every bimodule ${_BN_A}$ as in the statement. Notice that $N\in \Acom{\rdual{N}\tensor{B}N}$ with coaction $n\mapsto \sum_i e_i\tensor{A} (e_i^*\tensor{B}n)$, where $\{e_i,e_i^*\}_i$ is a dual basis for $N_A$. Thus we may consider the composition
$$
\xymatrix@R=5pt{
N^*\tensor{B}N \ar[r]^-{(*)} & N^*\tensor{T_N}N \ar[r]^-{\iota_N} & \rR(N^*\tensor{B}N) \ar[r]^-{\can{N^*\tensor{B}N}} & N^*\tensor{B}N \\
f \tensor{B} n \ar@{|->}[r] & f\tensor{T_N} n \ar@{|->}[r] & \overline{f\tensor{T_N} n} \ar@{|->}[r] & \sum_if(e_i)e_i^*\tensor{B}n = f\tensor{B}n
}
$$
where $(*)$ is the isomorphism of \cite[Lemma 3.9]{ComatrixGalois} and $T_N:=\mathsf{End}^{N^*\tensor{B}N}(N)$. This shows that $\can{N^*\tensor{B}N}$ is a (split) epimorphism of corings for every $N$ and $B$ as above and hence \ref{item:fmono1} follows from \ref{item:fmono3}.

Conversely, let us show that \ref{item:fmono1} implies \ref{item:fmono3}. Let $\alpha,\beta,E$ be as in \ref{item:fmono3} such that $f\circ\alpha = f\circ\beta$. Denote by $\pi:\bigoplus_{N\, \in\,\Acom{E}}N^*\tensor{T_N}N \to \rR(E)$ the canonical projection and by $j_N:N^*\tensor{T_N}N\to\bigoplus_{N\,\in\,\Acom{E}}N^*\tensor{T_N}N$ the canonical injection. Then we have that
$$
f\circ \alpha \circ \can{E} \circ \pi \circ j_N = f\circ \beta \circ \can{E} \circ \pi \circ j_N
$$
for every $N\in\Acom{E}$. In light of the hypothesis and since $\iota_N$ and $\pi$ are morphisms of corings, we have that $\alpha \circ \can{E} \circ \pi \circ \iota_N = \beta \circ \can{E} \circ \pi \circ \iota_N$. By the universal property of the coproduct, the surjectivity of $\pi$ and the fact that $\can{E}$ is a (split) epimorphism we get that $\alpha = \beta$.
\end{proof}

\begin{corollary}
For every $A$-ring $R$, the canonical morphism $\xi:\rsaft{R}\to \rdual{R}$ satisfies the equivalent properties of Lemma \ref{lemma:fmono}.
\end{corollary}

\begin{proof}
From the universal property of $\xi$ (see Remark \ref{rem:rsaft}), it satisfies \ref{item:fmono3} of Lemma \ref{lemma:fmono}.
\end{proof}

\begin{remark}
An open question at the present moment is whether $\zeta:\rcirc{R}\to \rdual{R}$ satisfies \ref{item:fmono3} of Lemma \ref{lemma:fmono} as well. Note that $\can{\rcirc{R}}$ is a split epimorphism of corings because $\can{\rcirc{R}}\circ \rR(\chi) = \id{\rcirc{R}}$. As we will see, an affirmative answer would be equivalent to require that $\Acom{\what{\zeta}}$ is an isomorphism of categories.
\end{remark}

Now, in one direction, we have that every object $(M, \varrho_{\Sscript{M}})$ in $\Acom{\rsaft{R}}$ becomes right $R$-module as follows:
\begin{equation}\label{Eq:Raction}
m \ldot r \,=\,  m_{\Sscript{(0)}} \xi(m_{\Sscript{(1)}})(r)
\end{equation}
for every $m\in M$ and $r\in R$. Its underlying $A$-module coincides with the image of $(M,\varrho_{\Sscript{M}})$ by the forgetful functor $\oO:\Acom{\rsaft{R}} \to \rmod{A}$. Clearly this construction is functorial and so  we have a functor
\begin{equation}\label{chi}
\cL': \Acom{\rsaft{R}} \longrightarrow \Amod{R}
\end{equation}
such that
\begin{equation}\label{eq:deflL}
\cL'\circ \Acom{\what{\zeta}} = \cL,
\end{equation}
where $\Acom{\what{\zeta}}:\Acom{\rcirc{R}}\to \Acom{\rsaft{R}}$ is the induced functor.
 Conversely, consider the functor
\begin{equation}\label{Eq:Lp}
\chi':= \Acom{\what{\zeta}}\circ \chi : \Amod{R} \longrightarrow \Acom{\rsaft{R}}.
\end{equation}
Notice that, since $\chi$ is a right inverse of $\cL$, we have
\begin{equation}\label{eq:L'X'}
\cL' \circ \chi' \stackrel{\eqref{Eq:Lp}}{=} \cL' \circ \Acom{\what{\zeta}}\circ \chi \stackrel{\eqref{eq:deflL}}{=} \cL \circ \chi = \id{\Amod{R}}.
\end{equation}

\begin{proposition}\label{prop:iso}
The functors $\mathcal{L}^{\prime }$ and $\chi ^{\prime }$ establish an isomorphism between the categories $\Acom{\rsaft{R}}$ and $\Amod{R}$.
\end{proposition}

\begin{proof}
In light of \eqref{eq:L'X'}, it is enough to prove that $\chi ^{{\prime} }\circ \mathcal{L}^{{\prime} }=\id{\Acom{\rsaft{R}}}$. Note that, still by \eqref{eq:L'X'}, we have $\cL'\circ \chi' \circ \cL' = \cL'$. From the latter equality the thesis follows once proved that $\cL'$ is cancellable on the left. Since $\cL'$ is always faithful, it remains to prove that it is injective on objects. Observe that for every $M\in\proj{A}$, the assignment $m\tensor{A}f \mapsto \left[r\mapsto mf(r)\right]$ yields an isomorphism of right $A$-modules $M\tensor{A}\rdual{R} \to \hom{A}{R}{M}$, which in turn induces a bijection
\begin{equation}\label{eq:taubij}
\tau:\hom{A}{M}{M\tensor{A}\rdual{R}} \to \hom{A}{M\tensor{A} R}{M}; \quad \rho \mapsto \left[m\tensor{A} r\mapsto m_{0}^{(\rho)}m_{1}^{(\rho)}(r)\right]
\end{equation}
where we set $m_{0}^{(\rho)}\tensor{A}m_{1}^{(\rho)} = \rho(m)$ for every $m\in M$.
Let $\left( N,\rho \right) $ be an object in $\Acom{\rsaft{R}}$ and consider $\cL'\left( N,\rho \right) $. This is the $A$-module $N$ endowed with the $R$-action \eqref{Eq:Raction}, \ie
$$
m.r=m_{(0)}\xi(m_{(1)})(r) \stackrel{\eqref{eq:taubij}}{=} \tau\left(\left(N\tensor{A}\xi\right)\circ \rho\right)(m\tensor{A}r)
$$
for every $m\in N, r\in R$. Denote it by $\mu_{\rho}$. Now, $\cL'(N,\rho_N) = \cL'(P,\rho_P)$ if and only if $(N,\mu_{\rho_N})=(P,\mu_{\rho_P})$, if and only if $N=P$ and $\mu_{\rho_N}=\mu_{\rho_P}$. Since $N=P$, we may consider $\rho_N,\rho_P\in {\Coac_A}(N,N\tensor{A}\rsaft{R})$ and since $\tau$ is bijective and $\xi$ satisfies \ref{item:fmono2} of Lemma \ref{lemma:fmono}, the relation $\mu_{\rho_N}=\mu_{\rho_P}$ is equivalent to $\rho_N=\rho_P$.
\end{proof}

\begin{corollary}\label{coro:circbullet}
Let $R$ be an $A$-ring. Then
\begin{enumerate}[label=(\arabic*), ref=\emph{(\arabic*)}]
\item\label{item:circbullet1} If $\zeta$ is injective, then $R^{\circ}$ is the largest Galois $A$-coring inside $\rdual{R}$ with respect to the property of equation \eqref{eq:zetaprop}.
\item\label{item:circbullet2} $\rsaft{R}$ is a Galois $A$-coring as in Definition \ref{def:Galois} if and only if the map  $ \what{\zeta}: R^{\circ}  \to \rsaft{R}$ of \eqref{Eq:zetahat} is an isomorphism of $A$-corings.
\end{enumerate}
\end{corollary}

\begin{proof}
By Proposition \ref{prop:iso}, we know that $\chi'$ and $\cL'$ establish an  isomorphism of categories $\Amod{R} \cong \Acom{\rsaft{R}}$. Therefore, using the functor $\rR$ of \S\ref{ssec:Tannaka}, we have that $\rR(\chi')$ is an isomorphism.
We compute
\begin{equation}\label{eq:can'}
\Sf{can}' \circ \rR(\chi') \stackrel{\eqref{Eq:Lp}}{=} \Sf{can}' \circ \rR\left(\Acom{\what{\zeta}}\right) \circ \rR(\chi) \stackrel{(\Sf{can} \text{ nat})}{=} \what{\zeta} \circ \Sf{can}  \circ \rR(\chi)  = \what{\zeta}
\end{equation}
where $\Sf{can}$ and $\Sf{can}'$ are the canonical morphisms of $\rcirc{R}$ and $\rsaft{R}$, respectively. Concerning \ref{item:circbullet1}, given a Galois $A$-coring $C$ endowed with an injective $A$-bimodule map $f:C\to \rdual{R}$ satisfying the analogue of \eqref{eq:zetaprop}, by Remark \ref{rem:rsaft} there is a unique $A$-coring map $\what{f}:C\to\rsaft{R}$ such that $\xi\circ \what{f}=f$. Note that $\what{f}$ is necessarily injective. Consider the $A$-coring map $f':=\rR(\chi')^{-1}\circ \rR\left(\Acom{\what{f}}\right) \circ\Sf{can}_C^{-1}:C\to \rcirc{R}$. Then, by \eqref{eq:can'}, we have $\what{\zeta}\circ f' = \Sf{can}'\circ \rR\left(\Acom{\what{f}}\right) \circ\Sf{can}_C^{-1} = \what{f}$, so that $f'$ is injective and hence $R^{\circ}$ is the largest Galois $A$-coring inside $\rdual{R}$ with respect to the property of equation \eqref{eq:zetaprop}. The fact that $\rcirc{R}$ is Galois follows from \cite[Proposition 3.3.2]{LaiachiGomez} together with the observation that $\Sf{can}  \circ \rR(\chi) = \id{\rcirc{R}}$. Concerning \ref{item:circbullet2}, it follows from \eqref{eq:can'}.
\end{proof}

\begin{remark}\label{rem:fgp}
Assume that $R$ is an $A$-ring which is finitely generated and projective as a right $A$-module, then the map $\psi:\rdual{R}\tensor{A}\rdual{R}\to \rdual{\left(R\tensor{A}R\right)}$ given by $\psi(f\tensor{A}g)(r\tensor{A}r')=f(g(r)r')$ is invertible, so that we can define $\Delta:=\psi^{-1}\circ \rdual{m}$ and $\varepsilon: \rdual{R}\to A$ by $\varepsilon(f)=f(1)$. As a consequence $(\rdual{R},\Delta,\varepsilon)$ is an $A$-coring. It is easy to check that the identity map $\id{}:\rdual{R}\to \rdual{R}$ fulfills \eqref{eq:zetaprop}. By the universal property of $\rsaft{R}$, there exists a unique morphism $\what{\id{}}:\rdual{R}\to \rsaft{R}$ of $A$-corings such that $\xi\circ \what{\id{}} = \id{}$. Clearly, $\xi \circ \what{\id{}} \circ \xi = \xi \circ \id{}$, so that we get the equality of the $A$-coring maps $\what{\id{}} \circ \xi = \id{}$. Thus $\xi$ is invertible. On the other hand, we know from \cite[Corollary 3.3.5]{LaiachiGomez}, that the map $\zeta: \Circ{R} \to \rdual{R}$ of equation \eqref{Eq:zeta} is, in a natural way, an isomorphism of $A$-corings. Therefore, the natural transformation $\what{\zeta}_{-}: \Circ{(-)} \to \rsaft{(-)}$  when restricted to $A$-rings with finitely generated and projective underlying right $A$-modules, leads to a natural isomorphism.
\end{remark}

Our next aim is to give a complete characterization of when the functors $\cL$ and $\chi$ establish an isomorphism between the categories $\Acom{\rcirc{R}}$  and $\Amod{R}$, by analogy with Proposition \ref{prop:iso}.

\begin{theorem}\label{thm:drwho}
The following are equivalent for $\what{\zeta}:\rcirc{R}\to\rsaft{R}$
\begin{enumerate}[label=({\roman*}),ref=\emph{(\roman*)}]
\item\label{item:drwho1} $\Acom{\what{\zeta}}$ is an isomorphism,
\item\label{item:drwho2} $\Acom{\what{\zeta}}$ is injective on objects,
\item\label{item:drwho3} $\cL$ is an isomorphism,
\item\label{item:drwho4} $\cL$ is injective on objects,
\item\label{item:drwho5} $\Coring_A(C,\what{\zeta})$ is bijective for every coring $C$ whose $\can{C}$ is a split epimorphism of corings,
\item\label{item:drwho6} $\Coring_A(C,\what{\zeta})$ is injective for every coring $C$ whose $\can{C}$ is a split epimorphism of corings.
\end{enumerate}
\end{theorem}

\begin{remark}\label{rem:drwho}
Observe that $\zeta$ injective (respectively monomorphism of corings) implies that $\what{\zeta}$ is injective (respectively monomorphism of corings), which in turn implies \ref{item:drwho6} of Theorem \ref{thm:drwho}.
\end{remark}

\begin{proof}[Proof of Theorem \ref{thm:drwho}]
The equivalences $\ref{item:drwho1} \iff \ref{item:drwho3}$ and $\ref{item:drwho2}\iff \ref{item:drwho4}$ follows immediately from \eqref{eq:deflL} and Proposition \ref{prop:iso}. Obviously, \ref{item:drwho3} implies \ref{item:drwho4}. Conversely, since $\cL$ is faithful and injective on objects, the implication follows from $\cL\circ\chi\circ\cL=\cL$, which in turn follows from $\cL\circ\chi=\id{\Amod{R}}$.

Moreover, notice that $\Acom{\what{\zeta}}$ is injective on objects is equivalent to \ref{item:fmono2} of Lemma \ref{lemma:fmono} for $f=\what{\zeta}$, which in turn is equivalent to \ref{item:fmono3} of the same lemma, that is to say, to \ref{item:drwho6}. Obviously \ref{item:drwho5} implies \ref{item:drwho6}. Conversely, assume \ref{item:drwho6} and pick a $g\in \Coring_A(C,\rsaft{R})$. It induces a functor $\Acom{g}:\Acom{C}\to \Acom{\rsaft{R}}$ and a diagram with commuting squares
$$
\xymatrix@C=35pt{
{\rR(\Acom{C})} \ar[r]^-{\rR(\Acom{g})} \ar[d]_-{\can{C}} & \rR(\Acom{\rsaft{R}}) \ar[d]_-{\can{}'} & \rR(\Acom{\rcirc{R}}) \ar[l]_-{\rR(\Acom{\what{\zeta}})} \ar[d]^-{\can{}} \\
C \ar[r]_-{g} & \rsaft{R} & \rcirc{R} \ar[l]^-{\what{\zeta}}
}
$$
By assumption, there is a coring map $\sigma : C\to \rR(\Acom{C})$ such that $\can{C}\circ \sigma = \id{C}$. Thus we may consider the coring map
$\tilde{g} := \can{} \circ \rR(\Acom{\what{\zeta}})^{-1} \circ \rR(\Acom{g}) \circ \sigma$
which satisfies $\what{\zeta}\circ \tilde{g} = g$ by definition. Therefore, $\Coring_A(C,\what{\zeta})$ is surjective as well.
\end{proof}

\subsection{The finite dual of co-commutative Hopf algebroid via SAFT}\label{ssec:rsaft}
In this subsection we use another general construction relying on the Special Adjoint Functor Theorem (SAFT)  \S\ref{ssec:saft}, in order to construct another functor from the category of (right) co-commutative Hopf algebroids to the category of commutative ones. We also compare both functors constructed so far.

Fix a commutative algebra $A$ and consider as before the categories $\CHalgd{A}$ and $\Halgd{A}$ of co-commutative and commutative Hopf algebroids, respectively. The functor $\ldual{(-)}:\op{\coanillos{A}}\to\ring{A}$, assigning to each $A$-coring $\fk{C}$ the ring $\hom{A-}{\fk{C}}{A}$ with the convolution product \eqref{eq:convolution}, admits a left adjoint, denoted by $\rsaft{(-)}$, in view of SAFT. Consider  the following diagram
\begin{equation}\label{Eq:diagH}
\begin{gathered}
 \xymatrix @R=15pt{  \ring{A} \ar@{->}^-{\rsaft{(-)}}[rr] & &  \coanillos{A}  \\  \CHalgd{A} \ar@{^{(}->}[u] & &  \ar@{^{(}->}[u] \Halgd{A},}
 \end{gathered}
\end{equation}
where the vertical functors are the canonical forgetful ones.

 \begin{theorem}\label{thm:Teoremone}
 The functor $\rsaft{(-)}$ in diagram \eqref{Eq:diagH} induces a contravariant functor
$$
\rsaft{(-)}:\CHalgd{A} \longrightarrow \Halgd{A}.
$$
Explicitly, given a cocommutative Hopf algebroid $(A,\cU)$ and the canonical $A$-bilinear map $\xi:\rsaft{\cU}\to \rdual{\cU}$, the structure of commutative Hopf algebroid of $\rsaft{\cU}$ is uniquely determined by the following relations
 \begin{equation}\label{eq:algddag}
 (\xi\circ\eta)(a\otimes b)(u)=\varepsilon(bu)a,\quad \xi(xy)(u)=\xi(x)(u_1)\xi(y)(u_2), \quad
 \xi\big(\cS(x)\big)(u)=\varepsilon\big(\xi(x)(u_{\Sscript{-}})u_{\Sscript{+}}\big).
 \end{equation}
The datum $(A,\rsaft{\cU},\xi)$ fulfils the following universal property. Let $(A,\cH)$ be a commutative Hopf algebroid and $f:\cH\to \rdual{\cU}$ an $A\tensor{}A$-algebra map satisfying \eqref{eq:zetaprop}, where the $A\tensor{}A$-algebra structure of $\rdual{\cU}$ is given by the convolution product and the unit is $a\otimes b\mapsto \left[u\mapsto \varepsilon(bu)a\right]$. Then the unique map $\what{f}:\cH\to \rsaft{\cU}$ given by the universal property of $\rsaft{\cU}$ as a coring becomes a morphism of commutative Hopf algebroids.
 \end{theorem}

\begin{proof}
Set $A^e:=A\otimes A$, the  enveloping algebra, which we consider as a commutative Hopf algebroid with base algebra $A$. By Remark \ref{rem:rsaft}, the map $f:A^e\to \rdual{\cU}$, given by the assignment $a\otimes b\mapsto \left[u\mapsto \varepsilon(bu)a\right]$,  yields a unique $A$-coring map $\eta:A^e\to \rsaft{\cU}$ such that $\xi\circ \eta = f$ (recall that the $A$-coring structure on $A^e$ is the one given in Example \ref{exm:Halgd}).
\begin{invisible}
Taken elements $a, b, a', b' \in A$ and $u,v \in \cU$, we compute
\begin{align*}
f((a\otimes b)_1)\big(f((a\otimes b)_2)(u)v\big) &  = f(a\otimes 1)\big(f(1\otimes b)(u)v\big) = f(a\otimes 1)(\varepsilon(bu)v) = \varepsilon(\varepsilon(bu)v)a \\
 & = \varepsilon((bu)v)a = \varepsilon(b(uv))a  = f(a\otimes b)(uv), \\
f(a\otimes b)(1_{\Sscript{\cH}}) & = \varepsilon(b1_{\Sscript{\cH}})a = \varepsilon(1_{\Sscript{\cH}}b)a = ab = \varepsilon(a\otimes b), \\
f(a'(a\otimes b)b')(u) & = f(a'a\otimes b'b)(u) = \varepsilon(b'bu)a'a= a'\varepsilon(bb'u)a=a'f(a\otimes b)(b'u).
\end{align*}
From this it follows, by the universal property of $\rsaft{\cU}$, that there is a unique $A$-coring map $\eta:A^e\to \rsaft{\cU}$ such that $\xi\circ \eta = f$ (recall that the $A$-coring structure on $A^e$ is the one given in Example \ref{exm:Halgd}).
\end{invisible}
To introduce the multiplication, we resort to the operation $\odot$ recalled in Remark \ref{rem:circdot} in the case $C=D=\rsaft{\cU}$. Then by Remark \ref{rem:rsaft} again, the map $h:\rsaft{\cU}\odot  \rsaft{\cU}\to \cU^*$, given by $h(x\odot  y)(u)=\xi(x)(u_1)\xi(y)(u_2)$ for all $x,y\in \rsaft{\cU}$ and $u\in \cU$, gives rise to a unique $A$-coring map $m:\rsaft{\cU}\odot \rsaft{\cU}\to \rsaft{\cU}$ such that $\xi\circ m = h$.

\begin{invisible}
which is well-defined by the following argument
\begin{equation*}
\xi(axb)(u_1)\xi(y)(u_2) \stackrel{\eqref{eq:zetaprop}}{=} a\xi(x)(bu_1)\xi(y)(u_2) \stackrel{\eqref{takeuchi}}{=}\xi(x)(u_1)a\xi(y)(bu_2) \stackrel{\eqref{eq:zetaprop}}{=}  \xi(x)(u_1)\xi(ayb)(u_2).
\end{equation*}
Let us check that $h$ satisfies the conditions \eqref{eq:zetaprop}. Compute
\begin{eqnarray*}
h((x\odot y)_1)\big(h((x\odot y)_2)(u)v\big) & =&  h(x_1\odot y_1)\big(h(x_2\odot y_2)(u)v\big) \\ &=&  h(x_1\odot y_1)\big(\xi(x_2)(u_1)\xi(y_2)(u_2)v\big)\\
 & =& \xi(x_1)\big(v_1\big)\xi(y_1)\big(\xi(x_2)(u_1)\xi(y_2)(u_2)v_2\big) \\
 & \stackrel{\eqref{takeuchi}}{=}&  \xi(x_1)\big(\xi(x_2)(u_1)v_1\big)\xi(y_1)\big(\xi(y_2)(u_2)v_2\big) \\
  & \stackrel{\eqref{eq:zetaprop}}{=}&  \xi(x)(u_1v_1)\xi(y)(u_2v_2) =h(x\odot y) (uv),
\end{eqnarray*}
\begin{eqnarray*}
h(a(x\odot y)b)(u) & =& h(x\odot ayb)(u)= h(axb\odot y)(u) \\ &=& \xi(axb)(u_1)\xi(y)(u_2)
  \\ &\stackrel{\eqref{eq:zetaprop}}{=}&  a\xi(x)(bu_1)\xi(y)(u_2) \\ &=& a\xi(x)((bu)_1)\xi(y)((bu)_2)\\
   & = &a h(x\odot y)(bu)
\end{eqnarray*}
and
\begin{align*}
  h(x\odot y)(1_\cU) & = \xi(x)(1_\cU)\xi(y)(1_\cU) \stackrel{\eqref{eq:zetaprop}}{=} \varepsilon(x)\varepsilon(y) = \varepsilon(x\odot y).
\end{align*}
Therefore, by the universal property of $\rsaft{\cU}$ again, there is a unique $A$-coring map $m:\rsaft{\cU}\odot \rsaft{\cU}\to \rsaft{\cU}$ such that $\xi\circ m = h$.
\end{invisible}

Consider now the map $\lambda:\rdual{\cU}\to \rdual{\cU}$ defined by $\lambda(\alpha)(u)=\varepsilon\left(\alpha(u_{\Sscript{-}})u_{\Sscript{+}}\right)$. Then the map $f:=\lambda\circ \xi:\rsaft{\cU}\to \rdual{\cU}$, regarded as a morphism from $\cop{\rsaft{\cU}}$ to $\rdual{\cU}$ (see Remark \ref{rem:cop}), induces by Remark \ref{rem:rsaft} a unique $A$-coring map $\cS:\cop{\rsaft{\cU}}\to \rsaft{\cU}$ such that $\xi\circ \cS=f=\lambda \circ \xi$.

\begin{invisible}
We compute
\begin{align*}
f(x_2)\left(f(x_1)(u)v\right) & =\lambda(\xi(x_2))\left(\lambda(\xi(x_1))(u)v\right) = \varepsilon\Big[\xi(x_2)\big[\left(\lambda(\xi(x_1))(u)v\right)_{\Sscript{-}}\big]\big(\lambda(\xi(x_1))(u)v\big)_{\Sscript{+}}\Big] \\
 & \stackrel{\eqref{form:beta1}}{=}\varepsilon\Big[\xi(x_2)\left(v_{\Sscript{-}}\right)\lambda(\xi(x_1))(u)v_{\Sscript{+}}\Big]= \varepsilon\Big[\xi(x_2)\left(v_{\Sscript{-}}\right)\varepsilon\left(\xi(x_1)(u_{\Sscript{-}})u_{\Sscript{+}}\right)v_{\Sscript{+}}\Big] \\
  & =\varepsilon\Big[\varepsilon\left(\xi(x_1)(u_{\Sscript{-}})u_{\Sscript{+}}\right)\xi(x_2)\left(v_{\Sscript{-}}\right)v_{\Sscript{+}}\Big] = \varepsilon\Big[\xi(x_1)(u_{\Sscript{-}})u_{\Sscript{+}}\xi(x_2)\left(v_{\Sscript{-}}\right)v_{\Sscript{+}}\Big] \\
   & \stackrel{\eqref{form:beta2}}{=} \varepsilon\Big[\xi(x_1)\Big(\xi(x_2)\left(v_{\Sscript{-}}\right)u_{\Sscript{-}}\Big)u_{\Sscript{+}}v_{\Sscript{+}}\Big] \stackrel{\eqref{eq:zetaprop}}{=} \varepsilon\Big[\xi(x)\Big(v_{\Sscript{-}}u_{\Sscript{-}}\Big)u_{\Sscript{+}}v_{\Sscript{+}}\Big] \\
    & \stackrel{\eqref{form:beta3}}{=} \varepsilon\Big[\xi(x)\Big((uv)_{\Sscript{-}}\Big)(uv)_{\Sscript{+}}\Big] = f\left(x\right)(uv),\\
f(x)(1_\cU) & = \varepsilon\left(\xi(x)((1_\cU)_{\Sscript{-}})(1_\cU)_{\Sscript{+}}\right) \stackrel{\eqref{form:beta4}}{=} \varepsilon\left(\xi(x)(1_{\cU})1_{\cU}\right) \stackrel{\eqref{eq:zetaprop}}{=} \varepsilon\left(\varepsilon(x)1_{\cU}\right)=\varepsilon(x), \\
f(bxa)(u) & = \varepsilon\left(\xi(bxa)(u_{\Sscript{-}})u_{\Sscript{+}}\right) \stackrel{\eqref{eq:zetaprop}}{=} \varepsilon\left(b\xi(x)(au_{\Sscript{-}})u_{\Sscript{+}}\right) \stackrel{\eqref{form:beta2}}{=} \varepsilon\left(b\xi(x)(u_{\Sscript{-}})u_{\Sscript{+}}a\right) \\
 & = \varepsilon\left(\xi(x)(u_{\Sscript{-}})bu_{\Sscript{+}}\right)a \stackrel{\eqref{form:beta1}}{=} a\varepsilon\left(\xi(x)((bu)_{\Sscript{-}})(bu)_{\Sscript{+}}\right) = af(x)(bu).
\end{align*}
The foregoing computations show that if we regard $f$ as a morphism from $\cop{\rsaft{\cU}}$ to $\rdual{\cU}$, then it satisfies the conditions of \eqref{eq:zetaprop} and hence by the universal property of $\rsaft{\cU}$ there exists a unique $A$-coring map $\cS:\cop{\rsaft{\cU}}\to \rsaft{\cU}$ such that $\xi\circ \cS=f=\lambda \circ \xi$.
\end{invisible}

So far, we have defined a map $\eta$, a multiplication $m$ and a map $\cS$ satisfying the relations in \eqref{eq:algddag}. Let us check that these maps convert $\rsaft{\cU}$ into a commutative Hopf algebroid.
\begin{invisible}
Notice that the cocommutativity of $\cU$ forces $\xi(xy)(u)=\xi(yx)(u)$ for every $u\in\cU$ and $x,y\in\rsaft{\cU}$ by Equation \eqref{eq:algddag}. Thus $\xi(xy)=\xi(yx)$. Now, it is easy to check that $\cJ=\mathsf{Span}_{A,A}\left\{xy-yx\mid x,y\in\rsaft{\cU}\right\}$ is a coideal of $\rsaft{\cU}$ because $(\pi \tensor{A}\pi)\Delta(\cJ)=0$ and $\varepsilon(\cJ)=0$, where $\pi:\rsaft{\cU}\to \rsaft{\cU}/\cJ$ denotes the canonical projection on the quotient. On the other hand, we have just checked that $\cJ\subseteq \ker{\xi}$, so that $\cJ=0$ in view of Lemma \ref{lem:maggico}. This proves that $\rsaft{\cU}$ is commutative.
\end{invisible}
One proves that, for $a,b\in A$ and $x,y,z\in\rsaft{\cU}$, the elements of the form
$$
xy-yx, \enskip \eta(a\otimes b)x-axb, \enskip (xy)z-x(yz), \enskip \cS(xy)-\cS(y)\cS(x), \enskip \cS\left(1_{\rsaft{\cU}}\right)- 1_{\rsaft{\cU}}, \enskip \cS^2(x)-x, \enskip \cS(x_1)x_2-\eta(1\otimes \varepsilon(x))
$$
span an $A$-bimodule $\cJ$ which is a coideal of $\rsaft{\cU}$ because $(\pi \tensor{A}\pi)\Delta(\cJ)=0$ and $\varepsilon(\cJ)=0$, where $\pi:\rsaft{\cU}\to \rsaft{\cU}/\cJ$ denotes the canonical projection on the quotient. Moreover, it is contained in $\ker{\xi}$ so that $\cJ=0$ in view of Lemma \ref{lem:maggico}. This proves that all the elements displayed above vanish in $\rsaft{\cU}$.
As a consequence, we get in addition that
\begin{itemize}
\item $\rsaft{\cU}$ is commutative.
\item the $A$-coring structure of $\rsaft{\cU}$ is the one induced by $\eta$. Furthermore, we deduce that $\eta(a\otimes b)=a1_{\rsaft{\cU}} b$ where $1_{\rsaft{\cU}}:=\eta(1\otimes 1)$. Thus it follows easily that $\eta(a\otimes b)\eta(a'\otimes b')=\eta(aa'\otimes bb')$ for every $a,b\in A$. Note also that $1_{\rsaft{\cU}}x=\eta(1\otimes 1)x=x$, so that $m$ is unital.
\item $\Delta$ and $\varepsilon$ are morphisms of algebras, since both $m$ and $\eta$ are morphisms of $A$-corings.
\item $s,t$ are algebra maps as $\eta$ is.
\end{itemize}
\begin{invisible}
For every $u\in\cU$, $a,b\in A$ and $x\in\rsaft{\cU}$ we compute
\begin{align*}
\xi\big(\eta(a\otimes b)x\big)(u) & \stackrel{\eqref{eq:algddag}}{=} \xi(\eta(a\otimes b))(u_1)\xi(x)(u_2) \stackrel{\eqref{eq:algddag}}{=} \varepsilon(bu_1)a\xi(x)(u_2) \stackrel{\eqref{takeuchi}}{=} \varepsilon(u_1)a\xi(x)(bu_2) \\
 & \stackrel{\eqref{eq:zetaprop}}{=} \xi(axb)(u_2)\varepsilon(u_1)= \xi(axb)(u_2\varepsilon(u_1)) = \xi(axb)(u),
\end{align*}
so that $\alpha(a,b,x):=\eta(a\otimes b)x-axb\in\ker{\xi}$. Since $m$ is an $A$-coring map, we have that
\begin{align*}
\Delta(\alpha(a,b,x)) & = \Delta(\eta(a\otimes b))\Delta(x)-a\Delta(x)b=\eta(a\otimes 1)x_1\tensor{A}\eta(1\otimes b)x_2-ax_1\tensor{A}x_2b \\
 & = \alpha(a,1,x_1)\otimes \eta(1\otimes b)x_2+ax_1\otimes \alpha(1,b,x_2), \\
\varepsilon(\alpha(a,b,x)) & = \varepsilon(\eta(a\otimes b))\varepsilon(x)-a\varepsilon(x)b=ab\varepsilon(x)- a\varepsilon(x)b=0.
\end{align*}
As a consequence, the $\alpha(a,b,x)$ for $a,b\in A$ and $x\in\rsaft{\cU}$ span an $A$-bimodule $\cJ$ which is a coideal of $\rsaft{\cU}$ and it is contained in $\ker{\xi}$. Again by Lemma \ref{lem:maggico} $\cJ=0$ and hence $\eta(a\otimes b)x=axb$.

This means that the $A$-coring structure of $\rsaft{\cU}$ is the one induced by $\eta$. Furthermore, we deduce that $\eta(a\otimes b)=a1_{\rsaft{\cU}} b$ where $1_{\rsaft{\cU}}:=\eta(1\otimes 1)$. Thus it follows easily that $\eta(a\otimes b)\eta(a'\otimes b')=\eta(aa'\otimes bb')$ for every $a,b\in A$. Note also that $1_{\rsaft{\cU}}x=\eta(1\otimes 1)x=x$, so that $m$ is unital.

To show that $m$ is also associative, notice that the elements $(xy)z-x(yz)$ for $x,y,z\in\rsaft{\cU}$ span an $A$-bimodule $\cJ$ which is a coideal of $\rsaft{\cU}$ as well. Moreover, this is contained in $\ker{\xi}$ since for all $u\in\cU$
\begin{align*}
\xi((xy)z)(u) & \stackrel{\eqref{eq:algddag}}{=}\xi(xy)(u_1)\xi(z)(u_2)\stackrel{\eqref{eq:algddag}}{=}\xi(x)\left((u_1)_1\right)\xi(y)\left((u_1)_2\right)\xi(z)(u_2)=\xi(x)\left(u_1\right)\xi(y)\left(u_2\right)\xi(z)\left(u_3\right)
\end{align*}
and the same result is obtained by computing $\xi(x(yz))(u)$.

Since both $m$ and $\eta$ are morphisms of $A$-corings, it follows that $\Delta$ and $\varepsilon$ are morphisms of algebras. Moreover, $s,t$ are algebra maps as $\eta$ is.

To conclude, let us show that $\cS$ satisfies the conditions to be an antipode. Again, notice that the elements $\cS(xy)-\cS(x)\cS(y)$ for $x,y\in\rsaft{\cU}$ span an $A$-bimodule $\cJ$ which is a coideal of $\rsaft{\cU}$. Moreover, this is contained in $\ker{\xi}$ since for all $u\in\cU$
\begin{align*}
\xi\left(\cS(xy)\right)(u) & \stackrel{\eqref{eq:algddag}}{=} \varepsilon\Big(\xi(xy)(u_{\Sscript{-}})u_{\Sscript{+}}\Big)\stackrel{\eqref{eq:algddag}}{=} \varepsilon\Big(\xi(x)\left((u_{\Sscript{-}})_1\right)\xi(y)\left((u_{\Sscript{-}})_2\right)u_{\Sscript{+}}\Big) \\
 & \stackrel{\eqref{form:beta5}}{=} \varepsilon \Big(\xi(x)\left(\big(u_{\Sscript{+}}\big)_{\Sscript{-}}\right)\xi(y)\left(u_{\Sscript{-}}\right)\big(u_{\Sscript{+}}\big)_{\Sscript{+}}\Big) \stackrel{\eqref{form:beta1}}{=} \varepsilon \Big(\xi(x)\left(\big(\xi(y)\left(u_{\Sscript{-}}\right)u_{\Sscript{+}}\big)_{\Sscript{-}}\right)\big(\xi(y)\left(u_{\Sscript{-}}\right)u_{\Sscript{+}}\big)_{\Sscript{+}}\Big) \\
 & = \xi(\cS(x))\Big(\xi(y)\left(u_{\Sscript{-}}\right)u_{\Sscript{+}}\Big) = \xi(\cS(x))\Big(\xi(y)\left(u_{\Sscript{-}}\right)\left(u_{\Sscript{+}}\right)_2\varepsilon\left(\left(u_{\Sscript{+}}\right)_1\right)\Big) \\
  & \stackrel{\eqref{takeuchi}}{=} \xi(\cS(x))\Big(\left(u_{\Sscript{+}}\right)_2\varepsilon\left(\xi(y)\left(u_{\Sscript{-}}\right)\left(u_{\Sscript{+}}\right)_1\right)\Big) \stackrel{\eqref{form:beta6}}{=} \xi(\cS(x))\Big(u_2\varepsilon\left(\xi(y)\left(u_1\right)_{\Sscript{-}}\left(u_1\right)_{\Sscript{+}}\right)\Big) \\
   & \stackrel{\eqref{eq:algddag}}{=} \xi(\cS(x))\left(u_2\xi(\cS(y))(u_1)\right) = \xi(\cS(y))(u_1)\xi(\cS(x))\left(u_2\right)  \stackrel{\eqref{eq:algddag}}{=} \xi(\cS(y)\cS(x))(u),
\end{align*}
Thus $\cS$ is multiplicative. Furthermore,
\begin{align*}
\xi\left(\cS\left(1_{\rsaft{\cU}}\right)\right)(u) & \stackrel{\eqref{eq:algddag}}{=} \varepsilon\left(\xi\left(1_{\rsaft{\cU}}\right)\left(u_{\Sscript{-}}\right)u_{\Sscript{+}}\right) \stackrel{\eqref{eq:algddag}}{=} \varepsilon\left(\xi\left(\eta(1\otimes 1)\right)\left(u_{\Sscript{-}}\right)u_{\Sscript{+}}\right) \stackrel{\eqref{eq:algddag}}{=} \varepsilon\left(\varepsilon\left(u_{\Sscript{-}}\right)u_{\Sscript{+}}\right) = \varepsilon\left(u_{\Sscript{-}}u_{\Sscript{+}}\right) \stackrel{\eqref{form:beta7}}{=} \varepsilon\left(\varepsilon(u)1_\cU\right)\\
 & = \varepsilon(u)=\xi\left(\eta(1\otimes 1)\right)(u) = \xi\left(1_{\rsaft{\cU}}\right)(u)
\end{align*}
for all $u\in\cU$. As above we deduce that $\cS\left(1_{\rsaft{\cU}}\right)- 1_{\rsaft{\cU}}$ generates a coideal and so it vanishes. Thus, $\cS$ is a morphism of algebras.

We compute further
\begin{align*}
\xi\left(\cS^2(x)\right)(u) & \stackrel{\eqref{eq:algddag}}{=} \varepsilon\Big(\xi\left(\cS(x)\right)\left(u_{\Sscript{-}}\right)u_{\Sscript{+}}\Big) \stackrel{\eqref{eq:algddag}}{=} \varepsilon\Big(\varepsilon\Big(\xi\left(x\right)\big(\left(u_{\Sscript{-}}\right)_{\Sscript{-}}\big)(u_{\Sscript{-}})_{\Sscript{+}}\Big)u_{\Sscript{+}}\Big) = \varepsilon\Big(\xi\left(x\right)\big(\left(u_{\Sscript{-}}\right)_{\Sscript{-}}\big)(u_{\Sscript{-}})_{\Sscript{+}}u_{\Sscript{+}}\Big) \\
 & \stackrel{\eqref{form:beta8}}{=}   \varepsilon\Big(\xi\left(x\right)\left(u\right)1_\cU \Big) = \xi\left(x\right)\left(u\right),
\end{align*}
for all $u\in\cU$ and $x\in\rsaft{\cU}$, so that $\xi\circ \cS^2 = \xi$ and since both $\cS^2$ and $\id{\rsaft{\cU}}$ are coring maps, they are equal.
\end{invisible}
The compatibility of $\cS$ with $s$ and $t$ follows from
\begin{equation*}
\cS(\eta(a\otimes b))=\cS(a1_{\rsaft{\cU}}b)\stackrel{(*)}{=}b\cS(1_{\rsaft{\cU}})a = b1_{\rsaft{\cU}}a = \eta(b\otimes a),
\end{equation*}
where in $(*)$ we used that $\cS:\cop{\rsaft{\cU}}\to \rsaft{\cU}$.
\begin{invisible}
Finally, we compute
\begin{align*}
\xi\left(\cS(x_1)x_2\right)(u) & \stackrel{\eqref{eq:algddag}}{=} \xi\left(\cS(x_1)\right)(u_1)\xi\left(x_2\right)(u_2) \stackrel{\eqref{eq:algddag}}{=} \varepsilon\Big(\xi(x_1)\left((u_1)_{\Sscript{-}}\right)(u_1)_{\Sscript{+}}\Big)\xi\left(x_2\right)(u_2)
= \varepsilon\Big(\xi(x_1)\left((u_1)_{\Sscript{-}}\right)(u_1)_{\Sscript{+}}\xi\left(x_2\right)(u_2)\Big)\\
& \stackrel{\eqref{form:beta2}}{=} \varepsilon\Big(\xi(x_1)\Big(\xi\left(x_2\right)(u_2)(u_1)_{\Sscript{-}}\Big)(u_1)_{\Sscript{+}}\Big) \stackrel{\eqref{eq:zetaprop}}{=} \varepsilon\Big(\xi(x)\Big(u_2(u_1)_{\Sscript{-}}\Big)(u_1)_{\Sscript{+}}\Big) \stackrel{(\blacktriangle)}{=} \varepsilon\Big(\xi(x)\Big(u_1(u_2)_{\Sscript{-}}\Big)(u_2)_{\Sscript{+}}\Big)  \\
 & \stackrel{\eqref{form:beta9}}{=} \varepsilon\Big(\xi(x)\left(1_\cU\right)u\Big) \stackrel{\eqref{eq:zetaprop}}{=} \varepsilon\left(\varepsilon(x)u\right) \stackrel{\eqref{eq:algddag}}{=} \xi\left(\eta(1\otimes \varepsilon(x))\right)(u)
\end{align*}
where in $(\blacktriangle)$ we used the co-commutativity of $\cU$. As above we deduce that the elements $\cS(x_1)x_2-\eta(1\otimes \varepsilon(x))$ for $x\in\rsaft{\cU}$ generate a coideal and so they vanishes.
\end{invisible}
Summing up $(A,\rsaft{\cU})$ is an object in $\Halgd{A}$.

Now let us check $\rsaft{(-)}$ is compatible with the morphisms. To this aim, let $(A,\cH)$ be a commutative Hopf algebroid and $f:\cH\to \rdual{\cU}$ an $A^e$-algebra map satisfying \eqref{eq:zetaprop}. The universal property of $\rsaft{\cU}$ yields a unique map $\what{f}:{}_{\scriptstyle{s}}\cH_{\scriptstyle{t}}\to \rsaft{\cU}$ of $A$-corings such that $\xi\circ \what{f} = f$. By the trick we used above, the elements $\what{f}(1_\cH)-1_{\rsaft{\cU}}$ and $\what{f}(x)\what{f}(y)-\what{f}(xy)$ for $x,y\in\cH$ vanish in $\rsaft{\cU}$ because they generate a coideal $\cJ$ which is contained in $\ker{\xi}$. We just point out that by $A$-bilinearity of the involved maps, $\xi\circ \what{f}\circ \eta_\cH = \xi \circ \eta_{\rsaft{\cU}}$ and hence the equality of the coring maps $\what{f}\circ \eta_\cH = \eta_{\rsaft{\cU}}$.
\begin{invisible}
\begin{align*}
\xi\left(\what{f}(1_\cH)\right)(u) & = \left(\xi\circ \what{f}\,\right)(1_{\cH})(u) = f\left(1_{\cH}\right)(u) = \varepsilon(u) \stackrel{\eqref{eq:algddag}}{=} \xi\left(1_{\rsaft{\cU}}\right)(u)
\end{align*}
for all $u\in\cU$ it follows that $\xi\left(\what{f}(1_\cH)\right)=\xi\left(1_{\rsaft{\cU}}\right)$. By $A$-bilinearity of the involved maps, we may claim that $\xi\circ \what{f}\circ \eta_\cH = \xi \circ \eta_{\rsaft{\cU}}$ and hence the equality of the coring maps $\what{f}\circ \eta_\cH = \eta_{\rsaft{\cU}}$. Furthermore, for all $u\in \cU$, we compute
\begin{align*}
\xi(\what{f}(x)\what{f}(y))(u) \stackrel{\eqref{eq:algddag}}{=} \xi(\what{f}(x))(u_1)\xi(\what{f}(y))(u_2) = f(x)(u_1)f(y)(u_2) = f(xy)(u) =\xi(\what{f}(xy))(u)
\end{align*}
so that $\what{f}(x)\what{f}(y)-\what{f}(xy)$ generates a coideal $\cJ$ which is contained in $\ker{\xi}$ whence $\what{f}$ is multiplicative.
\end{invisible}
Summing up, we showed that $\left(\,\id,\what{f}\,\right)$ is a morphism of commutative bialgebroids. Since the compatibility with the antipodes comes for free, we conclude that $\what{f}$ is a morphism of commutative Hopf algebroids.

Let $(\id{A},\phi):(A,\cU_1)\to (A,\cU_2)$ be a morphism of cocommutative Hopf algebroids. Apply the previous construction to $f=\rdual{\phi}\circ \xi_2$, once observed that $\xi_2$ is a morphism of $A$-rings in view of \eqref{eq:algddag}, that $\rdual{\phi}$ is so as well by a direct computation and that $f$ satisfies \eqref{eq:zetaprop} (see also Remark \ref{rem:univpropbullet}). As a consequence $\what{f}$, which is $\rsaft{\phi}$ by definition, becomes a morphism of commutative Hopf algebroids. This leads to the stated functor and finishes the proof.
\end{proof}

\begin{remark}\label{rem:UcUb}
Let $(A,\cU)$ be a co-commutative Hopf algebroid over $\Bbbk$ and consider both duals $(A,\Circ{\cU})$ and $(A,\rsaft{\cU})$ as commutative Hopf algebroids over $\Bbbk$.
\begin{enumerate}[label=(\arabic*),ref=(\arabic*)]
\item\label{item:UcBc1} In view of the second claim in Theorem \ref{thm:Teoremone} and of \eqref{eq:zetaprop}, the canonical map $\zeta:\cU^\bcirc\to \rdual{\cU}$ given in \eqref{Eq:zetaH} induces a unique morphism of commutative Hopf algebroids $\what{\zeta}:\cU^\bcirc\to\rsaft{\cU}$ such that $\xi\circ \what{\zeta} = \zeta$. If $A$  is a field, then $\what{\zeta}$ is an isomorphism of Hopf algebras.
\item\label{item:UcBc2}  If the underlying right $A$-module of $\cU$ is finitely generated and projective, then the map $\what{\zeta}$ induces an isomorphism of commutative Hopf algebroids $(A,\Circ{\cU})$ and $(A,\rsaft{\cU})$ (see also Remark \ref{rem:fgp}).
\item\label{item:UcBc3} Consider an $A$-ring $R$ and the map $\psi:\rdual{R}\tensor{A}\rdual{R}\to \rdual{\left(R\tensor{A}R\right)}$ given by $\psi(f\tensor{A}g)(r\tensor{A}r')=f(g(r)r')$. Given an $A$-coring $C$ and an $A$-bimodule map $f:C\to \rdual{R}$ satisfying \eqref{eq:zetaprop}, for every $x\in\ker{f}$ and for all $r,r'\in R$ we have
$$
0 = f(x)(rr') \stackrel{\eqref{eq:zetaprop}}{=} \psi\left(f(x_1)\tensor{A}f(x_2)\right)(r\tensor{A}r'), \qquad 0 = f(x)(1_R) \stackrel{\eqref{eq:zetaprop}}{=} \varepsilon(x).
$$
Thus $\varepsilon\left(\ker{f}\right)=0$. Moreover, if we assume $\psi$ injective, we also have $(f\tensor{A}f)\left(\Delta\left(\ker{f}\right)\right)=0$. These two equalities are very close but not sufficient to claim that $\ker{f}$ is a coideal of $C$. This would be useful in case $C=\rsaft{R}$ and $f=\xi$ to deduce that $\xi$ is injective by Lemma \ref{lem:maggico}.

The fact that $\ker{f}$ is a coideal of $C$ can be obtained under a further assumption as follows. Write $f$ as $\tilde{f}\circ \pi$ where $\pi:C\to C/\ker{f}$ is the canonical projection and $\tilde{f}:C/\ker{f}\to \rdual{R}$ is the obvious induced map. Then $(\tilde{f}\tensor{A}\tilde{f})(\pi\tensor{A}\pi)\left(\Delta\left(\ker{f}\right)\right)=0$. Thus, if we assume that $\tilde{f}\tensor{A}\tilde{f}$ is injective, we can conclude that $(\pi\tensor{A}\pi)\left(\Delta\left(\ker{f}\right)\right)=0$.
\end{enumerate}
\end{remark}

We finish this section by the following useful lemma.

\begin{lemma}\label{lema:data}
Let $(A,\cH)$ be a commutative Hopf algebroid and $(A,\cU)$ a co-commutative one. Then, there is a bijective correspondence between the following sets of data:
\begin{enumerate}[label={\alph*}),ref=\emph{\alph*})]
\item\label{item:6.4a} morphisms $\what{f}:\cH \to \rsaft{\cU}$ of commutative Hopf algebroids;
\item\label{item:6.4b} morphisms $f:\cH\to \rdual{\cU}$ of $A^e$-algebras satisfying \eqref{eq:zetaprop};
\item\label{item:6.4c} morphisms $h:\cU\to \ldual{\cH}$ of $A$-rings satisfying for all $a,b\in A$, $u\in \cU$ and $x,y\in\cH$
\begin{equation}\label{eq:simildag}
h(u)(\eta(a\otimes b))=\varepsilon(bu)a \qquad \text{and} \qquad h(u)(xy)=h(u_1)(x)h(u_2)(y).
\end{equation}
\end{enumerate}
\end{lemma}
\begin{proof}
From \ref{item:6.4b} to \ref{item:6.4a} we go by the second part of Theorem \ref{thm:Teoremone}. From \ref{item:6.4a} to \ref{item:6.4b} we compose $\what{f}$ with the canonical map $\xi$ which by \eqref{eq:algddag} is a morphism of $A^e$-algebras and satisfies \eqref{eq:zetaprop}. The correspondence is bijective because of the universal property of $\rsaft{\cU}$. A direct computation shows that we have a correspondence between $A$-bimodule maps $f:\cH\to \rdual{\cU}$ satisfying \eqref{eq:zetaprop} and $A$-ring maps $h:\cU\to \ldual{\cH}$ (see Remark \ref{rem:univpropbullet}). Thus the correspondence between \ref{item:6.4b} and \ref{item:6.4c} is given by $h(u)(x)  = f(x)(u)$ for all $x,y\in\cH$ and $u\in\cU$, since relations \eqref{eq:simildag} corresponds to $f$ being an $A^e$-algebra map.
\end{proof}

\begin{remark}\label{rem:I}
Let $(A,L)$ be a Lie-Rinehart algebra and set $(A,\cU)=(A,\VL)$ and $(A,\cH)=(A, \rsaft{\VL})$ in Lemma \ref{lema:data}.  So corresponding to the identity morphism of commutative Hopf algebroids $\id{\rsaft{\VL}}$, there is a morphism of $A$-rings $\fk{i}: \VL \to {}^{*}(\rsaft{\VL})$ satisfying relations \eqref{eq:simildag}.  On generators it is explicitly given by:
\begin{equation}\label{Eq:I}
\fk{i}: \VL \longrightarrow {}^{*}\big(\rsaft{\VL}\big), \quad \Big( \iota_{\Sscript{L}}(X) \longmapsto \Big[ z \mapsto \xi(z)(\iota_{\Sscript{L}}(X)) \Big]  \Big).
\end{equation}
\end{remark}


\section{Differentiation in Hopf algebroids framework}\label{sec:LI}

Given a commutative algebra $A$, the assignment that associates every $A$-module $M$ with the space $\Der{\K}{A}{M}$ of $\K$-linear derivations on $A$ with coefficients in $M$ gives a representable functor $\Der{\K}{A}{-}:\rmod{A}\to\mathsf{Set}$ whose representing object is the so-called module of K\"ahler differentials (or simply K\"ahler module) $\Omega_{\K}(A)$. In this section we are going to explore these facts in the Hopf algebroids framework. In addition, we will see how derivations on Hopf algebroids with coefficients in the base algebra are related with Lie-Rinehart algebras and provide for us a contravariant functor $\lL:\Halgd{A}\to\LieR{A}$, called \emph{the differential functor}. {This functor can be seen as the algebraic counterpart of the construction of Lie algebroid from a Lie groupoid. Analysing the case of split Hopf algebroids we will come across a construction described in \cite{DemazureGabriel} for affine group $\Bbbk$-scheme actions.}

\subsection{Derivations with coefficients in modules}\label{ssec: Der}
Next we fix a commutative Hopf algebroid $(A,\cH)$. All modules over $\cH$  are \emph{right} $\cH$-modules and with central action, that is, the left action is the same as the right action,  in the sense that $ m.u= u \, m$, for every $u \in \cH$ and $m \in M$ a right $\cH$-module.  Let us denote by $\rmod{\cH}$ the category of $\cH$-modules and their morphisms. When we restrict to $A$ via the unit map $\etaup$, we will denote by $\Ms$ and $\Mt$ the distinguished $A$-modules resulting from $M_{\Sscript{\cH}}$.  In particular, for $\cH_{\Sscript{\cH}}$ this means that we are considering $\Hs$ as an $A$-algebra via the source map $s$, while $\Ht$ is an $A$-algebra via the target map $t$.

\begin{definition}\label{def:Dernew}
Let $p:A\to\cH$ and $\varphi:\cH\to\cH$ be algebra morphisms and $M_{\cH}$ be an $\cH$-module. Set $M_\varphi:=\varphi_*(M)$, the $\cH$-module obtained by restriction of scalars via $\varphi$, \ie $m\cdot u=m.\varphi(u)$ for all $m\in M, u\in \cH$. It is assumed to be an $A$-module via further restriction of scalars: $m\cdot a = m.\varphi(p(a))$. We define the following right $\cH$-module
\begin{equation*}
\Der{A}{\Hp}{M_\varphi}:=\Big\{\delta \in \Hom{A}{\Hp}{M_{\varphi p}}\mid \delta(uv)=\delta(u)\cdot v+\delta(v)\cdot u=\delta(u).\varphi(v)+\delta(v).\varphi(u)\text{ for all } u,v\in\cH\Big\}
\end{equation*}
with $\cH$-action given by $(\delta v)(u)=\delta(u).v$ for all $u,v\in\cH,\delta\in \Der{A}{\Hp}{M_\varphi}$.
\end{definition}

\begin{remark}
Notice that the condition $\delta \in \Hom{A}{\Hp}{M_{\varphi p}}$ in the definition of $\Der{A}{\Hp}{M_\varphi}$ in Definition \ref{def:Dernew} means that
$$
\delta(up(a))=\delta(u).\varphi(p(a))
$$
for all $a\in A, u\in\cH$. Moreover, since the condition $\delta(uv)=\delta(u)\cdot v+\delta(v)\cdot u$ for all $u,v\in\cH$ implies that $\delta(1_{\cH})=0$, we have that $\delta(p(a))=\delta(1_{\cH}).\varphi(p(a))=0$ for all $a\in A$, whence $\delta\circ p=0$.
\end{remark}

As a matter of notation, if we have $p:A\to \cH$ an algebra map, $f:\Ht\to\Hp$ and $g:\Hs\to\Hp$ two $A$-algebra maps and $\delta:\Hs\to \Mp$ and $\lambda:\Ht\to\Mp$ two $A$-linear morphisms, then we will set
\begin{equation}\label{eq:*prods}
(f*g)(u):= f(u_1)g(u_2), \quad (f*\delta)(u):=  \delta(u_2).f(u_1) \quad \text{and} \quad (\lambda*g)(u):= \lambda(u_1).g(u_2)
\end{equation}
for every $u\in\cH$. Notice that the compatibility conditions with $A$ are needed to have that every $*$-product above is well-defined.

\begin{lemma}
Let $p,q:A\to \cH$ be algebra morphisms. Let also $\gamma:\Hq\to\Hp$, $\varphi,\, \beta:\Ht\to\Hp$ and $\psi,\, \alpha:\Hs\to\Hp$ be $A$-algebra morphisms. These induce $\cH$-module morphisms
\begin{gather}\label{eq:isosnew}
\begin{gathered}
\xymatrix@R=0pt{  \Der{A}{\Ht}{M_\varphi}  \ar@{->}[r]    &  \Der{A}{\Ht}{M_{\varphi*\psi}},  \\ \ar@{|->}^{}[r] \delta   & \delta * \psi }  \qquad
\xymatrix@R=0pt{  \Der{A}{\Hs}{M_\alpha} \ar@{->}[r]    & \Der{A}{\Hs}{M_{\beta*\alpha}},  \\ \ar@{|->}^{}[r] \delta   & \beta*\delta }  \\
\xymatrix@R=0pt{  \Der{A}{\Hp}{M_{\varphi}} \ar@{->}[r]    &  \Der{A}{\cH_{q}}{M_{\varphi\gamma}}.  \\ \ar@{|->}^{}[r] \delta   & \delta\circ\gamma }
\end{gathered}
\end{gather}
\end{lemma}

\begin{proof}
The proof is simply a matter of checking that the assignments are well-defined and $\cH$-linear. Let us do it for the upper left one and leave the others to the reader.

By definition \eqref{eq:*prods} we have that $(\delta*\psi)(u)=\delta(u_1).\psi(u_2)$ for all $u\in\cH$. For every $a\in A,u,v\in\cH$ we may compute directly
\begin{align*}
(\delta*\psi)(uv) & =\delta(u_1).\varphi(v_1)\psi(u_2)\psi(v_2)+\delta(v_1).\varphi(u_1)\psi(u_2)\psi(v_2) \\\
 & =(\delta*\psi)(u).(\varphi*\psi)(v)+(\delta*\psi)(v).(\varphi*\psi)(u), \\
(\delta*\psi)(ut(a)) & =(\delta*\psi)(u).(\varphi*\psi)(t(a))+(\delta*\psi)(t(a)).(\varphi*\psi)(u) \\
 & =(\delta*\psi)(u).(\varphi*\psi)(t(a))+\delta(1_\cH).\psi(t(a))(\varphi*\psi)(u) \\
 & =(\delta*\psi)(u).(\varphi*\psi)(t(a)), \\
\big((\delta v)*\psi\big)(u) & =(\delta v)(u_1).\psi(u_2) = \delta(u_1).v\psi(u_2)= (\delta*\psi)(u).v = \big((\delta*\psi).v\big)(u),
\end{align*}
and this concludes the required checks.
\end{proof}

\begin{remark}
Notice that the latter morphism in \eqref{eq:isosnew} is a particular instance of a more general result, claiming that for $A$-algebras $p:A\to  \cH$ and $q:A \to \cK$ every $A$-algebra morphism $\phi:\cH\to\cK$ induces a natural transformation $\phi_*\left(\Der{A}{\cK}{M}\right)\to \Der{A}{\cH}{\phi_*(M)},\,(\delta\mapsto \delta\circ\phi)$ in $\cH$-modules.
\end{remark}

\begin{corollary}\label{cor:PQ}
Let  $M$ be an $\cH$-module. For every $p, \, q \, \in \{s, t\}$,  we set:
\begin{equation}\label{eq:distDer}
 \Der{A}{\Hp}{\Mq}:=\Der{A}{\Hp}{M_{q\varepsilon}}\qquad \text{and}\qquad\Derk{p}{\cH}{M}:=\Der{A}{\Hp}{M_{\id{\cH}}}=\Der{A}{\Hp}{M}.
\end{equation}
Then we have the following  isomorphisms of $\cH$-modules
\begin{align}\label{eq:isos}
\begin{gathered}
\xymatrix@R=0pt@C=30pt{  \Derk{t}{\cH}{M} \ar@{->}^{\cong }[r]    &  \Der{A}{\Ht}{\Ms} , \\ \ar@{|->}^{}[r] \delta   & \delta*\cS \\  \gamma*\id{\cH}   & \ar@{|->}^{}[l]  \gamma }  \qquad
\xymatrix@R=0pt@C=30pt{  \Derk{s}{\cH}{M} \ar@{->}^{\cong }[r]    & \Der{A}{\Hs}{\Mt},  \\ \ar@{|->}^{}[r] \delta   & \cS*\delta  \\  \id{\cH}*\gamma  & \ar@{|->}^{}[l]  \gamma }  \\
\xymatrix@R=0pt@C=30pt{  \Der{A}{\Hp}{\Mq}  \ar@{->}^{\cong }[r]    &  \Der{A}{\Hq}{\Mq}.  \\ \ar@{|->}^{}[r] \delta   & \delta \circ \cS  \\  \gamma \circ \cS  &\ar@{|->}^{}[l] \gamma     }
\end{gathered}
\end{align}
\end{corollary}
\begin{proof}
Straightforward.
\end{proof}

Let us denote by $\cI:=\ker{\varepsilon}$ the augmentation ideal of $\cH$. For every $p\in \{s, t\}$, we have that $u-p(\varepsilon(u))\in\cI$ for all $u\in\cH$ and hence
\begin{equation}\label{eq:I/I2}
v(u-p(\varepsilon(u)))+\cI^2 = p(\varepsilon(v))(u-p(\varepsilon(u)))+\cI^2
\end{equation}
in $\cI/\cI^2$ for all $v\in\cH$. We can define the surjective  map associated to $\cI$
\begin{equation}\label{Eq:pis}
\xymatrix@R=0pt@C=40pt{  \Hp \ar@{->}^{\pip }[r]  &  \frac{\cI}{\cI^{2}}  \\ \ar@{|->}^{}[r] u   & \Big( u - p(\varepsilon(u))\, + \, \cI^{2}  \Big)  }
\end{equation}
which enjoys the following properties.

\begin{lemma}\label{lema:propPi}
Consider $\cI/\cI^{2}$ as an $\cH$-module via \eqref{eq:I/I2}. Then, for every $p\in \{s, t\}$ and $u, v \in \cH$, the map $\pip$ satisfies:
\begin{equation}\label{eq:propPi1}
\pip\circ p\,=\, 0, \qquad \pip(uv)\,=\, \pip(u) \, p(\varepsilon(v)) + p(\varepsilon(u))\,  \pip(v)
\end{equation}
In particular, $\pip \, \in \, \Derk{p}{\cH}{\cI/\cI^{2}}=\Der{A}{\Hp}{\big(  \cI/\cI^{2}\big)_{\Sscript{p}}}$. Furthermore, for every $u, v \in \cH$, we have
\begin{equation}\label{eq:propPi4}
u_{1}\tensor{A}u_{2}\, \pis(v) = u\tensor{A} \pis(v) \, \in \Hst \tensor{A} {}_{\Sscript{s}}\big(\cI/\cI^{2}\big), \qquad \pit(v) u_{1}\tensor{A}u_{2} = \pit(v)\tensor{A} u  \, \in\big(  \cI/\cI^{2}\big)_{\Sscript{t}}\tensor{A} \Hst.
\end{equation}
Moreover, the maps
\begin{equation}\label{Eq:tensorpis}
\psis: \Hst \longrightarrow \Hst\, \tensor{A} \, {}_{\Sscript{s}}\Big( \frac{\cI}{\cI^{2}} \Big), \; \Big[  u \longmapsto u_{1}\tensor{A}\pis(u_{2})\Big];\quad \psit: \Hst \longrightarrow \Big( \frac{\cI}{\cI^{2}} \Big)_{\Sscript{t}}\tensor{A} \Hst ,\;  \Big[  u \longmapsto \pit(u_{1})\tensor{A}u_{2}\Big]
\end{equation}
are well-defined left and right $A$-module morphisms, respectively.
\end{lemma}

\begin{proof}
The properties \eqref{eq:propPi1} follow easily by the definition of $\pip$. Concerning \eqref{eq:propPi4}, we have
$$u_{1}\tensor{A}u_{2}\pis(v)\overset{\eqref{eq:I/I2}}{=}u_{1}\tensor{A}s(\varepsilon(u_{2}))\pis(v)=u_{1}t(\varepsilon(u_{2}))\tensor{A}\pis(v)=u\tensor{A}\pis(v),$$
$$\pis(v)u_{1}\tensor{A}u_{2}\overset{\eqref{eq:I/I2}}{=}\pis(v)t(\varepsilon(u_{1}))\tensor{A}u_{2}=\pis(v)\tensor{A}s(\varepsilon(u_{1}))u_{2}=\pis(v)\tensor{A}v.$$
It is now clear that $\Derk{p}{\cH}{\cI/\cI^{2}}=\Der{A}{\Hp}{\big(  \cI/\cI^{2}\big)_{\Sscript{p}}}$ and that $\pip$ belongs to this set. As a consequence $\pip\in \Hom{A}{\Hp}{\big(  \cI/\cI^{2}\big)_{\Sscript{p}}}$ whence it makes sense to define $\psis:=(\Hst\tensor{A} \pis)\circ \Delta$ and $\psit:=(\pit\tensor{A} \Hst)\circ \Delta$.
\end{proof}

Now we show that, for every $p,q \in \{s,t\}$,  $\Der{A}{\Hp}{(-)_q}: \rmod{\cH} \to \rmod{\cH}$  is a kind of  a representable functor.

\begin{lemma}\label{lema:rep}
Given $p, q, r \in \{s,t\}$ with $p\neq q$ and $M$ is an $\cH$-module. Then there is a natural isomorphism
\begin{equation}\label{Eq:natural}
\xymatrix@R=0pt@C=30pt{  \Der{A}{\Hp}{\Mq}  \ar@{->}^-{\cong }[rr]  &   &  \Hom{A}{\big( \frac{\cI}{\cI^{2}} \big)_{\Sscript{p}}}{\Mq}  \\ \ar@{|->}^{}[rr] \delta   & & \overline{\delta}:=\Big[ \pip(u) \mapsto \delta(u)  \Big]  \\  f \circ \pip   & & \ar@{|->}^{}[ll]  f, }
\end{equation}
of $\cH$-modules.
\end{lemma}

\begin{proof}
First note that $\varepsilon\in \Hom{A}{\Hp}{A}$ so that it makes sense to consider the following diagram.
\begin{equation}\label{Diag:CoeqI}
\xymatrix@C=45pt{ \Hp \tensor{A} \Hp \ar@<.5ex>[rr]^-{mult}  \ar@<-.5ex>[rr]_-{\varepsilon\tensor{A}\cH\,+\,\cH\tensor{A}\varepsilon}&& \Hp \ar[rr]^-{\pip} &&\big(\frac{\cI}{\cI^{2}}\big)_{\Sscript{p}}.}
\end{equation}
Let us check that it is a coequalizer of $A$-modules. Let $N$ be an $A$-module and let $\delta \in\Hom{A}{\Hp}{N}$ such that $uv-up(\varepsilon(v)) - p( \varepsilon(u))v \in\ker{\delta}$  for every $u,v \in \cH$.
$$
\xymatrix{ 0 \ar[r] & \ker{\pip} \ar[r] & \Hp \ar[r]^-{\pip} & \big(\frac{\cI}{\cI^{2}}\big)_{p} \ar[r] & 0.}
$$
If $u\in\ker{\pip}$ (\ie $u-p(\varepsilon(u))\in \cI^2$), then $\delta(u)\stackrel{(\delta p=0)}{=}\delta(u-p(\varepsilon(u)))\in\delta(\cI^{2})\subseteq \cI p(\varepsilon(\cI))=0$ so that $\delta$ factors through a unique map $\overline{\delta}:\cI/\cI^{2}\to N $ such that $\overline{\delta}\circ\pip=\delta$.

On the other hand, by Lemma \ref{lema:propPi}, the map $\pip$ coequalizes the parallel pair in the diagram above. Thus \eqref{Diag:CoeqI} is a coequalizer as claimed. Now, for $N=\Mq$ it is clear that the maps $\delta \in\Hom{A}{\Hp}{N}$ coequalizing the parallel pair in \eqref{Diag:CoeqI} are exactly the elements in $\Der{A}{\Hp}{\Mq} $ so that they bijectively correspond to the elements in $\Hom{A}{\big( \frac{\cI}{\cI^{2}} \big)_{\Sscript{p}}}{\Mq}$ by the universal property of the coequalizer. This correspondence is clearly $\cH$-linear and natural in $M$.
\end{proof}

\subsection{The K\"ahler module of a Hopf algebroid}\label{ssec:Omega}
Next, we investigate the K\"ahler module of $\cH$ and construct the universal derivation. {The linear dual of this module with values in the base algebra, will have a structure of Lie-Rinehart algebra. This construction can be seen as the algebraic  counterpart of the geometric construction of a Lie algebroid from a given Lie groupoid\footnote{In Appendix \ref{ssec:LA-LG} we will review the latter construction,  from a slightly different point of view.}. In case the Hopf algebroid we start with is a split one, then we show that this construction already appeared  in the setting of  affine group $\Bbbk$-scheme actions \cite{DemazureGabriel}, see also  Appendix \ref{sec:A1} for more details.
}

Keep the above notations. For instance,  the underlying $A$-modules of the  $\cH$-module $(\cI/\cI^{2})$ are denoted by $\big(\cI/\cI^{2}\big){}_{\Sscript{p}}={}_{\Sscript{p}}\big(\cI/\cI^{2}\big)$, for every $p \in \{s,t\}$.

\begin{proposition}\label{prop:Omega}
For a Hopf algebroid $(A, \cH)$ and a $\cH$-module $M$, there is a natural isomorphism
\begin{equation}\label{Eq:NatOmega}
\xymatrix@R=0pt@C=30pt{  \Derk{s}{\cH}{M} \ar@{->}^-{\cong}[rr]  &   &  \Hom{\cH}{\Hst \, \tensor{A}\, {}_{\Sscript{s}}\big( \frac{\cI}{\cI^{2}}\big)}{M}  \\ \ar@{|->}^{}[rr] \delta   & & \Big[ u\tensor{A}\pis(v) \longmapsto u\, \cS(v_{1})\delta(v_{2})  \Big]  \\  \Big[ u\longmapsto u_{1}\, f(1\tensor{A} \pis(u_{2}) \Big]   & & \ar@{|->}^{}[ll]  f, }
\end{equation}
of  $\cH$-modules.
\end{proposition}
\begin{proof}
It follows from Corollary \ref{cor:PQ}, Lemma \ref{lema:rep} and the usual hom-tensor  adjunction.
\end{proof}

\begin{corollary}\label{coro:Omega}
Let $(A, \cH)$ be an Hopf algebroid. Then the K\"ahler module $\Omega_{A}^{s}(\cH)$ of $\cH$ with respect to the source map is, up to a canonical isomorphism,  given by:
$$
\Omega_{A}^{s}(\cH) \, \cong \, \Hst \tensor{A} {{{}_{\Sscript{s}}}}\Big(\frac{\cI}{\cI^{2}}\Big), \quad \Big( \psis: \Hs \longrightarrow  \Omega_{A}^{s}(\cH), \; \big[  u \longmapsto u_{1}\tensor{A} \pis(u_{2}) \big] \Big)
$$
where $\psis$ is the morphism of Eq. \eqref{Eq:tensorpis} and now becomes the universal derivation.
\end{corollary}
\begin{proof}
It is clear that,  if we take $M:= \Hst \tensor{A} {}_{\Sscript{s}}\Big(\frac{\cI}{\cI^{2}}\Big)$ in Proposition \ref{prop:Omega}, then  the map corresponding to $f:=\id{ }$ is exactly  the morphism $\psis$ so that $\psis\in\Derk{s}{\cH}{\Hst \tensor{A} {}_{\Sscript{s}}\Big(\frac{\cI}{\cI^{2}}\Big)}$.
\end{proof}

\begin{remark}
The analogous of Corollary \ref{coro:Omega} holds for $t$ as well, in the sense that we have an isomorphism of $\cH$-modules $\Derk{t}{\cH}{M}\cong \Hom{\cH}{({\cI}/{\cI^{2}}){}_{\Sscript{t}} \, \tensor{A}\, \Hst }{M}$ which makes of $\Omega_{A}^{t}(\cH) \, \cong \, ({\cI}/{\cI^{2}}){}_{\Sscript{t}} \, \tensor{A}\, \Hst$ the K\"ahler module with respect to the target. The universal derivation turns out to be the morphism $\psit$ of \eqref{Eq:tensorpis}.
\end{remark}

Next, we give another example of Lie-Rinehart algebra attached to a given Hopf algebroid.
Recall that the structure of  $A$-coring on $\cH$ is given  on the bimodule $\Hst$ and that a left $\cH$-comodule is a left $A$-module $N$ together with a coassociative and counital left $A$-linear coaction $\rho_N:{}_AN\to \Hst\tensor{A}{}_AN$. One can consider the distinguished left $\cH$-comodule $(\sH,\Delta)$.
The usual adjunction between  $\Hst \tensor{A}- : \lmod{A} \to \lcomod{\cH}$ and the forgetful functor $\oO: \lcomod{\cH} \to \lmod{A}$ leads to a bijection
\begin{equation}\label{Eq:theta}
\theta:\ldual{\cH}\longrightarrow \mathrm{End}^{\Sscript{\cH}}(\cH), \qquad\Big(\alpha\longmapsto[u\mapsto u_1t(\alpha(u_2))]\Big)
\end{equation}
where $\mathrm{End}^{\Sscript{\cH}}(\cH)$ denotes the endomorphism ring of the left $\cH$-comodule $(\sH,\Delta)$. It is, in fact, an $A$-ring via the ring map
\begin{equation}\label{Eq:AHH}
A \longrightarrow \mathrm{End}^{\Sscript{\cH}}(\cH),\qquad \Big(a\longmapsto \big[a\cdot\id{\cH}:u\mapsto ut(a)\big]\Big)
\end{equation}
As a consequence, there exists a unique $A$-ring structure on $\ldual{\cH}$ such that $\theta$ becomes an $A$-ring homomorphism and it is explicitly given by
\begin{equation}\label{Eq:ast}
A\longrightarrow \ldual{\cH}, \quad \Big(a\longmapsto \big[u\mapsto \varepsilon(u)a\big]\Big), \qquad \alpha * \beta: \ldual{\cH} \longrightarrow A, \quad \Big( u \longmapsto \alpha\big(u_{1}t(\beta(u_{2})\big)  \Big).
\end{equation}

\begin{remark}\label{rem:starH} Let us make the following observations.
\begin{enumerate}[label=(\arabic*)]
\item Notice that  the $\ldual{\cH}$ of equation \eqref{Eq:ast} is \emph{not} the convolution algebra of the $A$-coring $\Hst$ as defined in \eqref{eq:convolution}, but it is its opposite.
\item The $A$-bimodule structure on $\ldual{\cH}$ is explicitly given, for all $a,b\in A, u\in \cH$, by
\begin{align*}
(a\cdot \alpha\cdot b)(u) = \big((a\varepsilon)*\alpha*(b\varepsilon)\big)(u) = a\varepsilon\bigg(u_1t\Big(\alpha\Big(u_2t\big(b\varepsilon(u_3)\big)\Big)\Big)\bigg) = a\alpha\left(ut\left(b\right)\right).
\end{align*}
\item\label{item:starH} One may consider also the adjunction between  $- \tensor{A}\Hst : \rmod{A} \to \rcomod{\cH}$ and the forgetful functor $\oO: \rcomod{\cH} \to \rmod{A}$. By repeating the foregoing procedure for the distinguished $\cH$-comodule $(\Ht,\Delta)$ one may endow $\rdual{\cH}$ with an $A$-ring structure with product
\begin{equation}\label{eq:ast'}
(f*'g)(u)=f\Big(s\big(g(u_1)\big)u_2\Big).
\end{equation}
However, this turns out to be isomorphic as an $A$-ring to $\ldual{\cH}$ via $\ldual{\cH}\to \rdual{\cH}, \left(f\mapsto f\circ\cS\right)$ in light of \ref{item:Santicocomm} of Remark \ref{rem:Santicocomm}. Indeed, for all $f,g\in\ldual{\cH},u\in\cH$ we have $\varepsilon(\cS(u)) =\varepsilon(u)$ and
\begin{align*}
\Big((f\circ \cS)*'(g\circ\cS)\Big)(u) & \stackrel{\eqref{eq:ast'}}{=} (f\circ \cS)\Big(s\big((g\circ\cS)(u_1)\big)u_2\Big) = f\Big(\cS(u_2)t\big(g(\cS(u_1))\big)\Big) \\
& = f\Big(\cS(u)_1t\big(g(\cS(u)_2)\big)\Big) \stackrel{\eqref{Eq:ast}}{=} \big((f*g)\circ\cS\big)(u).
\end{align*}
\end{enumerate}
\end{remark}

In this direction, notice that $ \Derk{s}{\cH}{\cH} $ admits a Lie $\K$-algebra structure given by the commutator bracket. We can consider the (left) $A$-submodule of $\mathrm{End}^{\Sscript{\cH}}(\cH)$ defined by
\begin{multline}\label{Eq:Der}
\begin{gathered}
\DerHs \, :=\, \mathrm{End}^{\Sscript{\cH}}(\cH) \cap  \Derk{s}{\cH}{\cH} \\ \, =\LR{ \delta \in \Hom{\K}{\cH}{\cH}|\, \, \delta \circ s=0, \,\delta(uv)=\delta(u)v+u\delta(v), \, \Delta(\delta(u))= u_{1}\tensor{A} \delta(u_{2})\; \text{ for every}\, u, v \in \cH }
\end{gathered}
\end{multline}
which inherits form $ \Derk{s}{\cH}{\cH} $ a Lie $\K$-algebra structure.

From now on, we will denote by $\Aep$ the $\cH$-module with underlying $A$-module $A$ and action via the algebra map $\varepsilon$. Notice that (with the conventions introduced at the beginning of \S\ref{ssec: Der}) $\As =\At=A$, since we know that $\varepsilon \circ s=\varepsilon \circ t=id$. Thus, there is only one  $A$-module  structure on $\Derk{s}{\cH}{\Aep}$, given by
\begin{equation}\label{Eq:LR}
a\, \delta : \cH \longrightarrow \Aep, \quad \Big(  u \longmapsto a\delta(u)  \Big).
\end{equation}

\begin{lemma}\label{lema:LR}
The isomorphism $\theta$ of equation \eqref{Eq:theta} induces an isomorphism $\theta^\prime $ of $A$-modules which makes commutative the following diagram
$$
\xymatrix@R=18pt@C=40pt{  &\ldual{\cH} \ar@{->}^-{\theta}[r]  & \mathrm{End}^{\Sscript{\cH}}(\cH)         \\
\ldual{\Big(\frac{\cI}{\cI^{2}}\Big)}\ar@{->}^{\cong}[r]\ar@{->}^{\ldual{(\pis)}}[ur] & \ar@{-->}^{\theta^\prime}[r]   \Derk{s}{\cH}{\Aep}\ar@{^{(}->}^{}[u]    & \DerHs.   \ar@{^{(}->}^{}[u]  }
$$
Moreover, $\Derk{s}{\cH}{\Aep}$ admits a structure of Lie $\K$-algebra with bracket
\begin{equation}\label{Eq:bracket}
[\delta, \delta']:= \delta*\delta'-\delta'*\delta : \cH \longrightarrow \Aep, \quad \Big(  u \longmapsto \delta\big( u_{1}t(\delta'(u_{2}) \big) - \delta'\big( u_{1}t(\delta(u_{2}) \big) \Big)
\end{equation}
which turns $\theta^\prime$ into an isomorphism of Lie $\K$-algebras and this structure can be transferred to $\ldual{\Big(\frac{\cI}{\cI^{2}}\Big)}$ in a unique way making $\ldual{(\pis)}$ an inclusion of Lie $\K$-algebras.
\end{lemma}
\begin{proof}
Note that $\theta^{-1}(\delta)=\varepsilon\circ\delta$ for every $\delta\in \mathrm{End}^{\Sscript{\cH}}(\cH)$ so that it is clear that $\theta^{-1}( \DerHs)\subseteq \Derk{s}{\cH}{\Aep}$. On the other hand, given $\delta\in \Derk{s}{\cH}{\Aep} $, for every $a \in A$ and $u,v \in \cH$,  we have
\begin{gather*}
\theta(\delta)(s(a))=s(a)_1t(\delta(s(a)_2))=s(a)t(\delta(1))=0,\\
\Delta(\theta(\delta)(u)) = \Delta(u_1t(\delta(u_2))) = u_1\tensor{A}u_2t(\delta(u_3)) = u_1\tensor{A}\theta(\delta)(u_2)
\end{gather*}
and
\begin{gather*}
  \theta(\delta)(uv)=u_1v_1t(\delta(u_2v_2))=u_1v_1t\Big(\delta(u_2)\varepsilon(v_2)+\varepsilon(u_2)\delta(v_2)\Big)\\
  =u_1t(\delta(u_2))v_1t(\varepsilon(v_2))+u_1t(\varepsilon(u_2))v_1t(\delta(v_2))  =\theta(\delta)(u)v+u\theta(\delta)(v).
\end{gather*}
Therefore, $\theta(\Derk{s}{\cH}{\Aep} )\subseteq \DerHs$. It is now clear that $\theta$ induces an isomorphism
$$
\theta^\prime:\Derk{s}{\cH}{\Aep} \longrightarrow \DerHs
$$
making the right square diagram in the statement commutative. Since $\theta (\delta*\delta^\prime)=\theta (\delta)\circ \theta (\delta^\prime)$ and $\DerHs$ is a Lie subalgebra of $\mathrm{End}^{\Sscript{\cH}}(\cH)$ we get that $\Derk{s}{\cH}{\Aep}$ becomes a Lie subalgebra of $\ldual{\cH}$ with bracket defined as in the statement. Since $\Derk{s}{\cH}{\Aep}=\Der{A}{\Hs}{\As}$, we can apply Lemma \ref{lema:rep} to complete the diagram with the commutative triangle in the statement.
\end{proof}

In contrast with the Hopf algebra case, the Lie algebra $\Derk{s}{\cH}{\Aep}$ admits a richer structure. Namely that of Lie-Rinehart algebra. The anchor map is provided as follows.

\begin{proposition}\label{prop:LR}
Let $(A, \cH)$ be a Hopf algebroid. The pair $(A, \Derk{s}{\cH}{\Aep})$ is a Lie-Rinehart algebra with anchor map:
\begin{equation}\label{eq:LR}
\xymatrix@R=0pt@C=40pt{  \Derk{s}{\cH}{\Aep} \ar@{->}^-{\omega:=\Derk{s}{t}{\Aep}}[rr]  &   &  \mathrm{Der}_{\Sscript{\K}}(A)  \\ \ar@{|->}^{}[rr] \delta   & & \delta\circ t.   }
\end{equation}
\end{proposition}

\begin{proof}
The map $\omega$ is clearly a well-defined $A$-linear map. Let us check that its also a Lie $\K$-algebra map. Take $\delta, \delta' \in \Derk{s}{\cH}{\Aep}$ and an element $a \in A$, then
\begin{eqnarray*}
  (\delta*\delta')(t(a)) \stackrel{\eqref{Eq:ast}}{=}\delta(t(a)_1t(\delta'(t(a)_2))=\delta(t(\delta'(t(a))))=(\delta\circ t)((\delta'\circ t)(a))
\end{eqnarray*}
so that $(\delta*\delta')\circ t=(\delta\circ t)\circ(\delta'\circ t)$ and hence
\begin{align*}
  \omega\big([\delta, \delta']\big)=[\delta, \delta']\circ t=(\delta*\delta'-\delta'*\delta)\circ t=(\delta\circ t)\circ(\delta'\circ t)-(\delta'\circ t)\circ(\delta\circ t) \\ =\omega(\delta)\omega(\delta') - \omega(\delta')\omega(\delta)=[\omega(\delta), \omega(\delta')].
\end{align*}
Therefore, $\omega\big([\delta, \delta']\big) = [\omega(\delta), \omega(\delta')]$.  We still have to show that $\omega$ satisfies equation \eqref{Eq:Aseti}. So take $a \in A$ and $\delta, \delta'$ as above. Then, for any element $u \in\cH$,  we have,
\begin{eqnarray*}
[\delta,  a \delta'] (u) & =& \delta\big(u_{1}t(a\delta'(u_{2}))\big) - a\delta'\big(u_{1}t(\delta(u_{2}))\big)   \\   &=&   \delta\big(u_{1}t(a)t(\delta'(u_{2}))\big) - a(\delta'*\delta)(u)
\\   &=&    \delta(t(a))\varepsilon\Big(u_{1}t(\delta'(u_{2}))\Big) + \varepsilon(t(a)) \delta\big(u_{1}t(\delta'(u_{2})\big) - a(\delta'*\delta)(u)
\\   &=&    \delta(t(a)) \delta'(u) + a(\delta*\delta')(u)  -a(\delta'*\delta)(u)
\\ &=& a\big(  \delta*\delta' -\delta'*\delta\big)(u)   +\delta(t(a))\delta'(u)
\\ &=& a[\delta,  \delta'] (u) + \omega(\delta)(a)\delta'(u).
\end{eqnarray*}
This implies that $[\delta,  a \delta']= a[\delta,  \delta'] + \omega(\delta)(a)\delta'$ and the proof is complete.
\end{proof}

\begin{remark}\label{rem:LR}
One can perform another construction of a Lie-Rinehart algebra from a given Hopf algebroid $(A,\cH)$ by interchanging $s$ with $t$, however, the result will be the same up to a canonical isomorphism.  In fact, by resorting to \ref{item:Santicocomm} of Remark \ref{rem:Santicocomm}, Corollary \ref{cor:PQ} and \ref{item:starH} of Remark \ref{rem:starH}, one may prove that there is an isomorphism of Lie-Rinehart algebras
$$
\Derk{s}{\cH}{\Aep} \stackrel{\cong}{\longrightarrow} \Derk{t}{\cH}{\Aep}, \qquad \Big(\delta\longmapsto \delta\circ \cS\Big)
$$
where the latter is a Lie-Rinehart algebra with anchor map $\omega':=\Derk{t}{s}{\Aep}$.
\begin{invisible}
Indeed, for example,
\begin{align*}
\left([\delta,\delta']\circ \cS\right)(u) & = (\delta*\delta')(\cS(u))-(\delta'*\delta)(\cS(u)) \\
& \stackrel{\eqref{Eq:ast}}{=} \delta\Big(\cS(u)_1t\big(\delta'(\cS(u)_2)\big)\Big)-\delta'\Big(\cS(u)_1t\big(\delta(\cS(u)_2)\big)\Big) \\
& = \delta\Big(\cS(u_2)t\big(\delta'(\cS(u_1))\big)\Big)-\delta'\Big(\cS(u_2)t\big(\delta(\cS(u_1))\big)\Big) \\
& = (\delta\circ\cS)\Big(s\big((\delta'\circ\cS)(u_1)\big)u_2\Big)-(\delta'\circ\cS)\Big(s\big((\delta\circ\cS)(u_1)\big)u_2\Big) \\
& \stackrel{\eqref{eq:ast'}}{=} \left[\delta\circ\cS,\delta'\circ\cS\right](u).
\end{align*}
\end{invisible}
\end{remark}

\begin{example}\label{Exam:Main}
Let $(H,m,u,\Delta,\varepsilon,S)$ be a commutative Hopf $\K$-algebra and let $(A,\mu,\eta,\rho)$ be a left $H$-comodule commutative algebra, that is: an algebra in the monoidal category of left $H$-comodules which is commutative as $\K$-algebra. By the left-hand version of $(2)$ in Example \ref{exm:Halgd}, we know that $\cH:=H\otimes A$ is a split Hopf algebroid with its canonical algebra structure (i.e., $(x\otimes a)(y\otimes b)=xy\otimes ab$) and
\begin{gather*}
\eta_{\cH}(a\otimes b)=a_{-1}\otimes a_0b, \quad \Delta_{\cH}(x\otimes a)=(x_1\otimes 1)\tensor{A} (x_2\otimes a), \quad
\varepsilon_{\cH}(x\otimes a)=\varepsilon(x)a, \quad \cS(x\otimes a)=S(x)a_{-1}\otimes a_0.
\end{gather*}
Notice that tensoring by $A$ over $\K$ induces an anti-homomorphism of Lie algebras
\begin{equation*}
\tau:\Der{\K}{H}{\K_\varepsilon}\to \Derk{t}{\cH}{A_{\varepsilon_{\cH}}}; \quad \left[\delta\mapsto \delta\otimes A\right].
\end{equation*}
Indeed,
\begin{align*}
(\delta\tensor{}A)(xy\otimes ab) & = \delta(xy)ab = \delta(x)\varepsilon(y)ab+\varepsilon(x)\delta(y)ab  \\
& = (\delta\otimes A)(x\otimes a)\varepsilon_{\cH}(y\otimes b) + \varepsilon_{\cH}(x\otimes a)(\delta\otimes A)(y\otimes b),\\
\left[\tau(\delta),\tau(\delta')\right](x\otimes a) & = (\tau(\delta)*\tau(\delta'))(x\otimes a)-(\tau(\delta')*\tau(\delta))(x\otimes a) \\
 & = \tau(\delta)\left(s\left(\tau(\delta')(x_1\otimes 1)\right)(x_2\otimes a)\right)-\tau(\delta')\left(s\left(\tau(\delta)(x_1\otimes 1)\right)(x_2\otimes a)\right) \\
 & = \delta'(x_1)\tau(\delta)\left(x_2\otimes a\right)-\delta(x_1)\tau(\delta')\left(x_2\otimes a\right) \\
 & = \delta'(x_1)\delta(x_2)a-\delta(x_1)\delta'(x_2)a = \tau(\delta'*\delta-\delta*\delta')(x\otimes a).
\end{align*}
Consider now the composition
\begin{equation}\label{Eq:DG}
\xymatrix{
\Der{\K}{H}{\K_\varepsilon} \ar[r]^-{\tau} & \Derk{t}{\cH}{A_{\varepsilon_{\cH}}} \ar[r]^-{\omega} & \Ders{\K}{A}}
\end{equation}
where $\omega(\delta):=\Derk{t}{s}{A_{\varepsilon_{\cH}}}(\delta)=\delta\circ s$ is the anchor map of the Lie-Rinehart algebra $ \Derk{t}{\cH}{A_{\varepsilon_{\cH}}}$. For every $\delta\in\Der{\K}{H}{\K_\varepsilon}$, it follows by a direct check that for all $a\in A$
\begin{equation*}
\omega\left(\tau(\delta)\right)(a)=\tau(\delta)\left(a_{-1}\otimes a_0\right) = \delta(a_{-1})a_0.
\end{equation*}

Let us see now that the anti-homomorphism of Lie algebras of Equation \eqref{Eq:DG} already appeared in \cite{DemazureGabriel} in geometric terms. To this aim, notice that $H$ and $A$ give rise to an affine $\K$-group $\cG:=\Calg\left(H,-\right)$ and an affine $\K$-scheme $\cX:=\Calg\left(A,-\right)$, respectively. Hence the map in \eqref{Eq:DG} becomes the anti-homomorphism of Lie algebras $\cL ie(\cG)(\K) \to  \Ders{\K}{\mathscr{O}_\K(\cX)}$ (see \cite[II, \S4, n\textsuperscript{o}4, Proposition 4.4, page 212]{DemazureGabriel}). For the sake of completeness, we include such a construction in  \S\ref{sec:A1} of the Appendix and we show  that these two anti-homomorphisms of Lie algebras are essentially the same.
What we just showed is that the map \eqref{Eq:DG} descends from the anchor map of the Lie-Rinehart algebra $\Derk{t}{\cH}{A_{\varepsilon_{\cH}}}$ of $(A,\cH)$.
\end{example}

\subsection{The differential functor and base change}\label{ssec:lL}
Below we show that the construction performed in Proposition \ref{prop:LR} is functorial. We also discuss  the compatibility of this construction with the base ring change.

\begin{proposition}\label{prop:L}
Fix a commutative  algebra $A$. Then the correspondence
$$
\lL: \Halgd{A} \longrightarrow \LieR{A}, \quad \Big(  \cH \longrightarrow \lL(\cH):=\Derk{s}{\cH}{\Aep} \Big)
$$
establishes a contravariant functor form the category of commutative Hopf algebroids with base algebra $A$ to the category of Lie-Rinehart algebras over $A$.
\end{proposition}
\begin{proof}
Let $\phi: \cH \to \cK$ be a morphism in $\Halgd{A}$. We need to check that the map
\begin{equation}\label{eq:defL}
\lL_{\Sscript{\phi}}: \lL(\cK) \longrightarrow \lL(\cH), \quad \Big(  \delta \longmapsto \delta \circ \phi \Big)
\end{equation}
is a morphism of Lie-Rinehart algebras. This map is  clearly an $A$-linear and a Lie algebra morphism. Thus, we only need to check that it is compatible with the anchor, which is immediate as the following argument shows. For $a \in A$ and $\delta \in \lL(\cH)$, we have $\omega(\lL_{\Sscript{\phi}}(\delta))\,=\, \lL_{\Sscript{\phi}}(\delta)\circ t_{\cH}\,=\, \delta\circ\phi\circ t_{\cH}\,=\, \delta\circ t_{\cK}\, =\, \omega(\delta)$.
\end{proof}

The functor $\lL$ will be referred to as \emph{the differential functor}.
Notice that since the notion of a morphism of Lie-Rinehart algebras over different algebras is not always possible (mainly due to the problem of connecting $\Ders{\K}{A}$ and $\Ders{\K}{B}$ in a natural way), the differential functor cannot be defined on maps of Hopf algebroids with different base algebras.
Let us analyse closely this situation.

Let $(\phi_{0}, \phi_{1}):(A,\cH) \to (B,\cK)$ be  morphism of Hopf algebroids and consider the associated extended morphism of Hopf algebroids $(\id{},\, \phi): (B, B\tensor{A}\cH\tensor{A}B) \to (B,\cK)$ where $\phi(b\tensor{A}u\tensor{A}b')= s_{\cK}(b)\phi_{1}(u)t_{\cK}(b')$.  Define also the map $\kappa:\cH\to B\tensor{A}\cH\tensor{A}B$ which maps $u$ to $1\tensor{A}u\tensor{A}1$ and note that $\phi\circ \kappa=\phi_1$. Denote by $B_{\Sscript{\phi\varepsilon}}$ the $\cH$-bimodule $B$ with action given by the algebra extension $\phi_{0}\varepsilon: \cH \to B$.

In what follows, by abuse of notation, we will denote by $\ldual{f}$ the pre-composition with a morphism  $f$, i.e., the map $g\mapsto g\circ f$. The domain and codomain of this map will be clear from the context. Similarly we will use the notation ${}_*f$ for $g\mapsto f\circ g$.
In this way, we have the following linear maps :
\begin{eqnarray*}
{{}_{*}(\phi_0):=\Derk{s}{\cH}{\phi_0}:} \, \Derk{s}{\cH}{\Aep}  \longrightarrow \Derk{s}{\cH}{B_{\Sscript{\phi\varepsilon}}}, &  \Big( \delta \longmapsto   {\phi_0  \circ \delta} \Big) \\
{\lL_{\phi}=\ldual{\phi}:=\Derk{s}{\phi}{B}:} \,  \Derk{s}{\cK}{\Bep}  \longrightarrow \Derk{s}{B\tensor{A}\cH\tensor{A}B}{\Bep}, &  \Big( \delta \longmapsto  \delta \circ \phi \Big) \\
{\ldual{\kappa}:=\Derk{s}{\kappa}{B}:} \,  \Derk{s}{B\tensor{A}\cH\tensor{A}B}{\Bep}  \longrightarrow \Derk{s}{\cH}{B_{\Sscript{\phi\varepsilon}}}, &  \Big( \delta \longmapsto {\delta\circ \kappa }\Big) \\
{\ldual{t}:=\Derk{s}{t}{B}:} \,  \Derk{s}{\cH}{B_{\Sscript{\phi\varepsilon}}}   \longrightarrow  \Derk{}{A}{B}, &  \Big( \gamma \longmapsto {\gamma \circ t} \Big).
\end{eqnarray*}

\begin{proposition}\label{prop:Cartesian}
Let $(\phi_{0},\phi_{1}): (A,\cH) \to (B,\cK)$ be as above. Then we have a commutative diagram of $A$-modules
\begin{equation}\label{Eq:Etale}
\begin{gathered}
\xymatrix{  \Derk{s}{\cK}{\Bep}  \ar@{->}^{\ldual{\phi}}[rr]\ar[drr]|{\ldual(\phi_1)}  & &   \Derk{s}{B\tensor{A}\cH\tensor{A}B}{\Bep}   \ar@{->}^{\omega}[rr]  \ar@{->}^{\ldual{\kappa}}[d]   & & \Ders{\K}{B}   \ar@{->}^{\ldual(\phi_{0})}[d]  \\ \Derk{s}{\cH}{\Aep}  \ar@{->}^{{}_{*}(\phi_0)}[rr]  &  & \Derk{s}{\cH} {B_{\Sscript{\phi\varepsilon}}}   \ar@{->}^{\ldual{t}}[rr] & &  \Derk{}{A}{B} }
\end{gathered}
\end{equation}
where the right-hand side square is cartesian.  Moreover $\ldual{\phi}$ is a map of Lie-Rinehart algebras.
\end{proposition}
\begin{proof}
We only show that the square is cartesian. Define $\tau:B\to B\tensor{A}\cH\tensor{A}B:b\mapsto 1\tensor{A}1\tensor{A}b$. Then $\kappa(t(a))=1\tensor{A}t(a)\tensor{A}1=1\tensor{A}1\tensor{A}\phi_0(a)=\tau(\phi_0(a))$ so that $\kappa\circ t=\tau\circ\phi_0$. Note that $\omega=\ldual{\tau}:=\Derk{}{\tau}{B}$ so that $\ldual{t}\circ\ldual{\kappa}=\ldual{(\kappa\circ t)}=\ldual{(\tau\circ\phi_0)}=\ldual(\phi_0)\circ \ldual{\tau}$ and the square commutes. Hence we have the diagonal map
$$(\ldual{\kappa},\ldual{\tau}):\Derk{s}{B\tensor{A}\cH\tensor{A}B}{\Bep}\longrightarrow \Derk{s}{\cH} {B_{\Sscript{\phi\varepsilon}}}\, \underset{\Derk{}{A}{B}}{\times} \,\Ders{\K}{B}, \quad \Big(\delta\longmapsto (\ldual{\kappa}(\delta),\ldual{\tau}(\delta)) \Big).
$$
Let us check that this map is invertible. Take $\delta\in \Derk{s}{B\tensor{A}\cH\tensor{A}B}{\Bep}$.
Then $$\delta(b\tensor{A}u \tensor{A}b')=b\delta(\kappa(u)\tau(b'))=b\delta(\kappa(u))b'+b\phi_0(\varepsilon(u))\delta(\tau(b'))$$
so that $\delta(b\tensor{A}u \tensor{A}b')=b\delta_1(u)b'+b\phi_0(\varepsilon(u))\delta_2(b')$ where we set $\delta_1:=\delta\circ \kappa=\ldual{\kappa}(\delta)$ and $\delta_2:=\delta\circ \tau=\ldual{\tau}(\delta)$. Thus the map $(\ldual{\kappa},\ldual{\tau})$ is injective. It is also surjective as any pair $(\delta_1,\delta_2)$ in its codomain is image of
$$
\delta:B\tensor{A}\cH\tensor{A}B\to \Bep,\quad\Big(b\tensor{A}u \tensor{A}b'\mapsto b\delta_1(u)b'+b\phi_0(\varepsilon(u))\delta_2(b')\Big) .
$$
This is a well-defined map thanks to the equality  $ \ldual{t}(\delta_1)=\ldual(\phi_0) (\delta_2)$.
\begin{invisible}
Indeed,
\begin{align*}
\delta\left(b\phi_0(a)\tensor{A}u\tensor{A}\phi_0(a')b'\right) & = b\phi_0(a)\delta_1(u)\phi_0(a')b'+b\phi_0(a)\phi_0(\varepsilon(u))\delta_2(\phi_0(a')b') \\
 & = b\phi_0\Big(\varepsilon\big(s(a)\big)\Big)\delta_1(u)\phi_0\Big(\varepsilon\big(t(a')\big)\Big)b'+b\phi_0\Big(\varepsilon\big(s(a)u\big)\Big)\delta_2(\phi_0(a')b') \\
 & = b\delta_1\Big(s(a)u\Big)\phi_0\Big(\varepsilon\big(t(a')\big)\Big)b'+b\phi_0\Big(\varepsilon\big(s(a)u\big)\Big)\delta_2(\phi_0(a'))b'+b\phi_0(\varepsilon(s(a)u))\phi_0(a')\delta_2(b') \\
 & = b\delta_1\Big(s(a)u\Big)\phi_0\Big(\varepsilon\big(t(a')\big)\Big)b'+b\phi_0\Big(\varepsilon\big(s(a)u\big)\Big)\delta_1\big(t(a')\big)b'+b\phi_0\Big(\varepsilon\big(s(a)ut(a')\big)\Big)\delta_2(b') \\
 & = b\delta_1\Big(s(a)ut(a')\Big)b'+b\phi_0\Big(\varepsilon\big(s(a)ut(a')\big)\Big)\delta_2(b')
\end{align*}
\end{invisible}
Furthermore,  it is clear that $\delta \circ s =0$ and  one shows that $\delta$ is a derivation as follows. For every $b,b',c,c' \in B$ and $u,v \in \cH$, we have
\begin{align*}
  \delta((b\tensor{A}u \tensor{A}b')&(c\tensor{A}v \tensor{A}c')) = \delta(bc\tensor{A}uv \tensor{A}b'c') \\
	&= bc\delta_1(uv)b'c'+bc\phi_0(\varepsilon(uv))\delta_2(b'c')\\
  &=bc\Big(\delta_1(u)\phi_0(\varepsilon(v))+\phi_0(\varepsilon(u))\delta_1(v)\Big)b'c' + bc\phi_0(\varepsilon(u))\phi_0(\varepsilon(v))\Big(\delta_2(b')c'+b'\delta_2(c')\Big)\\
	&= \Big(b\delta_1(u)b'+b\phi_0(\varepsilon(u))\delta_2(b')\Big)c\phi_0(\varepsilon(v))c'+ b\phi_0(\varepsilon(u))b'\Big(c\delta_1(v)c' +c\phi_0(\varepsilon(v))\delta_2(c')\Big)\\
& = \delta(b\tensor{A}u \tensor{A}b')\varepsilon(c\tensor{A}v \tensor{A}c')+\varepsilon(b\tensor{A}u \tensor{A}b')\delta(c\tensor{A}v \tensor{A}c').
\end{align*}
Note that $\ldual{\kappa}\circ\ldual{\phi}=\ldual(\phi\circ\kappa)=\ldual(\phi_1)$ so that triangle drawn in the statement commutes.
Since $\ldual{\phi}=\lL_{\phi}$, we have that $\ldual{\phi}$ is by Proposition \ref{prop:L} a morphism of Lie-Rinehart algebras and this completes the proof.
\end{proof}

\begin{remark}\label{rem:Zuki}
 As one can expect there is no hope in general to obtain a morphism of Lie-Rinehart algebras which could relate $\Derk{s}{\cK}{\Bep}$ with ${\Derk{s}{\cH} {\Aep}}$ in diagram \eqref{Eq:Etale}. Even if we extended the $A$-module ${\Derk{s}{\cH} {\Aep}}$  to the $B$-module ${\Derk{s}{\cH} {\Aep}}\tensor{A}B$, then one still have to endow this $B$-module with a Lie-Rinehart algebra structure over $B$, which is not always feasible. Nevertheless, if we assume that ${}_*(\phi_0):\Derk{s}{\cH}{\Aep}  \to \Derk{s}{\cH} {B_{\Sscript{\phi\varepsilon}}}$ is a split-epimorphism, i.e., that there is some map $\gamma $ such that ${}_*(\phi_0)\circ \gamma=\id{ }$, then ${}_*(\phi_0)\circ (\gamma\circ\ldual(\phi_1))=\ldual(\phi_1)=\ldual{\kappa}\circ\ldual{\phi}$ so that $\gamma\circ\ldual(\phi_1):\Derk{s}{\cK}{\Bep} \to\Derk{s}{\cH} {\Aep}$ completes the diagram but it is not clear which kind of morphism it is.
\end{remark}


\section{Integrations functors in Lie-Rinehart algebras framework}\label{sec:iI}
In this section we construct functors from the category of Lie-Rinehart algebras to the category of commutative Hopf algebroids over a fixed commutative base algebra $A$. These functors are termed \emph{the integration functors}. There are in fact two ways of constructing the integration functor depending on which dual we are using, that is, depending on which contravariant functors we will use: $\Circ{(-)}$ or $\rsaft{(-)}$. Nevertheless, as we will see in the forthcoming section, the first one will lead (under some condition on the base algebra) to an adjunction only when restricted to Galois Hopf algebroids while the second one gives an adjunction to the whole category of commutative Hopf algebroids.

\begin{lemma}\label{lem:I}
Let $A$ be a commutative algebra. Then there are  contravariant functors
\begin{align*}
\iI:=\rcirc{(-)}\circ \cV_A: \LieR{A} \longrightarrow \Halgd{A}, \quad \Big(  L \longrightarrow \VL^{\circ} \Big),
\\ \iI':=\rsaft{(-)}\circ \cV_A: \LieR{A} \longrightarrow \Halgd{A}, \quad \Big(  L \longrightarrow \rsaft{\VL} \Big)
\end{align*}
together with  a natural transformation $\nabla :=\mb{\hat{\zeta}}_{\cV_A}: \iI \to \iI'$.
\end{lemma}
\begin{proof}
$\cV_A$ is the functor of Remark \ref{rem:primfunctor}, $\Circ{(-)}$ and $\rsaft{(-)}$ are those of Proposition \ref{prop:Fdual} and Theorem \ref{thm:Teoremone} respectively and $\mb{\hat{\zeta}}$ is the natural transformation of Example \ref{exam:zetahat}.
\end{proof}

Let $(A,L)$ be a Lie-Rinehart algebra and consider its universal enveloping Hopf algebroid $(A, \VL)$. Attached to this datum,  there are then two commutative Hopf algebroids  $(A,\VL^{\circ})$ and $(A, \rsaft{\VL})$ and one can apply the differentiation functor to these objects and obtain other two Lie-Rinehart algebras.
In fact there is a commutative diagram:
$$
\xymatrix{ (A,L) \ar@{->}^-{\Theta_L}[rr]  \ar@{->}_-{\Theta'_L}[drr] & & (A, \lL(\jJ(L)))  \\ & & (A,\lL(\jJ'(L)))\ar@{->}_-{\lL(\nabla_L)}[u]  }
$$
of morphisms of Lie-Rinehart algebras, where $\Theta$ and $\Theta'$ are natural transformations explicitly given in Appendix \ref{ssec:UAdj}.
The following is a corollary of Theorem \ref{thm:drwho}.
\begin{proposition}\label{prop:drwho}
Assume that $\nabla$ of Lemma \ref{lem:I} is a monomorphism of corings on every component. Then,
$$
\Hom{\Halgd{A}}{\cH}{\nabla_L}:\Hom{\Halgd{A}}{\cH}{\iI(L)}  \to \Hom{\Halgd{A}}{\cH}{\iI'(L)}
$$
is a bijection for every commutative Hopf algebroid $(A,\cH)$ such that $\can{\cH}$ is a split epimorphism of $A$-corings and for every Lie-Rinehart algebra $(A,L)$.
\end{proposition}

\begin{proof}
Since $\nabla_L=\mb{\what{\zeta}}_{\cV_A(L)}$, the equivalent conditions of Theorem \ref{thm:drwho} hold. In particular, $\Coring_A(\cH,\nabla_L)$ is bijective and hence $\Hom{\Halgd{A}}{\cH}{\nabla_L}$ is injective. Moreover, consider $g\in \Hom{\Halgd{A}}{\cH}{\iI'(L)}$. By bijectivity of $\Coring_A(\cH,\nabla_L)$ there exists a $f\in \Coring_A(\cH,\iI(L))$ such that $\mb{\what{\zeta}}_{\cV_A(L)}\circ f=g$.

Since $(A,\cH)$ is a commutative Hopf algebroid, its multiplication $m_{\cH}:\cH\otimes \cH\to \cH$ factors through a $A$-bilinear morphism $\bar{m}_{\cH}:\cH\odot \cH\to \cH$ and since $\Delta_{\cH}$ and $\varepsilon_{\cH}$ are algebra morphisms we get that $\bar{m}_{\cH}$ is a morphism of corings. Analogously, also $\bar{m}_{\iI(L)}$ is a morphism of corings. Since $f$ is a morphism of corings as well, it induces a coring map $f\odot f: \cH\odot \cH\to \iI(L)\odot \iI(L)$, where $\odot$ is recalled in Remark \ref{rem:circdot}. From the following computation
\begin{align*}
\mb{\what{\zeta}}_{\cV_A(L)}\circ f\circ \bar{m}_{\cH} & = g \circ \bar{m}_{\cH} = \bar{m}_{\iI'(L)} \circ (g\odot g) = \bar{m}_{\iI'(L)} \circ (\mb{\what{\zeta}}_{\cV_A(L)}\odot \mb{\what{\zeta}}_{\cV_A(L)}) \circ (f\odot f) = \mb{\what{\zeta}}_{\cV_A(L)} \circ \bar{m}_{\iI(L)} \circ (f\odot f)
\end{align*}
and the fact that $\mb{\what{\zeta}}_{\cV_A(L)}$ is a monomorphism of corings, we get that $f\circ \bar{m}_{\cH}=\bar{m}_{\iI(L)} \circ (f\odot f)$ so that $f$ is multiplicative. We also have that $\mb{\what{\zeta}}_{\cV_A(L)}\circ f\circ \eta_{\cH} = g\circ \eta_{\cH} = \eta_{\iI'(L)} = \mb{\what{\zeta}}_{\cV_A(L)}\circ \eta_{\iI(L)}$ and since $\eta_{\cH}$ and $\eta_{\iI(L)}$ are morphisms of corings we get as above that $f\circ \eta_{\cH} = \eta_{\iI(L)}$.

Finally, since $\cS_{\cH}:\cH^{\text{cop}}\to \cH$, where $\cH^{\text{cop}}$ has the structure as in \eqref{eq:Ccop}, is easily checked to be a morphism of corings, from the following computation
$$
\mb{\what{\zeta}}_{\cV_A(L)}\circ f\circ \cS_{\cH} = g\circ\cS_{\cH} = \cS_{\iI'(L)} \circ g = \cS_{\iI'(L)} \circ \mb{\what{\zeta}}_{\cV_A(L)}\circ f=\mb{\what{\zeta}}_{\cV_A(L)}\circ \cS_{\iI(L)} \circ f
$$
we deduce that $f\circ \cS_{\cH}  = \cS_{\iI(L)} \circ f$. We have so proved that $f$ is a morphism of commutative Hopf algebroids and hence that $\Hom{\Halgd{A}}{\cH}{\nabla_L}$ is surjective as well.
\end{proof}

We give now a criterion for the existence of a morphism $(A,L) \to (A, \lL(\cH))$ of Lie-Rinehart algebra.

\begin{lemma}\label{lema:Htilda}
Let $(A,L)$ be a Lie-Rinehart algebra and $(A,\cH)$ a commutative Hopf algebroid. Assume that there is a morphism  $\widetilde{\sigma}: L \to \lL(\cH)= \Derk{s}{\cH}{A_{\Sscript{\varepsilon_{\cH}}}}$ of Lie-Rinehart algebras. The map $\sigma: \VL \to {}^*\cH$ given by $\sigma  \circ \iota_{\Sscript{A}}(a)=a \varepsilon_{\Sscript{\cH}}$, for every $a \in A$, and $\sigma  \circ \iota_{\Sscript{L}}(X)= -\widetilde{\sigma}(X)$, for every $X \in L$, is an $A$-ring map which satisfy the equalities of equation \eqref{eq:simildag}. That is,  for all $a,b\in A$, $u\in \VL$ and $x,y\in\cH$, we have
\begin{equation*}
\sigma(u)(\eta(a\otimes b))=\varepsilon_{\Sscript{\cV}}(\iota_A(b)u)a   \quad \text{ and } \quad  \sigma(u)(xy)=\sigma(u_{\Sscript{1}})(x)\sigma(u_{\Sscript{2}})(y).
\end{equation*}
\end{lemma}

\begin{proof}
Define $\phi_{\Sscript{A}}:A\to \ldual{\cH}$ sending $a$ to the map $a\varepsilon_{\Sscript{\cH}}$ and $\phi_{\Sscript{L}}:L\to\ldual{\cH}$ which sends $X$ to $-\widetilde{\sigma}(X)$.
\begin{invisible}
We compute for every $x\in\cH, X\in L, a\in A$
\begin{align*}
\left(\phi_{\Sscript{L}}(X)*\phi_{\Sscript{A}}(a)-\phi_{\Sscript{A}}(a)*\phi_{\Sscript{L}}(X)\right)(x) & \stackrel{\eqref{eq:convolution}}{=} \phi_{\Sscript{A}}(a)\Big(x_1t\left(\phi_{\Sscript{L}}(X)(x_2)\right)\Big)-\phi_{\Sscript{L}}(X)\Big(x_1t\left(\phi_{\Sscript{A}}(a)(x_2)\right)\Big) \\
 & = -a \varepsilon_{\cH} \left(x_1t\left(\widetilde{\sigma}(X)(x_2)\right)\right)+\widetilde{\sigma}(X)\big(x_1t\left(a \varepsilon_{\cH}(x_2)\right)\big) \\
 & = -a \varepsilon_{\cH}\left(x_1\right)\widetilde{\sigma}(X)(x_2)+\widetilde{\sigma}(X)(xt(a)) \\
 & = -a\widetilde{\sigma}(X)(s\left(\varepsilon_{\cH}\left(x_1\right)\right)x_2)+\widetilde{\sigma}(X)(x)a+\varepsilon_{\cH}(x)\widetilde{\sigma}(X)(t(a)) \\
 & = \varepsilon_{\cH}(x)\widetilde{\sigma}(X)(t(a)) = \phi_{\Sscript{A}}\left(\widetilde{\sigma}(X)(t(a))\right)(x) \\
 & = \phi_{\Sscript{A}}\left(\omega\left(\widetilde{\sigma}(X)\right)(a)\right)(x) = \phi_{\Sscript{A}}(\omega(X)(a))(x),\\
\Big(\phi_L(X)*\phi_A(a)\Big)(x) & = -a\widetilde{\sigma}(X)(x) = -\widetilde{\sigma}(aX)(x) = \phi_L(aX)(x).
\end{align*}
\end{invisible}
By the universal property of $\cV_{\Sscript{A}}(L)$ there exists a unique algebra morphism $\sigma:\cV_{\Sscript{A}}(L)\to \ldual{\cH}$ such that $\sigma\circ\iota_A =\phi_{\Sscript{A}}$ and $\sigma \circ \iota_{\Sscript{L}} = \phi_{\Sscript{L}}$. Since $\phi_A$ gives the $A$-ring structure of $\ldual{\cH}$, we have that $\sigma$ is an $A$-ring map.
Notice that
\begin{equation}\label{eq:sigmalin}
\sigma\big((\iota_A(a)u\big)(x) = \big(a\sigma(u)\big)(x) \stackrel{\eqref{eq:convolution}}{=} \sigma(u)\big(xt(a)\big)
\end{equation}
for all $u\in\cV_A(L), x\in\cH, a\in A$.
Let us check that $\sigma$ fulfils \eqref{eq:simildag}.
To this aim, let us denote by $B$ the subset of the elements $u\in\cV_{\Sscript{A}}(L)$ such that relations \eqref{eq:simildag} hold for all $a,b\in A$ and $x,y\in\cH$. By means of \eqref{eq:sigmalin} it is straightforward to check that $\iota_A(a)u,uv,1_{\cV_{\Sscript{A}}(L)},\iota_L(X)\in B$ for every $a\in A, X\in L$ and $u,v\in B$.
\begin{invisible}
Assume $u\in B$. Since $\sigma$ is $A$-bilinear, for all $a,b,c\in A$ we may compute
\begin{equation*}
\sigma\left(\iota_A(a)u\right)\left(\eta(b\otimes c)\right)  \stackrel{\eqref{eq:sigmalin}}{=} \sigma\left(u\right)\left(\eta(b\otimes c)\eta(1\otimes a)\right) = \sigma(u)\left(\eta(b\otimes ca)\right)  \stackrel{\left(u\,\in\,  B\right)}{=} \varepsilon_{\cV}(\iota_A(ca)u)b = \varepsilon_{\cV}(\iota_A(c)\iota_A(a)u)b,
\end{equation*}
and for all $x,y\in\cH$
\begin{align*}
\sigma\left(\iota_A(a)u\right)\left(xy\right) & \stackrel{\eqref{eq:sigmalin}}{=} \sigma(u)(xy\eta(1\otimes a)) \stackrel{\left(u \, \in \,  B\right)}{=} \sigma(u_1)(x)\sigma(u_2)(y\eta(1\otimes a)) \\
 & = \sigma(u_1)(x)\sigma(\iota_A(a)u_2)(y) = \sigma(\iota_A(a)u_1)(x)\sigma(u_2)(y) = \sigma\Big(\big(\iota_A(a)u\big)_1\Big)(x)\sigma\Big(\big(\iota_A(a)u\big)_2\Big)(y),
\end{align*}
so that $\iota_A(a)u \in B$.
Moreover, if $v \in B$ as well then
\begin{align*}
\sigma\left(uv\right)\left(\eta(a\otimes b)\right) & =  \big(\sigma\left(u\right)*\sigma\left(v\right)\big)\left(\eta(a\otimes b)\right) = \sigma\left(v\right)\Big(\eta(a\otimes 1)t\Big(\sigma\left(u\right)\left(\eta(1\otimes b)\right)\Big)\Big) \\
 & \stackrel{\left(u\, \in\,  B\right)}{=}  \sigma\left(v\right)\Big(\eta(a\otimes 1)t\Big(\varepsilon_{\cV}(\iota_A(b)u)\Big)\Big) = \sigma\left(v\right)\Big(\eta(a\otimes \varepsilon_{\cV}(\iota_A(b)u))\Big) \\
 & \stackrel{\left(v\, \in \,  B\right)}{=} \varepsilon_{\cV} \left(\iota_A\left(\varepsilon_{\cV}(\iota_A(b)u)\right)v\right)a \stackrel{\eqref{eq:doubleepsilon}}{=} \varepsilon_{\cV}\left(\iota_A(b)uv\right)a
\end{align*}
and
\begin{align*}
\sigma(uv)(xy) & = \Big(\sigma(u)*\sigma(v)\Big)(xy) = \sigma(v)\Big(x_1y_1t\Big(\sigma(u)\big(x_2y_2\big)\Big)\Big) \\
 & \stackrel{\left(u\, \in\,  B\right)}{=}  \sigma(v)\Big(x_1y_1t\Big(\sigma(u_1)(x_2)\Big)t\Big(\sigma(u_2)(y_2)\Big)\Big) \\
 & \stackrel{\left(v\, \in\,  B\right)}{=}  \sigma(v_1)(x_1)\sigma(v_2)\Big(y_1t\Big(\sigma(u_1)(x_2)\Big)t\Big(\sigma(u_2)(y_2)\Big)\Big) \\
 & = \sigma(v_1)(x_1)\sigma(v_2)\Big(y_1t\Big(\sigma(u_2)(y_2)\Big)t\Big(\sigma(u_1)(x_2)\Big)\Big) \\
 & \stackrel{\eqref{eq:sigmalin}}{=} \sigma(v_1)(x_1)\,\sigma\Big(\,\iota_A\Big(\sigma(u_1)(x_2)\Big)v_2\Big)\Big(y_1t\Big(\sigma(u_2)(y_2)\Big)\Big) \\
 & \stackrel{\eqref{takeuchi}}{=} \,\sigma\Big(\,\iota_A\Big(\sigma(u_1)(x_2)\Big)v_1\Big)(x_1)\sigma(v_2)\Big(y_1t\Big(\sigma(u_2)(y_2)\Big)\Big) \\
 & \stackrel{\eqref{eq:sigmalin}}{=} \sigma(v_1)\Big(x_1t\Big(\sigma(u_1)(x_2)\Big)\Big)\sigma(v_2)\Big(y_1t\Big(\sigma(u_2)(y_2)\Big)\Big) \\
 & = \Big(\sigma(u_1)*\sigma(v_1)\Big)(x)\Big(\sigma(u_2)*\sigma(v_2)\Big)(y) = \sigma((uv)_1)(x)\sigma((uv)_2)(y)
\end{align*}
so that $uv\in B$. Furthermore,
\begin{align*}
\sigma\left(1_{\cV_{\Sscript{A}}(L)}\right)(\eta(a\otimes b))  & = \varepsilon_{\cH}(\eta(a\otimes b)) = ab = \varepsilon_{\cV}
(\iota_A(b)1_{\cV_{\Sscript{A}}(L)})a,\\
\sigma\left(1_{\cV_{\Sscript{A}}(L)}\right)(xy) & = \varepsilon_{\cH}(xy) = \varepsilon_{\cH}(x)\varepsilon_{\cH}(y) = \sigma\left(1_{\cV_{\Sscript{A}}(L)}\right)(x)\sigma\left(1_{\cV_{\Sscript{A}}(L)}\right)(y)
\end{align*}
whence $1_{\cV_{\Sscript{A}}(L)}\in B$. Finally, let us check that for all $X\in L$, $\iota_L(X)\in B$. For all $a,b\in A$ we have
\begin{align*}
\sigma(\iota_L(X))\left(\eta(a\otimes b)\right) & = \phi_L(X)\left(\eta(a\otimes b)\right) = -\widetilde{\sigma}(X)\left(\eta(a\otimes b)\right) = -\widetilde{\sigma}(X)\left(s(a)t(b)\right) = -a\widetilde{\sigma}(X)\left(t(b)\right) \\
 & = -a\omega\left(\widetilde{\sigma}(X)\right)(b) = -a\omega\left(X\right)(b)  = a\varepsilon(b\otimes X) = \varepsilon\Big(\iota_A(b)\iota_L(X)\Big)a
\end{align*}
and
\begin{align*}
\sigma(\iota_L(X))(xy) & = \phi_L(X)(xy) = -\widetilde{\sigma}(X)(xy) = -\widetilde{\sigma}(X)(x) \varepsilon_{\cH}(y)-\varepsilon_{\cH}(x)\widetilde{\sigma}(X)(y) \\
 & = \sigma\big(\iota_L(X)\big)(x)\sigma\left(1_{\cV_{\Sscript{A}}(L)}\right)(y)+\sigma\left(1_{\cV_{\Sscript{A}}(L)}\right)(x)\sigma\big(\iota_L(X)\big)(y) \\
 & = \sigma\big(\iota_L(X)_1\big)(x)\sigma\big(\iota_L(X)_2\big)(y).
\end{align*}
Therefore $\iota_L(X)\in B$ and so, in light of the foregoing observations and the fact that $\cV_{\Sscript{A}}(L)$ is generated as an $A$-ring by the images of $\iota_A$ and $\iota_L$, we deduce that $\cV_{\Sscript{A}}(L)\subseteq B$.
\end{invisible}
Therefore in light of the fact that $\cV_{\Sscript{A}}(L)$ is generated as an $A$-ring by the images of $\iota_A$ and $\iota_L$, we deduce that $\cV_{\Sscript{A}}(L)\subseteq B$. Summing up, $\cV_{\Sscript{A}}(L) = B$ and hence $\sigma$ satisfies relations \eqref{eq:simildag}, for all $u\in \cV_{\Sscript{A}}(L)$.
\end{proof}

\begin{lemma}\label{lema:dataI}
Let $(A,L)$ a Lie-Rinehart algebra with anchor map $\omega:L\to \Ders{\K}{A}$ and take $\cU=\cV_{\Sscript{A}}(L)$. Then there is a  bijective correspondence between the following sets of data:
\begin{enumerate}[label=\alph*),ref=\emph{\alph*)}]
\item\label{item:6.4d} morphisms $h:\cV_{\Sscript{A}}(L)\to \ldual{\cH}$ of $A$-rings satisfying \eqref{eq:simildag};
\item\label{item:6.4e} morphisms $\widetilde{h}:L\to \Derk{s}{\cH}{A_{\Sscript{\varepsilon}}}=\lL(\cH)$ of Lie-Rinehart algebras.
\end{enumerate}
\end{lemma}
\begin{proof}
Given $h:\cV_{\Sscript{A}}(L)\to \ldual{\cH}$ as in \ref{item:6.4d}, we define $\widetilde{h}(X):=-h(\iota_{\Sscript{L}}(X))$, for any $X \in L$. The latter is left $A$-linear so that $\widetilde{h}(X)\in\Derk{s}{\cH}{A_{\Sscript{\varepsilon}}}$ in view of the following computation
\begin{align*}
\widetilde{h}(X)(xy) & = -h(\iota_{\Sscript{L}}(X))(xy) \stackrel{\eqref{eq:simildag}}{=} -h(\iota_{\Sscript{L}}(X))(x)h\left(1_{\cV_{\Sscript{A}}(L)}\right)(y)-h\left(1_{\cV_{\Sscript{A}}(L)}\right)(x)h(\iota_{\Sscript{L}}(X))(y) \\
 & = \widetilde{h}(X)(x)\varepsilon(y)+\varepsilon(x)\widetilde{h}(X)(y).
\end{align*}
Since $\iota_{\Sscript{L}}$ is right $A$-linear Lie algebra map and $h$ is an $A$-ring morphism, we get that $\widetilde{h}$ is a right $A$-linear Lie algebra map, more precisely $\widetilde{h}(aX)=\widetilde{h}(X)a$ (since we are taking $L$ as a left $A$-module). Moreover, by Proposition \ref{prop:LR}
\begin{align*}
\omega\left(\widetilde{h}(X)\right)(a) & = \widetilde{h}(X)(t(a)) = -h\left(\iota_{\Sscript{L}}(X)\right)(t(a)) = -h\left(\iota_{\Sscript{L}}(X)\right)(1_\cH a) \stackrel{\eqref{eq:convolution}}{=} -\left(ah\left(\iota_{\Sscript{L}}(X)\right)\right)(1_\cH) \\
 & = -h\left(\iota_A(a)\iota_{\Sscript{L}}(X)\right)(1_\cH) \stackrel{\eqref{eq:simildag}}{=} -\varepsilon\left(\iota_A(a)\iota_{\Sscript{L}}(X)\right) \stackrel{\eqref{eq:compLRalg}}{=} \varepsilon\left(\iota_A\left(\omega(X)(a)\right)-\iota_{\Sscript{L}}(X)\iota_A(a)\right) = \omega(X)(a).
\end{align*}
Conversely, starting with $\widetilde{h}:L\to \Derk{s}{\cH}{A_{\Sscript{\varepsilon}}}=\lL(\cH)$ as in \ref{item:6.4e}. By applying Lemma \ref{lema:Htilda}, we know that there is an $A$-ring map $h:  \VL \to {}^*\cH$  as in \ref{item:6.4d}. The bijectivity of the correspondence between the set of maps as  in \ref{item:6.4d} and those of \ref{item:6.4e},  is easily checked.
\end{proof}

\begin{remark}\label{rem:Integrable}
Let $(A,\cH)$ be a Hopf algebroid and consider the canonical map $g: \cV_{\Sscript{A}}(\lL(\cH)) \to \ldual{\cH}$, which corresponds by Lemma \ref{lema:dataI} to $-\id{\Sscript{\lL(\cH)}}$ (in the above notation this means that $\widetilde{g}=-\id{\Sscript{\lL(\cH)}}$).
By Lemma \ref{lema:LR}, we have an algebra morphism $h :=\theta\circ g : \cV_{\Sscript{A}}(\lL(\cH)) \to \End{\Bbbk}{\cH} $. Let us consider the canonical injective maps
\begin{multline}\label{Eq:V}
i_{\Sscript{A}}: A \longrightarrow  \End{\Bbbk}{\cH}, \quad  \Big( a \mapsto \big[  u \mapsto u t(a) \big] \Big)  \\ i_{\Sscript{\lL(\cH)}}: \lL(\cH) \longrightarrow \End{\Bbbk}{\cH}, \quad \Big( \delta \longmapsto\big[ u \mapsto \uo t(\delta(\udos))  \big] \Big)
\end{multline}
of algebras and Lie algebras, respectively. Denote by $\cV$ the sub $\Bbbk$-algebra of $\End{\Bbbk}{\cH} $ generated by the images of $i_{\Sscript{A}}$ and $i_{\Sscript{\lL(\cH)}}$.  The isomorphism stated in Lemma \ref{lema:LR} shows that $\cV$ is the subalgebra of the algebra of differential operators of $\cH$ generated by $A$ and  the derivations of $\cH$ which are right $\cH$-colinear and kill the source map. Clearly  the maps $i_{\Sscript{A}}$ and
$i_{\Sscript{\lL(\cH)}}$ satisfy the equalities of equation \eqref{eq:compLRalg}. Moreover, $h\circ i_{\Sscript{\lL(\cH)}}=\iota_{\Sscript{\lL(\cH)}}$ and $h\circ i_A=\iota_A$. Therefore, $h : \cV_{\Sscript{A}}(\lL(\cH)) \to \End{\Bbbk}{\cH} $ is the unique morphism arising from the universal property of the enveloping algebroid and, as a consequence, we have that it factors through the inclusion $\cV \subset \End{\Bbbk}{\cH} $. In contrast with the classical case of Lie $\Bbbk$-algebras ($\Bbbk$ is of characteristic zero), it is not clear here if the map $h$ is injective or not. Nevertheless, we believe that the first step in studying the problem of integrating a Lie-Rinehart algebra passes through the analysis of the $A$-algebra map $h$.
\end{remark}


\section{Differentiation as a right adjoint functor of the integration functor}\label{sec:adjunctions}

Now that we collected all the required constructions and notions, we can extend the duality between commutative Hopf algebras and Lie algebras given by the differential functor to the framework of commutative Hopf algebroids, as we claimed at the very beginning of \S\ref{sec:duality}.

\begin{theorem}\label{thm:Ap}
Let us keep the notations of Lemma \ref{lem:I}. There is a natural isomorphism
$$
\xymatrix{  \Hom{\Halgd{A}}{\cH}{\iI'(L)}  \ar@{->}^{\cong}[rr] & &  \Hom{\LieR{A}}{L}{\lL(\cH)}, }
$$
for any commutative Hopf algebroid $(A,\cH)$ and Lie-Rinehart algebra $(A,L)$. That is, the integration functor $\iI':\LieR{A}\to\Halgd{A}^{\text{op}}$ is left adjoint to the differentiation functor $\lL:\Halgd{A}^{\text{op}}\to\LieR{A}$.
\end{theorem}

\begin{proof}
The natural isomorphism is constructed as follows. Given a morphism of commutative Hopf algebroids $\phi: \cH \to \iI'(L)=\rsaft{\VL}$, we have by Lemmas \ref{lema:data} and  \ref{lema:dataI}, the following Lie-Rinehart algebra map: $\cL_{\Sscript{\phi}}: L \to \Derk{s}{\cH}{\Aep}$ sending $X \mapsto \Big[  u \mapsto -\xi(\phi(u)) (\iota_{\script{L}}(X)) \Big]$. As it was shown in those Lemmas, this is a bijective correspondence, which is clearly a natural morphism.
\end{proof}

Notice that, by Theorem \ref{thm:Ap}, we always have the map
\begin{equation}\label{eq:nablamap}
\xymatrix @C=33pt{  \Hom{\Halgd{A}}{\cH}{\iI(L)}  \ar@{->}^{\Hom{\Halgd{A}}{\cH}{\nabla_L}}[rr] & &\Hom{\Halgd{A}}{\cH}{\iI'(L)}  \ar@{->}^{\cong}[r] &  \Hom{\LieR{A}}{L}{\lL(\cH)}, }
\end{equation}
induced by the natural transformation $\nabla=\mb{\hat{\zeta}}\cV_A$ of Lemma \ref{lem:I}. Under some additional hypotheses, this becomes an isomorphism as well.

\begin{theorem}\label{thm:A}
Let $A$ be a commutative algebra for which the map $\mb{\zeta}_{\Sscript{R}}$ of equation \eqref{Eq:zeta} is injective for every $A$-ring $R$ (e.g.,~$A$ is a Dedekind domain). Then there is a natural isomorphism
$$
\xymatrix{  \Hom{\GHalgd{A}}{\cH}{\iI(L)}  \ar@{->}^{\cong}[rr] & &  \Hom{\LieR{A}}{L}{\lL(\cH)}, }
$$
for any commutative Galois Hopf algebroid $(A,\cH)$ and Lie-Rinehart algebra $(A,L)$. That is, the integration functor $\iI:\LieR{A}\to\GHalgd{A}^{\text{op}}$ is left adjoint to the differentiation functor $\lL:\GHalgd{A}^{\text{op}}\to\LieR{A}$.
\end{theorem}

\begin{proof}
First of all, in light of Remark \ref{rem:circGalois} we know that $\iI(L)$ is a commutative Galois Hopf algebroid, whence the statement makes sense. Moreover, since $\GHalgd{A}$ is a full subcategory of $\Halgd{A}$, we have $\Hom{\GHalgd{A}}{\cH}{\iI(L)}=\Hom{\Halgd{A}}{\cH}{\iI(L)}$. In light of Proposition \ref{prop:drwho}, the injectivity of $\mb{\zeta}_{\cV_A(L)}$ implies that $\Hom{\Halgd{A}}{\cH}{\nabla_L}$ is bijective and hence, by Theorem \ref{thm:Ap}, \eqref{eq:nablamap} is a bijection as well.
\end{proof}

\begin{remark}\label{rem:circGalois2}
Observe that in Theorem \ref{thm:A} we may replace the category $\GHalgd{A}$ with the subcategory of $\Halgd{A}$ of all those commutative Hopf algebroids whose canonical map is a split epimorphism of corings, once noticed that $\iI(L)=\rcirc{\cV_A(L)}$ is always in this category because $\can{}\circ \rR(\chi)=\id{\rcirc{R}}$ for every $R$. In addition, the injectivity of $\mb{\zeta}_R$ for every $A$-ring $R$ can be replaced by asking that $\mb{\hat{\zeta}}$ is either injective or a monomorphism of corings on every component. Notice also that these requirements on $\mb{\hat{\zeta}}$ implies that $\chi$ is an isomorphism in view Theorem \ref{thm:drwho}. Hence, by the foregoing, $\can{}$ is invertible and so $\rcirc{R}$ is a Galois coring for every $R$.
\end{remark}

 When we restrict to the category of commutative Hopf algebras, that is, assuming that $A$ is the base field $\Bbbk$ (the source is equal the target in such a case, since all Hopf algebroids are over $\Bbbk$), we have the following well-known adjunction  (recall from Remark \ref{rem:UcUb} \ref{item:UcBc3} that $\iI=\iI'$).
\begin{corollary}
There is a natural isomorphism
$
\Hom{{\sf{CHAlg}}_{\Sscript{\Bbbk}}}{H}{\iI(L)}  \cong \Hom{{\sf{Lie}}_{\Sscript{\Bbbk}}}{L}{\lL(H)},
$
for any commutative Hopf algebra $H$ and Lie algebra $L$. That is, the integration functor $\iI:{\sf{ Lie}}_{\Sscript{\Bbbk}} \to     {\sf{CHAlg}}_{\Sscript{\Bbbk}}^{\Sscript{\text{op}}}$ is left adjoint to the differentiation functor $\lL:{\sf{CHAlg}}_{\Sscript{\Bbbk}}^{\Sscript{\text{op}}} \to {\sf{ Lie}}_{\Sscript{\Bbbk}}$.
\end{corollary}


\section{Separable morphisms of Hopf algebroids}\label{sec:Separable}

We conclude the theoretical part of the paper by finding equivalent conditions to the surjectivity of the morphism $\lL_{\phi}:\Derk{s}{\cK}{A}\to\Derk{s}{\cH}{A}$ induced by a Hopf algebroid map $\phi:(A,\cH)\to (A,\cK)$. Inspired by \cite[Theorem 4.3.12]{Abe}, we also suggest a definition of \emph{separable morphism} between commutative Hopf algebroids based on this characterization.

Let $(A,\cH)$ be a commutative Hopf algebroid. Consider the category $\rmod{\cH}$ as in \S\ref{ssec: Der}. Let us denote by $\cR_\cH:\rmod{\cH}\to \rmod{\cH}$ the functor given by $\cR_\cH(M):=\Derk{s}{\cH}{M}$ on objects and by $\cR_\cH(f)={_*f}$ on morphisms. Let $\cI=\ker{\varepsilon}$ and set $\indec{\cH}:={_{\Sscript{s}}\left(\cI/\cI^2\right)}$. Given a morphism of commutative Hopf algebroids $(\id,\phi):(A,\cK)\to (A,\cH)$, the universal property of the coequalizer \eqref{Diag:CoeqI} applied to $\cK$ gives a unique $A$-module map $\indec{\phi}:\indec{\cK}\to \indec{\cH}$ such that $\indec{\phi}\circ \pi_\cK^s = \pi_\cH^s\circ \phi$. In this way we get a functor
$$
\indec{-}:\Halgd{A}\longrightarrow  \rmod{A}.
$$
Note that the morphism $\phi\tensor{A}\indec{\phi}:\cK\tensor{A}\indec{\cK}\to \cH\tensor{A}\indec{\cH}$ yields a morphism $\Omega_{A}^s(\phi):\Omega_{A}^s(\cK)\to \Omega_{A}^s(\cH)$ by Corollary \ref{coro:Omega}.

\begin{remark}
We know from Proposition \ref{prop:Omega} that $\cR_\cH(M)\cong \hom{\cH}{\Omega_{A}^s(\cH)}{M}$ whence $\cR_\cH$ admits a left adjoint, namely $\cL_\cH=-\tensor{\cH}\Omega_A^s(\cH)$. Notice that $\Omega_A^s(\cH)\cong \cH\tensor{A} \indec{\cH}$ as $\cH$-modules by Corollary \ref{coro:Omega} and $A\tensor{\cH} \Omega_A^s(\cH) \cong \indec{\cH}$ as $A$-modules. Therefore $\cR_\cH$ preserves small colimits if and only if $\Omega_A^s(\cH)$ is finitely generated and projective as $\cH$-module, if and only if $\indec{\cH}$ is finitely generated and projective as $A$-module.
\end{remark}

\begin{invisible}
Assume $P$ is a finitely generated and projective module over $A$, then $\Hom{A}{P}{-}\cong -\tensor{A}P^*$ whence it preserves all small colimits. Conversely, if $\Hom{A}{P}{-}$ preserves all small colimits then it is right exact (hence $P$ is projective) and it preserves direct limits (hence by Stenstroem, Proposition 3.4, $P$ is finitely presented, in particular finitely generated).
\end{invisible}

\begin{theorem}\label{th:sep}
Let $(\id,\, \phi):(A,\cK)\to (A,\cH)$ be a morphism of commutative Hopf algebroids. Assume that $\indec{\cH}$ and $\indec{\cK}$ are finitely generated and projective $A$-modules. The following assertions are equivalent
\begin{enumerate}[label=({\alph*}),ref=\emph{({\alph*})}]
\item\label{5.2item:a} $\indec{\phi}$ is split-injective.
\item\label{5.2item:b} $\lL_{\Sscript{\phi}}$ is surjective.
\item\label{5.2item:c} $\Derk{s}{\phi}{-}:\Derk{s}{\cH}{-}\to \Derk{s}{\cK}{\phi_{\Sscript{*}}(-)}$ is surjective on each component.
\item\label{5.2item:d} $\Derk{s}{\phi}{\cH}:\Derk{s}{\cH}{\cH}\to \Derk{s}{\cK}{\cH}$ is surjective.
\item\label{5.2item:e} $\cH\tensor{\cK}\Omega_{A}^s(\cK)\to \Omega_{A}^s(\cH):\,h\tensor{\cK}w\mapsto h\Omega_{A}^s(\phi)(w)$ is split-injective.
\end{enumerate}
\end{theorem}

\begin{proof}
To prove the equivalence between \ref{5.2item:a} and \ref{5.2item:b}, observe that $\indec{\phi}$  is a split-monomorphism of $A$-modules if and only if $\ldual{(\indec{\phi})}:\ldual{(\indec{\cH})}\to \ldual{(\indec{\cK})}$ is a split-epimorphism of $A$-modules. However, as $\indec{\cK}$ is finitely generated and projective, $\ldual{(\indec{\cK})}$ is finitely generated and projective as well and hence requiring that $\ldual{(\indec{\phi})}$ splits is superfluous. By Lemma \ref{lema:LR}, the map $\ldual{(\indec{\cH})}\to\cL(\cH)$ which assigns to every $f$ the composition $f\circ \pi_\cH^s$ is an isomorphism of $A$-modules. In view of the relation $\indec{\phi}\circ \pi_\cK^s = \pi_\cH^s\circ \phi$ and of the definition \eqref{eq:defL} of $\cL_{\Sscript{\phi}}$ we have that the following diagram commutes
\begin{equation*}
\xymatrix@C=45pt{
\ldual{\big(\indec{\cH}\big)} \ar[d]_-{\ldual{\big(\indec{\phi}\big)}} \ar[r]^-{\ldual{\big(\pi_\cH^s\big)}} & \cL(\cH) \ar[d]^-{\cL_{\Sscript{\phi}}} \\
\ldual{(\indec{\cK})}\ar[r]_-{\ldual{\big(\pi_\cK^s\big)}} & \cL(\cK)
}
\end{equation*}
so that $\ldual{(\indec{\phi})}$ is an epimorphism of $A$-modules if and only if $\cL_{\Sscript{\phi}}$ is. The implications from \ref{5.2item:c} to \ref{5.2item:b} and \ref{5.2item:d} are obtained evaluation the natural transformation $\Derk{s}{\phi}{-}$ on $\Aep$ and $\cH$ respectively. To prove that \ref{5.2item:b} implies \ref{5.2item:c}, consider the following diagram for every $M\in\rmod{\cH}$.
\begin{equation*}
\xymatrix{
M\tensor{A}\Derk{s}{\cH}{\Aep} \ar[r]^-{\id{}} \ar@{.>}[d]_-{\varsigma_\cH} & M\tensor{A}\Der{A}{\Hs}{\At} \ar[r]^-{\cong}_-{\eqref{Eq:natural}} & M\tensor{A} \Hom{A}{\indec{\cH}}{\At} \ar[d]^-{\cong}  \\
\Derk{s}{\cH}{M} \ar[r]_-{\cong}^-{\eqref{eq:isos}} & \Der{A}{\Hs}{\Mt} \ar[r]_-{\cong}^-{\eqref{Eq:natural}}  & \Hom{A}{\indec{\cH}}{\Mt}
}
\end{equation*}
The undashed vertical arrow is the map $m\tensor{A}f\mapsto \left[q\mapsto mtf(q)\right]$ which is invertible because $\indec{\cH}$ is finitely generated and projective. As a result we get the dashed vertical isomorphism $\varsigma_\cH$ given by $m\tensor{A}\delta \mapsto \left[u\mapsto mu_1t\delta(u_2)\right]$ which is clearly $\cH$-linear (with respect to the action of $\cH$ on $M$) and natural in $\cH$. This naturality implies that $\Derk{s}{\phi}{M}$ is an epimorphisms whenever $\cL_\phi$ is. To show the implication from \ref{5.2item:d} to \ref{5.2item:b}, notice that the above naturality implies in particular that $\cH\tensor{A}\cL_\phi$ is an epimorphism of $A$-modules. Now since $\cL_\phi$ can be recovered from $\cH\tensor{A}\cL_\phi$ by applying the functor $A\tensor{\cH}-$, it is an epimorphism as well. Finally, observe that the map in \ref{5.2item:e} can be easily identified with $\cH\tensor{\cK}\indec{\phi}$ since $\Omega_{A}^s(\cK)\cong \cK\tensor{A}\indec{\cK}$ and analogously for $\cH$. Now it is clear that \ref{5.2item:a} implies \ref{5.2item:e} and the other implication follows by applying the functor $A\tensor{\cH}-$, and this finishes the proof.
\end{proof}

\begin{remark}
Assume that $\cH$ and $\cK$ are ordinary commutative Hopf algebras over $A=\K$ and also integral domains such that $G:=\Calg(\cH,\K)$ and $E:=\Calg(\cK,\K)$ are connected affine algebraic $\K$-groups. Notice that any one of these algebras is smooth and then both  $Q(\cH)$ and $Q(\cK)$ are finite-dimensional $\Bbbk$-vector spaces.  Let $\varphi:=\Calg(\phi,\K):G\to E$. By resorting to the notation of \cite[3.1]{Abe}, we have that $d\varphi=\lL_\phi$. Therefore, in view of \cite[Theorem 4.3.12]{Abe}, the separability of the morphism $\varphi$ can be rephrased at the level of commutative Hopf algebroids by requiring that the morphism $\phi$ satisfies the equivalent conditions of Theorem \ref{th:sep}. In this way, a morphism of commutative Hopf algebroids with smooth total algebras may be called a \emph{separable morphism} when it satisfies one of the equivalent conditions of Theorem \ref{th:sep}.
\end{remark}


\section{Some applications and examples}\label{sec:Apendix}

\newcommand{\bo}[1]{\boldsymbol{#1}}
\newcommand{\x}[1]{x_{\Sscript{#1}}}
\newcommand{\y}[1]{y_{\Sscript{#1}}}
\newcommand{\Cc}{\mathbb{C}}
\newcommand{\kn}{\bo{k}_{(k_{\Sscript{1}},\, \cdots,\, k_{\Sscript{n}})}}
\newcommand{\KK}[1]{\bo{k}(#1)}
\newcommand{\lrfrac}[2]{\left( \frac{#1}{#2} \right)}
\newcommand{\affLine}[1]{\mathbb{A}_{\Sscript{#1}}^{1}}
\newcommand{\Derc}[2]{{\rm Der}_{\Sscript{\Cc}}^{\Sscript{s}}(#1, #2_{\Sscript{\varepsilon}})}

\newcommand{\balpha}{\bo{\alpha}}
\newcommand{\bbeta}{\bo{\beta}}
\newcommand{\blambda}{\bo{\lambda}}
\newcommand{\bsigma}{\bo{\sigma}}
\newcommand{\bgamma}{\bo{\gamma}}
\newcommand{\bdelta}{\bo{\delta}}
\newcommand{\brho}{\bo{\varrho}}
\newcommand{\bpartial}{\bo{\partial}}
\newcommand{\biota}{\bo{\iota}}
\newcommand{\bV}{\bo{V}}
\newcommand{\bW}{\bo{W}}
\newcommand{\bv}{\bo{v}}
\newcommand{\bw}{\bo{w}}
\newcommand{\Hx}{\cH_{\Sscript{x}}}
\newcommand{\Kx}{\Bbbk_{\Sscript{x}}}
\newcommand{\Kep}{\Bbbk_{\Sscript{\varepsilon}}}
\newcommand{\LHx}{\lL(\cH)_{\Sscript{x}}}

 This section illustrate some of our theoretical construction elaborated in the previous sections.

\subsection{The isotropy Lie algebra as the Lie algebra of the isotropy Hopf algebra}\label{ssec:isotropy}
In analogy with the Lie groupoid theory, we will show here that the isotropy Lie algebra of the Lie-Rinehart algebra of a given Hopf algebroid coincides, up to a canonical isomorphism, with the Lie algebra of the isotropy Hopf algebra.

Let $(A,\cH)$ be a commutative Hopf algebroid whose character groupoid is not empty. This amounts to the assumption $A(\Bbbk)=\Calg{(A,\Bbbk)} \neq \emptyset$, that is, $\Calg{(A, -)}$ admits $\Bbbk$-points. Take a point $x \in A(\Bbbk)$, and consider the isotropy Hopf $\Bbbk$-algebra $(\Bbbk, \Hx)$ at the point $x$. By definition, see \cite[Definition 5.1]{ElKaoutit:2015} and \cite[Example 1.3.5]{LaiachiGomez}, $\Hx=\Kx \tensor{A}\cH\tensor{A}\Kx$ is the base extension Hopf algebroid of $(A,\cH)$ along the algebra map $x: A \to \K$ (the notation $\Kx$ means that we are considering $\K$ as an $A$-algebra via $x$). The Lie algebra of the commutative Hopf algebra $(\Bbbk, \Hx)$ is by definition the $\Bbbk$-vector space $\Der{\K}{\Hx}{\K_{\Sscript{x \varepsilon}}}$.

On the other hand, for a given point $x \in A(\K)$, we set
\begin{equation}\label{Eq:LHx}
\LHx: = \Big\{ \delta \in {\rm Der}^{\Sscript{s}}_{\Sscript{\K}}(\cH, \K_{\Sscript{x \varepsilon}}) |\,\, \delta \circ t\,=\,0  \Big\}\footnote{By abuse of notation we employ  $\lL(\cH)$ the Lie-Rinehart algebra of $(A,\cH)$ in this equation. However,  this can be justified using  the identification  of the  $A$-module of global sections  $\Gamma(\xX)$  with  $\lL(\cH)$, as  stated in Proposition \ref{prop:Gamma}.}
\end{equation}
These are vectors in the fibre $\xX(\K)_{\Sscript{x}}$ of the vector bundle $\xX(\K)$ at the point $x$, which are killed by the anchor \eqref{eq:LR}; in the notation   of Appendix \ref{ssec:FA} and equation \eqref{Eq:xX}, this is the vector space $\xX^{\Sscript{\ell}}(\K)_{\Sscript{x}}$.  The vector space $\LHx$, $x \in A(\K)$, is referred to as \emph{the isotropy Lie algebra of the Lie algebroid} $\lL(\cH)$. The terminology is justified by the following result.

\begin{proposition}\label{prop:isotropy}
Let $(A,\cH)$ be a commutative Hopf algebroid over $\Bbbk$ with $A(\K) \neq \emptyset$. Then
\begin{enumerate}
\item[(i)] for a given point $x \in A(\K)$, the $\K$-vector space $\LHx$ of equation \eqref{Eq:LHx} admits a structure of Lie algebra whose bracket is given by
$$
[\delta, \delta']: \cH \longrightarrow \K_{\Sscript{x \varepsilon}}, \qquad \Big( u \longmapsto \big( \delta(u_1)\delta'(u_2) - \delta'(u_1)\delta(u_2) \big) \Big);
$$
\item[(ii)] there is an isomorphism of Lie algebras given by
\begin{equation}\label{eq:nablaLX}
\nabla: \LHx \longrightarrow \lL(\Hx)=\Der{\K}{\Hx}{\K_{\Sscript{x \varepsilon}}}, \qquad \Big(  \delta \longmapsto \Big[ 1\tensor{A}u \tensor{A} 1 \mapsto \delta(u)\Big]\Big)
\end{equation}
\end{enumerate}
\end{proposition}

\begin{proof}
$(ii)$. The map $\nabla$ is a well-defined $\K$-linear morphism, since any vector in $\LHx$ is an $A$-linear map with respect to both source and target. The inverse of $\nabla$ sends any derivation $\gamma \in \lL(\Hx)$ to the derivation $\gamma \pi_{\Sscript{x}}$, where $\pi_{\Sscript{x}}: \cH \to \Hx$ is the canonical algebra map sending $u \mapsto 1_{\Sscript{\Bbbk_x}}\tensor{A}u \tensor{A} 1_{\Sscript{\Bbbk_x}}$. Now, it is easy to check that the bracket of $\lL(\Hx)$ induces the one in $(i)$ via $\nabla$.
\begin{invisible}[Old proof]
$(i)$. The bracket is a well defined  $\K$-linear map, since any vector in $\LHx$ is an $A$-linear map with respect to both source and target.  The fact that it lands in $\LHx$ and satisfies the Jacobi identity are routine  verifications and left to the reader. As for the proof of part $(ii)$, one employ the previous argument to see that $\nabla$ is a well defined $\K$-linear map. Its compatibility with the brackets is immediate. The inverse of $\nabla$ sends any derivation $\gamma \in \lL(\Hx)$ to the derivation $\gamma \upsilon$, where $\upsilon: \cH \to \Hx$ is the canonical algebra map sending $u \mapsto 1\tensor{A}u \tensor{A} 1$.
\end{invisible}
\end{proof}

\subsection{The Lie-Rinehart algebra of Malgrange's Hopf algebroids}\label{ssec:Umemura}
In this final subsection we compute the Lie-Rinehart algebras of some Hopf algebroids which arise from  differential Galois theory over differential Noetherian algebras.  Inspired by  \cite{Malgrange:2000, Malgrange:2001}, \cite{Morikawa/Umemura:2009} and \cite{Umemura:2009}, some of these Hopf algebroids were introduced and described in \cite{LaiachiGomez}.  We also construct a morphism from the  Lie-Rinehart algebra of one those Hopf algebroids,  to the one arising from the global smooth sections of the Lie algebroid of the invertible jets groupoid attached to this Hopf algebroid.

Let us consider the polynomial complex algebra  $A=\mathbb{C}[X]$ and $\{x_0,y_n\mid n\in \N\}$ a set of indeterminates. For a given element $p \in A$, we denote by $\partial p$ its derivative, where  $\partial:= \partial/\partial X$ is the differential of $A$. Consider the Hopf algebroid $(A,\cH)$ over $\mathbb{C}$, where
$$
\cH : = \Cc[x_0,y_0,y_1,\cdots,y_n,\cdots,\frac{1}{y_1}],
$$
is the polynomial $\Cc$-algebra, and  where the structure maps are given as follows.

The source and the target are given by:
\begin{equation}\label{Eq:stH}
s: A \to \cH, \;\Big( X \mapsto x_0:=x  \Big) \quad \text{ and } \quad t: A \to \cH, \; \Big( X \mapsto y_0:=y  \Big)
\end{equation}
The comultiplication is:
\begin{equation}\label{Eq:DH}
\begin{gathered}
\xymatrix@R=0pt{  {}_{\Sscript{s}}\cH_{\Sscript{t}} \ar@{->}^-{\Delta}[rr]  & & {}_{\Sscript{s}}\cH_{\Sscript{t}} \tensor{A} {}_{\Sscript{s}}\cH_{\Sscript{t}}  } \\  \Delta(x)\,\,=\,\, x\tensor{A}1, \quad \Delta(y)\,\,=\,\, 1\tensor{A} y,  \\
\Delta(y_n)\,\,=\,\,  \sum\limits_{\substack{(k_{1},\, k_{2}, \cdots,\, k_{n}) \\ \\ k_{1}+2k_{2}+ \cdots+nk_{n}=n }} \frac{n!}{k_1!\, \cdots\, k_n!} \Big( \lrfrac{y_1}{1!}^{k_1}\, \lrfrac{y_2}{2!}^{k_2}\, \cdots\, \lrfrac{y_n}{n!}^{k_n} \Big) \tensor{A} y_{k_1+k_2+\cdots+k_n}, \; \text{ for  } n \geq 1.
\end{gathered}
\end{equation}
(see \cite{LaiachiGomez} for the symbols in the sum).  Thus, for $n=1, 2, 3, 4$, the image by $\Delta$ of the variables $y_n$'s reads  as follows:
\begin{multline*}
\Delta(y_1)=y_1\tensor{A}y_1, \quad \Delta(y_2)=y_2\tensor{A}y_1+y_1^2\tensor{A}y_2, \quad \Delta(y_3)=y_3\tensor{A}y_1+ 3y_1y_2\tensor{A}y_2+y_1^3\tensor{A}y_3, \\
\Delta(y_{4})= y_{4}\tensor{A}y_{1} + 4 y_{3}y_{1}\tensor{A}y_{2} + 6 y_{2}y_{1}^{2}\tensor{A}y_{3} + 3y_{2}\tensor{A}y_{2} + y_{1}^{4}\tensor{A}y_{4},   \quad \cdots . \hspace{3cm}
\end{multline*}
\begin{invisible}
The antipode is the algebra map:
\begin{equation}\label{Eq:SH}
\begin{gathered}
\xymatrix@R=0pt{  {}_{\Sscript{s}}\cH_{\Sscript{t}} \ar@{->}^-{\sS}[rr]  & & {}_{\Sscript{t}}\cH_{\Sscript{s}}   } \\
\cS(x) \,=\, y,\quad \cS(y)\,=\, x, \quad  \cS(y_1)\,=\, y_1^{-1}, \quad \text{ and }
\\  \cS(y_n) \,=\, \sum\limits_{\substack{(k_{1},\, k_{2}, \cdots,\, k_{n}) \,\neq \,  (n,0,\cdots,0) \\ \\ k_{1}+2k_{2}+ \cdots+nk_{n}=n }} -\frac{n!}{k_1!\, \cdots\, k_n!} \, \cS\big(  y_{k_1+k_2+\cdots+k_n}\big)\,  \Big( \lrfrac{y_1}{1!}^{k_1-n}\, \lrfrac{y_2}{2!}^{k_2}\, \cdots\, \lrfrac{y_n}{n!}^{k_n} \Big), \text{ for  } n \geq 2,
\end{gathered}
\end{equation}
for instance,
\begin{multline*}
\cS(y_1)=y_1^{-1} , \quad \cS(y_2)=-y_2y_1^{-3}, \quad  \cS(y_3)=-y_3y_{1}^{-4} + 3y_{2}^{2}y_{1}^{-5}, \\ \cS(y_{4}) = -y_{4}y_{1}^{-5} + 10 y_{3}y_{2}y_{1}^{-6}- 15 y_{2}^{3}y_{1}^{-7},  \quad \cdots . \hspace{1cm}
\end{multline*}
\end{invisible}
Lastly the counit is given by:
\begin{equation}\label{Eq:EH}
\begin{gathered}
\xymatrix@R=0pt{  {}_{\Sscript{s}}\cH_{\Sscript{t}} \ar@{->}^-{\varepsilon}[rr]  & &  A  } \\
 \varepsilon(x)\,=\, X,\quad \varepsilon(y)\,=\, X, \quad  \varepsilon(y_n)\,\,=\,\, \delta_{1,\, n}, \quad \text{ for every } \; n \geq 1.
\end{gathered}
\end{equation}
An explicit formula for the antipode $\cS:{}_{\Sscript{s}}\cH_{\Sscript{t}} \to {}_{\Sscript{t}}\cH_{\Sscript{s}}$ can be found in \cite[\S5.6]{LaiachiGomez}.

\begin{proposition}\label{prop:H1}
Consider the above Hopf algebroid $(A,\cH)$ over the complex numbers. Then the Lie-Rinehart algebra $\lL(\cH)$ of $(A,\cH)$ has underlying $A$-module the free module $A^{\mathbb{N}}$ whose anchor map is
$$
\omega: A^{\mathbb{N}} \longrightarrow {\rm Der}_{\Sscript{\mathbb{C}}}(A), \quad \Big( \fk{a}:=(a_{n})_{n\, \in \, \mathbb{N}} \longmapsto \Big(  p \mapsto a_{0}\partial p \Big)  \Big)
$$
and the bracket is defined as follows. For sequences $\fk{a}$ and $\fk{b}$ as above, the sequence $[\fk{a}, \fk{b}]$ is given by:
\begin{gather*}
\big[  \fk{a}, \fk{b} \big]_{0} = a_{0}\partial b_{0} - b_{0}\partial a_{0}, \quad  \big[  \fk{a}, \fk{b} \big]_{1} = a_{0}\partial b_{1} - b_{0}\partial a_{1}, \quad \big[  \fk{a}, \fk{b} \big]_{2} = a_{2}b_{1} - b_{2}a_{1}+a_{0}\partial b_{2} - b_{0}\partial a_{2}, \\
\big[  \fk{a}, \fk{b} \big]_{n} \,\, =\,\,  \sum_{i=1}^{n} \binom{n}{i} \big( a_{i}b_{n-i+1}-b_{i}a_{n-i+1}\big) \,+\, \big( a_{0}\partial b_{n} - b_{0}\partial a_{n}\big),\quad \text{ for }\, n \geq 3.
\end{gather*}
\end{proposition}
\begin{proof}
Let $\delta$ be an element in $\lL(\cH)=\Derc{\cH}{A}$, then $\delta$ is entirely determined by the sequence of polynomials $(\delta(x_{0}), \delta(y_{0}), \delta(y_{1}), .....)$. Since we know that $\delta(x_{0})=0$, we have a sequence
$$
(\delta(y_{0}), \delta(y_{1}), \delta(y_{2}), .....) \, \in \,  A^{\Sscript{\mathbb{N}}}
$$
Namely, we know that any such $\delta$ satisfies the following equalities:
\begin{equation}\label{Eq:deltas}
\delta(y_{1}^{-1})=-\delta(y_{1}), \quad \delta(y_{i}^{j})=0, \text{ for every } i , j\geq 2, \; \text{ and }\; \delta(p(y_{0}))= \delta(y_{0}) \frac{\partial p(x_{0})}{\partial x_{0}}.
\end{equation}
The last equality gives us the anchor map. Now, for the bracket we need to involve the comultiplication of equation \eqref{Eq:DH} and the formula of equation \eqref{Eq:bracket}. For lower cases, that is, for $n=1,2$, one uses directly these formulae. As for $n\geq 3$, one should observe, using equations \eqref{Eq:deltas}, that when applying  the rule \eqref{Eq:bracket} to the comultiplication \eqref{Eq:DH}, the only terms which survive in the sum are the summands corresponding the following $n$-tuples
$$
(n,0,\ldots, 0),\quad (n-i,0, \ldots,0,\underset{i\mathrm{-th}}{\underbrace{1}},0,\ldots, 0),\quad \text{for}\;2\leq i\leq n,
$$
which give the summands claimed in the bracket $\big[  \fk{a}, \fk{b} \big]_{n}$.
\end{proof}

The $\Cc$-algebra $\cH$ is in fact a differential  algebra, whose differential is given by:
\begin{equation}\label{Eq:DiffH}
\begin{gathered}
\xymatrix@R=0pt{  \cH \ar@{->}^-{\bdelta}[rr]  & &  \cH  }
\\ \bdelta(x)\,=\, 1,\quad \bdelta(y)\,=\, y_1, \quad \bdelta(y_n)\,=\,y_{n+1}, \; \text{ for } n \geq 1.
\end{gathered}
\end{equation}
Thus, we have
$$
\bdelta\,\,=\,\, \displaystyle \frac{\partial}{\partial x}  \,+\, \sum_{i=0}^{\infty} y_{i+1} \frac{\partial}{\partial y_i}.
$$

A \emph{Malgrange's Hopf algebroid over $\Cc$  with base $A$} is a Hopf algebroid of the form $(A,\cH/\cI)$, where $\cI$ is a Hopf ideal which is also a differential ideal (i.e., $\bdelta(\cI) \subseteq \cI$). For instance, the ideal $\cI=\langle y_{n}\rangle_{n\geq 2}$, is clearly a differential ideal and $\cH/\cI \cong \Cc[x, y, z^{\pm 1}]$, which is a Hopf algebroid with base $A$ and grouplike elements $z^{\pm 1}$.  It can be identified with the polynomial algebra $(A\tensor{\Cc}A)[z^{\pm 1}]$, whose presheaf of groupoids is the induced groupoid of the multiplicative group by the affine line (see \cite{ElKaoutit:2015} for this general construction).

The following corollary is immediate.
\begin{corollary}\label{coro:subLR}
Let $(A, \cH/\cI)$ be a Malgrange Hopf algebroid with base $A$. Then the Lie-Rinehart algebra $\lL(\cH/\cI)$ is a sub-Lie-Rinehart algebra of $\lL(\cH)$. Precisely, an element $\delta \in \lL(\cH)$ belongs to $\lL(\cH/\cI)$, if and only if $\delta(\cI)=0$.
\end{corollary}

For instance, by Proposition \ref{prop:H1}, we have that $\lL(\cH/\cI)=A\times A$, where $\cI=\langle y_{n}\rangle_{n\geq 2}$, is the Lie-Rinehart algebra with anchor
$
(a_{0}, a_{1}) \mapsto (  p \mapsto a_{0}\partial p )
$
and the bracket is given by
$$
[(a_{0},a_{1}), (b_{0},b_{1})]\, \, = \, \, \big(a_{0}\partial b_{0} - b_{0}\partial a_{0}, \,  a_{0}\partial b_{1} - b_{0}\partial a_{1}\big).
$$

\begin{remark}\label{rem:limit}
The Hopf algebroid $(A,\cH)$ is the a direct limit of the Hopf algebroids $(A,\cH_{\Sscript{r}})$, $r \in \mathbb{N}$, where $\cH_{\Sscript{r}}$ is the subalgebra of $\cH$ generated up to the variable $y_{r}$, that is, we have $\cH=\varinjlim{\cH_{\Sscript{r}}}$. Applying the differentiation functor $\lL$, we obtain a projective limit of Lie-Rinehart algebras $\lL(\cH)=\varprojlim{\lL(\cH_{\Sscript{r}})} $.
\end{remark}

In the remainder of this subsection we will relate the Lie-Rinehart algebra of $(A,\cH)$ and the Lie-Rinehart algebra of the (polynomial) global sections of the Lie-groupoid attached to the varieties associated to the pair of algebras $(A,\cH)$. To this end,
consider the invertible jet groupoid attached to $(A, \cH)$. This, by definition \cite{Malgrange:2001}, is the Lie groupoid $(\jJ_{*}(\affLine{\Cc}) , \affLine{\Cc} )$, where $\affLine{\Cc}$ is the complex affine line and $ \jJ_{*}(\affLine{\Cc})  \subseteq \affLine{\Cc} \times (\affLine{\Cc})^{\mathbb{N}}$ is defined by the points $(x_{0}, y_{0}, y_{1}, \cdots, y_{n}, \cdots) \in \affLine{\Cc} \times (\affLine{\Cc})^{\mathbb{N}}$ with $y_{1}\neq 0$.  In other words,  this groupoid is the character groupoid of the Hopf algebroid $(A, \cH)$, see \cite{ElKaoutit:2015} for this definition.  Denote by $\cE$ the Lie algebroid of this Lie groupoid (see Appendix \ref{ssec:LA-LG}) below).
Then,  one can show that there is a morphism $\lL(\cH) \to \Gamma(\cE)$  of Lie-Rinehart algebras, where $\Gamma(\cE)$ is the $A$-module of global sections of the Lie algebroid $\cE$. This claim will be achieved in the forthcoming steps.

First let us denote by
$$
\xymatrix@C=65pt{ \cH(\Cc)={\rm CAlg}_{\Sscript{\Cc}}(\cH, \Cc) \, \simeq \, \jJ_{*}(\affLine{\Cc})\; \;\ar@<1ex>@{->}|(.5){\scriptstyle{s^{*}}}[r] \ar@<-1ex>@{->}|(.5){\scriptstyle{t^{*}}}[r] & \ar@{->}|(.4){ \scriptstyle{\varepsilon^{*}}}[l]  \; \; A(\Cc)={\rm CAlg}_{\Sscript{\Cc}}(A,\Cc)\, \simeq \, \affLine{\Cc}, }
$$
the structure maps of this groupoid, where  the source and the target are, respectively, the first and second projections, and the identity map coincides with the map $x \mapsto (x, x, 1, 0, \cdots)$, see \cite{Malgrange:2001}. Here we are considering $\cH(\Cc)$ and $A(\Cc)$ as algebraic varieties whose ring of polynomial functions coincide with $\cH$ and $A$, respectively. In this way the elements of $\cH$ and $A$ are considered as polynomial functions from $\cH(\Cc)$  and $A(\Cc)$ to $\Cc$, respectively.

We know that the fibers of $\cE$ are of the form $\ker{T_{x} s^{*}}$, for $x \in \affLine{\Cc}$.  Specifically, given  a point $x \in \affLine{\Cc}$, we identify it with the associated algebra map  $\bara{x}: A \to \Cc$ sending $X\mapsto x$. In this way, the notation $\Cc_{\Sscript{{x}}}$ stands for $\Cc$ considered as an extension algebra of $A$  via $\bara{x}$, and the identity arrow of the object $x$ is $\varepsilon^{*}(x)=\bara{x} \varepsilon : \cH \to \Cc$. The same notations will be employed for $\jJ_{*}(\affLine{\Cc})$.
Now, for any point $x \in \affLine{\Cc}$, a derivation $d$  in the vector space $\ker{T_{x} s^{*}}$, is nothing but an element $d \in  {\rm Der}_{\Sscript{\Cc}}(\cH, \Cc_{\Sscript{\varepsilon*(x)}})$ such that $ ds=0$. Therefore, we have the following  identifications of vector spaces:
$$
\ker{T_{x} s^{*}}\, =\, \{ d \in {\rm Der}_{\Sscript{\Cc}}(\cH, \Cc_{\Sscript{\varepsilon*(x)}}) |\, d s=0 \} \, \,  = \, \,
{\rm Der}_{\Sscript{\Cc}}^{\Sscript{s}}(\cH, \Cc_{\Sscript{{x}\varepsilon}})\,=\, \xX(\Cc)_{\Sscript{x}}, \; \text{ for every } x \in \affLine{\Cc},
$$
where the $\xX(\Cc)_{\Sscript{x}}$'s are the fibers of the presheaf of equation \eqref{Eq:xX} at the base field $\Cc$. This gives us the identification of vector bundles $\cE= \xX(\Cc)$.

On the other hand, any (polynomial) section of the vector bundle $\xX(\Cc)$ can be extended ``uniquely'', as follows, to a (polynomial) section of the vector bundle $\cup_{\Sscript{g \, \in \, \cH(\Cc)}} {\rm Der}^{\Sscript{s}}_{\Sscript{\Cc}}(\cH, \Cc_{\Sscript{g}})$. This extension is the same as the one given in Proposition \ref{prop:OXY} of the Appendix. Take a section $\{\delta_{x}\}_{\Sscript{x\, \in \, \affLine{\Cc}}}$ of $\xX(\Cc)$\footnote{Here, we are assuming that, for every $u \in \cH$,   the function $x \mapsto \delta_{x}(u)$ is polynomial, where  $x \in \affLine{\Cc}$, or  equivalently, each of  the functions $x \mapsto \delta_{x}(x_{0}), \,  \delta_{x}(y_{0}),\,  \delta_{x}(y_{1}),\,  \cdots,\,  \delta_{x}(y_{n}),\, \cdots$,
 is polynomial.}, we set
$$
\td{\delta}_{g}: \cH \longrightarrow \Cc_{g}, \quad \Big(  u \longmapsto g(u_{1}) \, \delta_{gt}(u_{2})\Big), \quad \text{ for every } g \in \cH(\Cc).
$$
These are called  \emph{left invariant sections tangent to the fiber of $s$}.
For a fixed polynomial function $u \in \cH$, we have a polynomial function $\td{\delta}_{-}(u) :\cH(\Cc) \to \Cc$ sending $g \mapsto \td{\delta}_{g}(u)$, which we identify with its image in $\cH$. This function satisfies the following equalities\footnote{For the sake of clearness, $a\delta:x\mapsto x(a)\delta_x$ and $u\widetilde{\delta}_{-}(v):g\mapsto g(u)\widetilde{\delta}_{g}(v)$.}:
\begin{equation}\label{Eq:barril}
\td{\delta}_{-}(s(a))\,\,=\,\, 0,\qquad  \td{(a \delta)}_{-}(u)\,\, =\,\, t(a) \td{\delta}_{-}(u),\qquad  {x}\varepsilon\big( \td{\delta}_{-}(u)\big)\,\, =\,\, \delta_{x}(u),
\end{equation}
for every $a \in A$, $u \in \cH$ and $x \in  \affLine{\Cc}$. Furthermore, there is a derivation of $\cH$, defined by $u \mapsto \td{\delta}_{-}(u)$. Namely, for every point $g \in \cH(\Cc)$ and two polynomial functions $u, v \in \cH$,  we have
\begin{eqnarray*}
\td{\delta}_{-}(uv)(g) &=& \td{\delta}_{g}(uv)\, \, =\, \, g(u_{1}v_{1})\delta_{gt}(u_{2}v_{2}) \\ &=&  g(u_{1})g(v_{1})\Big( gt\varepsilon(u_{2})\delta_{gt}(v_{2}) + \delta_{gt}(u_{2})gt\varepsilon(v_{2}) \Big)\\ &=&  g(u_{1})g(v_{1})\, gt\varepsilon(u_{2}) \delta_{gt}(v_{2}) + g(u_{1})g(v_{1})\,\delta_{gt}(u_{2})gt\varepsilon(v_{2})  \\ &=&  g(u)\, g(v_{1})\delta_{gt}(v_{2}) + g(u_{1})\delta_{gt}(u_{2})\,g(v)  \\ &=&  g(u)\, \td{\delta}_{-}(v)(g) +  \td{\delta}_{-}(u)(g) \, g(v)  \\ &=&  u(g)\, \td{\delta}_{-}(v)(g) +  \td{\delta}_{-}(u)(g) \, v(g)
\end{eqnarray*}
Therefore, we have
\begin{equation}\label{Eq:calor}
\td{\delta}_{-}(uv) = u \td{\delta}_{-}(v) + \td{\delta}_{-}(u) v.
\end{equation}

Next, we describe the anchor map and the bracket of $\Gamma(\cE)$. Given a section $\delta \in \Gamma(\cE)$, its  anchor  at a given polynomial $a \in A$,    is defined as the polynomial function $\omega(\delta)(a): \affLine{\Cc} \to \Cc$ sending $x \mapsto \delta_{x} (t(a))$. As for the bracket, taking two sections $\delta, \gamma$, we set the section
$x \mapsto [\delta, \gamma]_{x}$, defined by
$$
[\delta, \gamma]_{x}: \cH \longrightarrow 	\Cc_{x \varepsilon}, \quad \Big(  u \longmapsto \big(\delta_{x}(\, \td{\gamma}_{-}(u)\, ) - \gamma_{x}(\, \td{\delta}_{-}(u)\, ) \big) \Big).
$$
Take $u, v \in \cH$, we compute
\begin{eqnarray*}
[\delta, \gamma]_{x}(uv) &=& \delta_{x}(\, \td{\gamma}_{-}(uv)\, ) - \gamma_{x}(\, \td{\delta}_{-}(uv)\, ) \\
&\overset{\eqref{Eq:calor}}{=}& \delta_{x}\big(u \td{\gamma}_{-}(v)+  \td{\gamma}_{-}(u)v\big) - \gamma_{x}\big(u\td{\delta}_{-}(v) + \td{\delta}_{-}(u)v \big) \\
&=& x\varepsilon(u)\delta_{x}(\, \td{\gamma}_{-}(v)\, )+\delta_{x}(u) x\varepsilon(\td{\gamma}_{-}(v)) + x\varepsilon(\td{\gamma}_{-}(u))\delta_{x}(v) +  \delta_{x}(\, \td{\gamma}_{-}(u)\, )x\varepsilon(v) \\
&\,\, & -x\varepsilon(u) \gamma_{x}(\, \td{\delta}_{-}(v)\, ) - \gamma_{x}(u)x\varepsilon(\td{\delta}_{-}(v))- x\varepsilon(\td{\delta}_{-}(u))\gamma_{x}(v) - \gamma_{x}(\, \td{\delta}_{x}(u)\,) x\varepsilon(v) \\ &\overset{\eqref{Eq:barril}}{=}& x\varepsilon(u)[\delta, \gamma]_{x}(v) + [\delta, \gamma]_{x}(u) x\varepsilon(v).
\end{eqnarray*}
Thus   $[\delta, \gamma]_{x} \in {\rm Der}_{\Sscript{\Cc}}^{\Sscript{s}}(\cH, \Cc_{\Sscript{{x}\varepsilon}})$. It is not hard now to check that this bracket endows $\Gamma(\cE)$ with a structure of Lie algebra  and it is compatible with the anchor map, that is, satisfies equation \eqref{Eq:Aseti}. This completes the Lie-Rinehart algebra structure of $\Gamma(\xX(\Cc)) \, =\,  \Gamma(\cE)$.

The desired morphism of Lie-Rinehart algebras $\lL(\cH) \to \Gamma(\cE)$, is now deduced as follows. Using the isomorphism of Proposition \ref{prop:Gamma} in conjunction with the canonical map
$$
 \Gamma(\xX) \longrightarrow \Gamma\big(\xX(\Cc)\big), \qquad \Big( \tau \longmapsto \tau_{\Cc}  \Big)
$$
of Lie-Rinehart algebras, we obtain  a morphism $\lL(\cH) \cong  \Gamma(\xX) \longrightarrow \Gamma\big(\xX(\Cc)\big) = \Gamma(\cE)$ of Lie-Rinehart algebras.

\begin{remark}\label{rem:staralgebrica}
Given $(A,\cH)$ as above we already observed that the fibers of $\cE$ are of the form $\ker{T_{x} s^{*}}={\rm Der}_{\Sscript{\Cc}}^{\Sscript{s}}(\cH, \Cc_{\Sscript{{x}\varepsilon}})$, for $x \in \affLine{\Cc}$. As explained in Remark \ref{rem:deradjoint}, we can consider in the category of augmented algebras the following cokernel
\[
\xymatrix @C=15pt{
A \ar[r]^{s} & \cH \ar[r]^-{\pi_{\Sscript{(x)}}} & \cH_{\Sscript{(x)}} \ar[r] & \Cc,
}
\]
where $A$ has augmentation $x$, while $\cH$ has augmentation $x\circ \varepsilon$. Note that, by construction, $\cH_{\Sscript{(x)}}$ is $\cH$ quotient by the ideal $\langle s(a)-x(a)1_\cH\mid a\in A\rangle$. By the remark quoted above we get an isomorphism of vector spaces
$$
{\rm Der}_{\Sscript{\Cc}}^{\Sscript{s}}(\cH, \Cc_{\Sscript{{x}\varepsilon}}) \,\cong\, \Der{\K}{\cH_{\Sscript{(x)}}}{\Cc_{\varepsilon_x}},$$
where $\varepsilon_x:\cH_{\Sscript{(x)}}\to \Cc$ is the unique algebra map such that $\varepsilon_x\circ \pi_{\Sscript{(x)}}= x\circ \varepsilon$.
Since we know that
\[
{\rm CAlg}_{\Sscript{\Cc}}(\cH_{\Sscript{(x)}},\Cc) \cong \Big \{g\in \cH(\Cc)\mid g(s(a)-x(a)1_\cH)=0, \,\forall \,a\in A\Big\} = \Big\{g\in \cH(\Cc)\mid s^*(g) = g\circ s = x\Big\} = \cH(\Cc)_x,
\]
then $\cH_{\Sscript{(x)}}\cong \K_x\tensor{A}{_s\cH}$ is the coordinate algebra of the subvariety $\cH(\Cc)_x$ known as the left star of the point $x$ in the groupoid $\cH(\mathbb{C})$. Furthermore, the morphism  of Hopf algebroids $\pi_{\Sscript{x}}: \cH \to \Hx$, where $(\mathbb{C}, \Hx)$ is, as in subsection \ref{ssec:isotropy},  the isotropy Hopf algebra of $\cH$ at the point $x$, factors throughout the morphism $\pi_{\Sscript{(x)}}$, leading to a morphism of augmented $\mathbb{C}$-algebras $\cH_{\Sscript{(x)}} \to \Hx$\footnote{This algebra maps is nothing but the canonical surjective map: $\mathbb{C}[ x_0,y_0, y_1,y_2,\cdots, \frac{1}{y_1}]/\langle x_0-x \rangle  \to \mathbb{C}[  X, y_1,y_2,\cdots, \frac{1}{y_1}]/\langle X-x\rangle$.}. Applying the derivations functor ${\rm Der}_{\Sscript{\mathbb{C}}}(-, \mathbb{C})$ to this latter morphism gives rise to the canonical injection of Lie algebras
$$
\LHx \stackrel{\eqref{eq:nablaLX}}{=} \Big\{ \delta \in {\rm Der}^{\Sscript{s}}_{\Sscript{\mathbb{C}}}(\cH, \mathbb{C}_{\Sscript{x \varepsilon}}) |\,\, \delta \circ t\,=\,0  \Big\} \hookrightarrow  {\rm Der}_{\Sscript{\Cc}}^{\Sscript{s}}(\cH, \Cc_{\Sscript{{x}\varepsilon}}).
$$
Notice that all these observations are valid for any Hopf algebroid $(A,\cH)$ over $\Bbbk$ such that $A(\Bbbk)\neq \emptyset$.
\end{remark}


\appendix

\section{The functorial approach,  the units of the adjunctions and Lie groupoids.}\label{sec:Functors}

In this section we provide an alternative construction of the differential functor $\lL$ constructed in \S\ref{ssec:lL}. This is done by mimicking the differential calculus on affine group schemes \cite[II \S 4]{DemazureGabriel} parallel to the construction of a Lie algebroid from a Lie groupoid. Moreover, we provide an alternative (direct) construction of the unit of the adjunction in Theorem \ref{thm:A}. Finally, we revisit also the construction of the Lie algebroid of a Lie groupoid under an algebraic point of view.

The following remark will be used all along the appendices.

\begin{remark}\label{rem:deradjoint}
Recall that the category $\Aug$ of augmented algebras has as objects pairs $(A,\varepsilon)$ where $A$ is an algebra and $\varepsilon:A\to \K$ is a distinguished algebra map, called \emph{augmentation}, and as morphism algebra maps preserving the augmentation. Analogously,  the category of coaugmented coalgebras $\Coaug$ has as objects pairs $(C,g)$ where $C$ is a coalgebra and $g$ is a distinguished group-like element in $C$ and as morphism coalgebra maps preserving the group-likes. The duality $\xymatrix@C=15pt{ \rdual{(-)}: \Coalgk  \ar@<0.4ex>[r]  &  \ar@<0.4ex>[l]    \Algk^{\mathsf{op}}  : \rcirc{(-)} }$ induces a duality between $\Coaug$ and $\left(\Aug\right)^{\mathsf{op}}$, namely $(C,g)^* = (C^*,g^*)$, where $g^*:C^*\to \K$ is the evaluation at $g$, and $\rcirc{(A,\varepsilon)}=(\rcirc{A},\varepsilon)$.
In addition, we have an adjunction between the category of vector spaces $\veck$ and $\Coaug$ given by the functor $\cP:\Coaug\to \veck$ sending every $(C,g)$ to $\cP(C,g):=\{c\in C\mid \Delta(c)=c\otimes g + g\otimes c\}$ and its left adjoint sending $V$ to $(\K\oplus V,1_{\K})$, where $\Delta(v)=v\otimes 1_\K+ 1_\K\otimes v$ for every $v\in V$.

Note that composing the right adjoints we get
$
\cP\left(\rcirc{(A,\varepsilon)}\right) = \cP(\rcirc{A},\varepsilon) = \Der{\K}{A}{\K_{\varepsilon}}.
$
As a consequence, the functor $\left(\Aug\right)^{\mathsf{op}}\to \veck$ sending every $(A,\varepsilon)$ to $\Der{\K}{A}{\K_{\varepsilon}}$ is a right adjoint. In particular, it preserves kernels once observed that $\Aug$ has $(\K,\id{\K})$ as zero object.
By the existence of this zero object, given a morphism of augmented $\K$-algebras  $s: A_0\to A_1$
we can consider in $\Aug$ the cokernel
\[
\xymatrix @C=15pt{
A_0 \ar[r]^{s} & A_1 \ar[r]^-{\pi} & A_2 \ar[r] & \K,
}
\]
which is defined as the coequalizer of the pair $(s,u_1\circ\varepsilon_0)$ in the category of algebras with the induced augmentation. Here we denoted by $\varepsilon_i:A_i\to\K$ the augmentations and by $u_i:\K\to A_i$ the units. By the foregoing, we get the following kernel of vector spaces
\[
\xymatrix @C=15pt{
0 \ar[r] & \Der{\K}{A_2}{\K_{\varepsilon_2}} \ar[r]^-{\pi^*} & \Der{\K}{A_1}{\K_{\varepsilon_1}} \ar[r]^-{s^*} & \Der{\K}{A_0}{\K_{\varepsilon_0}}.
}
\]
Summing up, $\pi^*$ induces an isomorphism
\begin{equation}\label{eq:derivations}
\Der{\K}{A_2}{\K_{\varepsilon_2}} \cong \ker{s^*}=\{\delta_1\in \Der{\K}{A_1}{\K_{\varepsilon_1}} \mid \delta_1\circ s=0\} = {\rm Der}_{\Sscript{\K}}^{\Sscript{s}}(A_1, \K_{\varepsilon_1}).
\end{equation}
\end{remark}

\subsection{The functorial approach to the differential functor}\label{ssec:FA}
Let us introduce  some useful notation. Given two algebras $T$ and $R$ we denote by $T(R):=\Calg({T},{R})$ the set of all algebra maps from $T$ to $R$, and by $\Calg$ the category of all commutative algebras.  To any commutative Hopf  algebroid $(A,\cH)$ one associates the presheaf of groupoids $\hH: \Calg \to \Grpd$ assigning to an algebra $R \in \Calg$ the groupoid
$$
\xymatrix@C=45pt{ \hHa(R):=\cH(R) \ar@<1ex>@{->}|(.4){\scriptstyle{s}}[r] \ar@<-1ex>@{->}|(.4){\scriptstyle{t}}[r] & \ar@{->}|(.4){ \scriptstyle{\iota}}[l]  A(R) :=\hHo(R) }
$$
whose structure is given as follows: For any $g \in \cH(R), x\in A(R)$, we have  $s(g)=g s$, $t(g)=gt$, $\iota_{x}=x\varepsilon$, $g^{-1}=g \cS$,  and if $gs=g't$ for some other $g' \in \cH(R)$, then $g.g': \cH \to R$ sends $u \mapsto g'(u_{1})g(u_{2})$.

Let us define the following functor:
\begin{equation}\label{Eq:xX}
{\xX: \Calg \longrightarrow \Sets}, \quad \left(  R \longmapsto \biguplus_{x \, \in \, A(R)} \Derk{s}{\cH}{\Rxe} \right),
\end{equation}
where $\biguplus$ denotes the disjoint union of sets.
For each $R \in \Calg$, $\xX(R)$ can be seen as a bundle (in the sense of \cite[Definition 1.1, chapter 2]{Husemoller}) of $\cH$-modules over $\hHo(R)$  with  canonical projection $\Pir: \xX(R) \to A(R)$ sending $\delta\in\Derk{s}{\cH}{\Rxe}$ to $x$. Now, $\xX$ is a functor as for any morphism $f :R \to T$, the map $\xX(f):\xX(R) \to \xX(T)$ is fiberwise defined by composition with $f$. This makes $\pi: \xX \to \hHo$ a natural transformation.

Following  \cite{DemazureGabriel}, let us consider the trivial extension algebra $R[\hbar]$ of a given algebra $R$, that is, $\hbar^{2}=0$ together with the canonical algebra injection $\fk{i}: R \to R[\hbar]$, $r \mapsto (r,0)$.  Denote by $p: \R[\hbar] \to R$ the algebra projection to the first component and by $p':R[\hbar] \to R$ the $R$-linear  projection to the second component. Then we have a morphism of groupoids $\hH(p): \hH(R[\hbar]) \to \hH(R)$. For a fixed $x \in A(R)$, we set
$$
\dD_{\Sscript{x}}(R):= \Big\{  \gamma \in \cH(R[\hbar])|\; \;  p' \gamma s=0,\; p\gamma=x \varepsilon \Big\}.
$$
Clearly, any arrow $\gamma \in \dD_{\Sscript{x}}(R)$ belongs to the kernel of $\hH(p)$, \ie $\left\{\gamma\in \hH_1(R[\hbar])\mid \hH_1(p)(\gamma)\in \iota\left(\hH_0(R)\right)\right\}$. Furthermore, if we denote by $\td{\gamma}:=p' \gamma$, then $\td{\gamma}$ becomes a $x\varepsilon$-derivation, in the sense that
$$
\td{\gamma}(uv)=x\varepsilon(u)\td{\gamma}(v) + \td{\gamma}(u)x\varepsilon(v),
$$
for every $u, v \in \cH$.
Each of the fibers $\dD_{\Sscript{x}}(R)$ is as follows a $\Bbbk$-vector space:
$$
\lambda \gamma :=( x \varepsilon, \lambda \td{\gamma}), \quad \gamma + \gamma' :=( x \varepsilon, \td{\gamma}+ \td{\gamma'}), \quad \text{ for every } \lambda \in \Bbbk, \text{ and } \gamma , \gamma' \in \dD_{\Sscript{x}},
$$
where the notation is the obvious one for diagonal morphisms.  We have then constructed a functor
\begin{equation}\label{Eq:dD}
{\dD: \Calg \longrightarrow \Sets}, \quad \left(  R \longmapsto \biguplus_{x \, \in \, A(R)} \dD_{\Sscript{x}}(R) \right),
\end{equation}
where for any morphism $f :R \to T$, the map $\dD(f):\dD(R) \to \dD(T)$ is fiberwise defined by composition with $(f,f)$. The functor $\dD$ is naturally isomorphic to $\xX$. Namely, the isomorphism is fiberwise given by
$$
\dD_{\Sscript{x}}(R)  \longrightarrow \Derk{s}{\cH}{\Rxe}, \; \Big( \gamma \longmapsto \td{\gamma} \Big); \qquad  \Derk{s}{\cH}{\Rxe}  \longrightarrow \dD_{\Sscript{x}}(R), \; \Big( \delta \longmapsto (x\varepsilon, \delta) \Big).
$$
Under this isomorphism, the elements of $\xX(R)$, for a given algebra $R$, can be seen as arrows in the groupoid $\hH(R[\hbar])$, although, contrary to the classical situation,  they only form a subcategory and not necessarily a subgroupoid.
Let us show that the set $\xX^{\Sscript{\ell}}(R)$ of loops in the category $\xX(R)$ is a groupoid-set in the following sense (for the definition of groupoid-set,  see \eg \cite{ElKaoutit:2015}).

An element $\delta \in \xX(R)$ belongs to $\xX^{\Sscript{\ell}}(R)$ provided that it satisfies also the equation $ \delta t=0$. Thus, $\delta$ is a $x\varepsilon$-derivation which kills both source an target and we can write
\begin{equation}\label{Eq:xXl}
{\xX^{\Sscript{\ell}}: \Calg \longrightarrow \Sets}, \quad \left(  R \longmapsto \biguplus_{x \, \in \, A(R)} \Derk{s,t}{\cH}{\Rxe} \right).
\end{equation}
It is easily checked that $\xX^{\Sscript{\ell}}$ is a functor. What we are claiming is that $\xX^{\Sscript{\ell}}$ with the structure map given by the restriction of $\pi$ is actually an $\hH$-set, in the sense of presheaves of groupoids. Taken the natural transformations $\pi : \xX^{\Sscript{\ell}} \to \hHo$ and $t: \hHa \to \hHo$, consider the fiber product
$$
\xX^{\Sscript{\ell}} \, \due{\times}{\Sscript{\pi}}{\Sscript{t}} \,  \hHa : {\Calg \longrightarrow \Sets}, \quad \Big(  R \longrightarrow  \xX^{\Sscript{\ell}}(R) \, \due{\times}{\Sscript{\Pir}} {\Sscript{t}}  \, \hHa(R) \,=\, \Big\{ (\delta, g) \, \in \, \xX^{\Sscript{\ell}}(R) \times \hHa(R) |\, \, gt=\Pir(\delta)\Big\}  \Big).
$$
Given an element $(\delta, g) \in  \xX^{\Sscript{\ell}}(R)\, \due{\times}{\Sscript{\Pir}}{\Sscript{t}}\,  \hHa(R)$, we define the \emph{conjugation action} by
\begin{equation}\label{Eq:conj}
 \delta \, .\, g: \cH \longrightarrow R_{\Sscript{gs\varepsilon}},\quad \Big(  u \longmapsto g(\uo)\delta(\udos) g(\cS(\utres)) \Big).
\end{equation}
Notice that this map is well-defined as
$$
\delta(s(a)ut(b)) = x(a)\delta(u)x(b)\;\;  \text{ and } \;\; g(ut(a))=g(u)g(t(a)) = g(u)\Pir(\delta)(a)=g(u)x(a)
$$
for every $a,b\in A, u\in \cH$, where $\delta\in \Derk{s}{\cH}{\Rxe}$. The following is the desired claim.

\begin{lemma}\label{lema:consejo}
For every algebra $R \in \Calg$, the  pair $(\xX^{\Sscript{\ell}}(R), \Pir)$ is a right $\hH(R)$-set with action given by conjugation as in \eqref{Eq:conj}. Furthermore, this is a functorial action, that is,  $(\xX^{\Sscript{\ell}}, \pi)$ is a right $\hH$-functor.
\end{lemma}
\begin{proof}
It is straightforward to show that $\delta\, . \,g$ belongs to $\xX(R)$ with projection $gs \in A(R)$, where $x=\Pir(\delta)=gt$. The rest of the first claim is clear.

Now let $f:R\to S$ be an algebra map. If $\delta \in \xX^{\Sscript{\ell}}(R)$ with $\Pir(\delta)=x$, then clearly $\Pis(\xX(f)(\delta))=fx$. On the other hand, the following diagram
$$
\xymatrix@C=50pt@R=25pt{  \xX^{\Sscript{\ell}}(R)\, \due{\times}{\Sscript{\Pir}}{\Sscript{t}}\,  \hHa(R) \ar@{->}^{}[rr]  \ar@{->}|-{\Sscript{\xX^{\Sscript{\ell}}(f)\, \due{\times}{\Sscript{\Pir}}{\Sscript{t}}\,  \hHa(f)}}[d] & &  \xX^{\Sscript{\ell}}(R)  \ar@{->}|-{\xX^{\Sscript{\ell}}(f)}[d] \\ \xX^{\Sscript{\ell}}(S)\, \due{\times}{\Sscript{\Pir}}{\Sscript{t}}\,  \hHa(S) \ar@{->}^{}[rr] & &  \xX^{\Sscript{\ell}}(S) }
$$
commutes, which means that $\xX^{\Sscript{\ell}}(f)$ is a right $\hH$-equivariant map. This shows that the natural transformation $ \xX^{\Sscript{\ell}}\, \due{\times}{\Sscript{\pi}}{\Sscript{t}}\,  \hHa  \to \xX^{\Sscript{\ell}}$ defines effectively a right $\hH$-action on $\xX^{\Sscript{\ell}}$.
\end{proof}

Viewing $\xX$ as a bundle over $\hHo$, one can define its module of sections as follows
\begin{equation}\label{Eq:OmegaX}
\Gamma(\xX)\, \, =\,\, \Big\{ \tau \in \Nat{\hHo, \xX}  \mid \pi \circ \tau =id  \Big\}.
\end{equation}
This is a vector space, whose operations are defined fiberwise.

On the other hand, for any algebra $R \in \Calg$, we may consider the following  bundle
\begin{equation}\label{Eq:yY}
\xymatrix{ \yY(R)\,=\, \underset{g \, \in \, \hHa(R)}{\biguplus} \Derk{s}{\cH}{\Rg} \ar@{->}^-{\Pir}[rr] & & \hHa(R).}
\end{equation}
When $R$ runs in $\Calg$, $\yY$  gives a functor, and one can consider as before its vector space of sections $\Gamma(\yY)$.

\begin{proposition}\label{prop:OXY}
Let $\Gamma(\xX)$ and $\Gamma(\yY)$ be as above. Then we have the following properties:
\begin{enumerate}[label=({\alph*}),ref=\emph{(\alph*)}]
\item\label{item:a} For any algebra map  $f:R \to S$ in $\Calg$, any object $x \in \hHo(R)$ and any $\tau \in \Gamma(\xX)$, we have:
\begin{equation}\label{Eq:a}
x \circ \taua(\ida) \,\, =\,\, \taur(x) , \qquad f \circ \taur(x) \,\,=\,\, \taus(fx).
\end{equation}
In particular,  we have
\begin{equation}\label{Eq:ap}
\varepsilon \circ \tauh(t)\,\,=\,\, \taua(\ida) \quad \text{and} \quad x\varepsilon\Big( \uo \tauh(t)(\uoo) \Big)\,\,=\,\, \taur(x)(u),
\end{equation}
for every $x \in\hHo(R)$ and $u \in \cH$.
\item\label{item:b} Both $\Gamma(\xX)$ and $\Gamma(\yY)$ admit a structure of $A$-module given as follows:
\begin{equation}\label{Eq:AT}
(a\,. \tau)_{\Sscript{R}}(x)\,\,=\,\, x(a) \, . \taur(x),\qquad (a\,. \alpha)_{\Sscript{R}}(g)\,\,=\,\, gt(a)\, .\alphar(g),
\end{equation}
for $R$ in $\Calg$, $x \in \hHo (R)$, $g \in \hHa (R)$ and for every $\tau \in \Gamma(\xX)$, $\alpha \in \Gamma(\yY)$ and  $a \in A$.
\item\label{item:c} The following map
$$
\xymatrix@R=0pt{  \Gamma(\xX) \ar@{->}^{\Sigma}[rr]  & & \Gamma(\yY) \\ \tau \ar@{|->}[rr] & &  {\left[ \arraycolsep=2pt\begin{array}{ccrcll} \Sigtr &: & \hHa(R) & \longrightarrow & \xX(R) & \\ && g & \longmapsto & \Sigtr(g):& H \to \Rg \\ &&&&&u  \mapsto  g(\uo)\taur(gt)(\uoo) \end{array} \right],   }}
$$
where  $R \in \Calg$ and $g\in\hHa(R)$, is a monomorphism of $A$-modules. Thus, any section of $\xX$ extends uniquely to a section of $\yY$.
\end{enumerate}
\end{proposition}

\begin{proof}
Part \ref{item:a} follows from naturality of $\tau$. Part \ref{item:b} is straightforward. As for part \ref{item:c}, let us first check that $\Sigma$ is a well-defined map. Take $\tau \in \Gamax$, $g \in \hHa(R)$ and set $x=gt$. By using the fact that $\taur(x)$ is a derivation, one easily checks that $\Sigtr(g)\in \Derk{s}{H}{\Rg}$.
Assume we are given $\tau, \tau' \in \Gamax$ such that $\Sigma(\tau)=\Sigma(\tau')$. Then, for every $g \in \hHa(R)$, we have that
$$
g(\uo)\taur(gt)(\uoo)\,\,=\,\, g(\uo)\ptaur(gt)(\uoo)
$$
for every $u \in\cH$. Now, take an arbitrary $x'\in \hHo(R)$ and set $g=x'\varepsilon$. Hence, for every $u \in \cH$, we obtain
\begin{multline*}
x'\varepsilon(\uo)\taur(x')(\uoo)= x'\varepsilon(\uo)\ptaur(x')(\uoo) \, \Rightarrow \, \taur(x')(s\varepsilon(\uo)\uoo)= \ptaur(x')(s\varepsilon(\uo)\uoo) \Rightarrow \, \ptaur(x')= \taur(x').
\end{multline*}
Therefore $\tau=\tau'$ and $\Sigma$ is injective. The fact that $\Sigma$ is $A$-linear is immediate and this finishes the proof.
\end{proof}

\begin{proposition}\label{prop:Gamma}
Let $(A,\cH)$ be a Hopf algebroid with associated presheaf $\hH$ and consider the  bundle $(\xX, \pi)$ as given in \eqref{Eq:xX}.
Then we have a bijection
$$
\nabla: \Gamax \longrightarrow \Derk{s}{\cH}{\Aep},\quad \Big( \tau \longmapsto \taua(\ida) \Big).
$$
In particular, the $A$-module of global sections $\Gamma(\xX)$ admits a unique structure of Lie-Rinehart algebra in such a way that $\nabla$ becomes an isomorphism of Lie-Rinehart algebras. Explicitly, for any $R \in \Calg$ the bracket $[\tau,\tau']_{\Sscript{R}}: \hHo(R) \to \xX(R)$ and the anchor $\omega'$ are respectively given by
$$
\xymatrix@R=0pt{  \cH \ar@{->}[r]^-{[\tau,\tau']_{\Sscript{R}} (x)} &  \Rxe \\ u \ar@{|->}[r] & \taur(x)\Big( \uo\, \ptauh(t)(\uoo) \Big) -  \ptaur(x)\Big( \uo\, \tauh(t)(\uoo) \Big) }
\qquad
\xymatrix@R=0pt{  \Gamax \ar@{->}^-{\omega'}[rr] & & \Ders{\K}{A} \\  \tau \ar@{|->}[rr]  & & \taua(\ida)\circ t  }
$$
\end{proposition}
\begin{proof}
In light of Yoneda's Lemma, we have a bijection $\nabla:\Nat{\hHo,\xX}\cong \xX(A)$ sending every natural transformation $\eta\in\Nat{\hHo,\xX}$ to $\nabla(\eta):=\eta_A(\id{A})$. It turns out that this bijection restricts to $\nabla:\Gamax \cong \xX'(A)$ where $\xX'(A)=\left\{\delta\in\xX(A)\mid \pi_A(\delta)=\ida\right\}$. By definition of $\pi_A$, we have $\pi_A(\delta)=\ida$ for every $\delta\in\Derk{s}{\cH}{\Aep}$ so that $\xX'(A)=\Derk{s}{\cH}{\Aep}$. This induces on $\Gamax$ the given Lie-Rinehart algebra structure since for $\tau,\sigma\in \Gamax$, $R\in\Calg, x\in\hHo(R),a\in A$ and $u\in\cH$ we have
\begin{align*}
\Big([\tau,\sigma]_R(x)\Big)(u) & = \bigg(\nabla^{-1}\Big(\left[\nabla(\tau),\nabla(\sigma)\right]\Big)_R(x)\bigg)(u) = \bigg(\nabla^{-1}\Big(\left[\taua(\ida),\sigma_A(\ida)\right]\Big)_R(x)\bigg)(u) = \Big(x\circ [\tau_A(\id{A}),\sigma_A(\id{A})]\Big)(u) \\
 & \stackrel{\eqref{Eq:bracket}}{=} x\bigg(\tau_A(\id{A})\Big(\uo t\big(\sigma_A(\id{A})(\uoo)\big)\Big)\bigg) - x\bigg(\sigma_A(\id{A})\Big(\uo t\big(\tau_A(\id{A})(\uoo)\big)\Big)\bigg) \\
 & \stackrel{\eqref{Eq:a}}{=} \tau_R(x)\Big(\uo \sigma_{\cH}(t)(\uoo)\Big) - \sigma_R(x)\Big(\uo \tau_{\cH}(t)(\uoo)\Big), \\
\omega'(\tau)(a) & = \omega(\nabla(\tau))(a) = \omega\left(\tau_A(\id{A})\right)(a) \stackrel{\eqref{eq:LR}}{=} \tau_A(\id{A})(t(a)).
\end{align*}
This concludes the proof.
\end{proof}

\begin{remark}\label{rem:Gammaell}
By mimicking Proposition \ref{prop:Gamma}, we get a bijection $\nabla^{\ell}:\Gamma(\xX^{\Sscript{\ell}})\to \Derk{s,\, t}{\cH}{\Aep}$, induced by $\nabla$ of the same proposition, where $\Gamma(\xX^{\Sscript{\ell}})$ is the $A$-module of global sections of the bundle $\xX^{\Sscript{\ell}}$ and $\Derk{s,\, t}{\cH}{\Aep}$ is the $A$-module of $\Bbbk$-algebra derivations $\delta: \cH \to \Aep$ such that $\delta s = \delta t=0$, which  in turns is the kernel of the anchor map given in equation \eqref{eq:LR}. Consider the so-called total isotropy Hopf algebroid $\cH^{\ell}:= \cH /\langle s-t \rangle$ of $\cH$ and denote by $\pi:H\to \cH^{\ell}$ the canonical projection\footnote{Here $\langle s-t \rangle$ stands for the Hopf ideal generated by the set $\{ s(a) -t(a)\}_{\Sscript{a\,  \in \, A}}$. Moreover the Hopf $A$-algebra $\cH^{\ell}$ is considered as a Hopf algebroid with base algebra $A$ with source equal to the target.}.
Note that given a symmetric $A$-bimodule $M$ (\ie $am=ma$ for all $a\in A,m\in M$) we have an isomorphism $\Hom{A-}{H^\ell}{M} = \Hom{A-A}{H^\ell}{M} \to \Hom{A-A}{H}{M}$ given by pre-composition by $\pi$. This isomorphism induces an isomorphism $\Derk{s}{\cH^\ell}{M_{\varepsilon^\ell}}\cong\Derk{s,\, t}{\cH}{M_\varepsilon}$. As a consequence, $\Gamma(\xX^{\Sscript{\ell}})$ is isomorphic to the Lie-Rinehart algebra $\lL(\cH^{\ell})$ of $\cH^{\ell}$.  If $x\in A(\Bbbk)$, then the fiber $\xX^{\Sscript{\ell}}(\Bbbk)_x=\Derk{s,\, t}{\cH}{\K_{x\varepsilon}}$ of the bundle $\xX^{\Sscript{\ell}}(\Bbbk)$ coincides by \eqref{Eq:LHx} with the isotropy Lie algebra $\lL(\cH)_x$ of $\lL(\cH)$. On the other hand, since $\Derk{s,\, t}{\cH}{\K_{x\varepsilon}} \cong \Derk{s}{\cH^\ell}{\K_{x\varepsilon}} = \lL(\cH^\ell)_x$ we get that $\xX^{\Sscript{\ell}}(\Bbbk)_x \cong \lL(\cH^\ell)_x$.
\end{remark}

\begin{remark}
Note that the isomorphism $\nabla$ of Proposition \ref{prop:Gamma} can be adapted to get an isomorphism $\nabla':\Gamay \to \Derk{s}{\cH}{\cH}$. Via these isomorphisms, one can see that the morphism $\Sigma$ from Proposition \ref{prop:OXY} corresponds to a morphism $\Derk{s}{\cH}{\Aep}\to \Derk{s}{\cH}{\cH}$, whose corestriction to its image is $\theta'$ of Lemma \ref{lema:LR}. This makes also clear why $\Sigma$ is injective.
\end{remark}


\subsection{Units of the adjunction between differentiation and integration}\label{ssec:UAdj}
We give here an  explicit description of the unit and  counit of the adjunction proved in \S\ref{sec:adjunctions}.

\begin{proposition}\label{prop:Theta}
Let $(A, L)$ be  a Lie-Rinehart algebra. Then there is a natural transformation
\begin{equation}\label{eq:Theta}
\Theta_{\Sscript{L}}: L \longrightarrow \Derk{s}{\VL^{\circ}}{\Aep}=\lL\iI(L),\; \Big( X \longmapsto \Big[ z \mapsto -\zeta(z)\big(\iota_{\Sscript{L}}(X)\big) \Big]  \Big)
\end{equation}
of Lie-Rinehart algebras.  Moreover, this morphism factors as follows and leads to
\begin{equation}\label{Eq:Triangle}
\begin{gathered}
\xymatrix{   L  \ar@{->}^-{\Theta_{L}}[rr]   \ar@{->}_-{\Theta'_{L}}[rrd]  &&   \Derk{s}{\VL^{\circ}}{\Aep}  \\ &&  \Derk{s}{\rsaft{\VL}}{\Aep}  \ar@{->}_-{\lL(\hat{\zeta})}[u] }
\end{gathered}
\end{equation}
a commutative diagram of Lie-Rinehart algebras, where $\Theta'_{\Sscript{L}}$ is the map which corresponds, by using the bijection of Lemma \ref{lema:dataI}, to the $A$-ring morphism $\fk{i}: \VL \to {}^{*}(\rsaft{\VL})$ defined in equation \eqref{Eq:I}.
\end{proposition}
\begin{proof}
By Lemma \ref{lema:dataI}, $\Theta'_{\Sscript{L}}$ is a morphism of Lie-Rinehart algebras. By Proposition \ref{prop:L}, we know that $\lL(\hat{\zeta})$ is Lie-Rinehart as well. As a consequence, $\Theta_{\Sscript{L}}:=\lL(\hat{\zeta})\circ\Theta'_{\Sscript{L}}$ is a morphism of Lie-Rinehart algebras.
It remains to check that it behaves as in \eqref{eq:Theta}. By using \ref{item:6.4e} of Lemma \ref{lema:dataI},  we know that $\Theta'_{\Sscript{L}}=\widetilde{\fk{i}}$. Therefore, for any $X \in L$, we have $\Theta_{\Sscript{L}} (X) = \lL(\hat{\zeta}) \left( \Theta'_{\Sscript{L}} (X)\right) =   \Theta'_{\Sscript{L}} (X)  \circ \hat{\zeta} = \widetilde{\fk{i}}(X) \circ \hat{\zeta}=- \fk{i}(\iota_{\Sscript{L}}(X)) \circ \hat{\zeta} $ and so, by \eqref{Eq:I}, we get $\Theta_{\Sscript{L}} (X)(z)=-\Big(\xi \circ \hat{\zeta} \Big)\big( z \big) \big(\iota_{\Sscript{L}}(X) \big) = -\zeta\big( z \big) \big(\iota_{\Sscript{L}}(X) \big)$.
\end{proof}

Now, consider $(A, \cH)$ a commutative Hopf algebroid and let $\cV_{\Sscript{A}}(\lL(\cH))$ be the universal algebroid of the Lie-Rinehart algebra $(A, \lL(\cH))$ of derivations of $\cH$. Take an object $(V, \varrho_{\Sscript{V}}) $ in the full subcategory $\Acom{\cH}$ (that is, a right $\cH$-comodule such that $V_{\Sscript{A}}$ is finitely generated and projective\footnote{Actually this assumption is not needed for the next construction.}), then we have a map
\begin{equation}\label{Eq:LV}
\begin{gathered}
\xymatrix@R=0pt@C=40pt{  \lambda^{\Sscript{V}}:  \lL(\cH)  \ar@{->}[rr] & &  \Endk{V} \\ \delta   \ar@{|->}[rr] & & \Big[ v \mapsto -  \vi \, \delta(\vii) \Big].   }
\end{gathered}
\end{equation}

\begin{proposition}\label{prop:VL}
Let $(A, \cH)$ be as above. Then the map \eqref{Eq:LV} induces a structure of right $\cV_{\Sscript{A}}(\lL(\cH))$-module on $V$. Moreover, this establishes a symmetric monoidal  functor
$$
\nabla: \Acom{\cH} \longrightarrow \Amod{\cV_{\Sscript{A}}(\lL(\cH))}, \quad \Big( (V, \varrho_{\Sscript{V}})    \longrightarrow (V, \lambda^{\Sscript{V}}) \Big)
$$
which commutes with the fibers functor, and so we obtain
$$
\rR(\nabla): \infcomatrix{\Acom{\cH}}{\Sigma}  \longrightarrow \infcomatrix{\Amod{\cV_{\Sscript{A}}(\lL(\cH))}}{\Sigma}= \cV_{\Sscript{A}}(\lL(\cH))^{\circ}.
$$
a morphism of commutative Hopf algebroids. Furthermore, there is a natural transformation
$$
\xymatrix{ \Omega_{\cH}: \, \cH \ar@{->}^-{\can{}^{-1}}[rr] & & \infcomatrix{\Acom{\cH}}{\Sigma}
\ar@{->}^-{\rR(\nabla)}[rr] & & \infcomatrix{\Amod{\cV_{\Sscript{A}}(\lL(\cH))}}{\Sigma}= \cV_{\Sscript{A}}(\lL(\cH))^{\circ} }
$$
whenever $\cH$ is a Galois Hopf algebroid.
\end{proposition}
\begin{proof}
Let us check first that $\lambda:=\lambda^{\Sscript{V}}$ is an anti-Lie algebra map. So take $v \in  V$ and $\delta, \delta' \in \lL(\cH)$. We compute on the one hand:

\begin{equation*}
\lambda([\delta, \delta'])(v) = -\vi\, [\delta, \delta'](\vii) = -\vi\, \Big( \delta\big( \vii \, t(\delta'(\vdos)) \big) - \delta'\big( \vii\, t(\delta(\vdos)) \big)   \Big),
\end{equation*}
and on the other hand,
\begin{align*}
[\lambda(\delta), \lambda(\delta')](v) &= \lambda(\delta)\big( \lambda(\delta')(v) \big) - \lambda(\delta' )\big( \lambda(\delta)(v) \big) =-\lambda(\delta)\big(\vi\,  \delta'(\vii) \big) + \lambda(\delta' )\big( \vi\,  \delta(\vii) \big) \\
&= \vi\, \Big( \delta\big( \vii \, t(\delta'(\vdos)) \big) - \delta'\big( \vii\, t(\delta(\vdos)) \big)   \Big),
\end{align*}
which implies that $\lambda([\delta, \delta'])= -[\lambda(\delta), \lambda(\delta')]$.  Let us denote by $l_{a}\in \Endk{V}$ the $A$-action on $V$ by $a$. So, for every $v \in V$,  $a \in A$ and $\delta \in \lL(\cH)$,  we have that
\begin{align*}
\Big( l_{a} \circ \lambda(\delta) - \lambda(\delta) \circ l_{a}\Big) (v) &= -a \vi \delta(\vii) + \vi
\delta(\vii t(a))  = -a \vi \delta(\vii) + \vi \delta(\vii ) a + \vi \varepsilon(\vii) \delta(t(a) ) \\
 & =  v \, \delta(t(a)) =  l_{\omega(\delta)(a)} (v),
\end{align*}
so that $\lambda(\delta) \circ l_{a} - l_{a} \circ \lambda(\delta) = l_{-\omega(\delta)(a)}$.  Summing up, $V$ is a right representation of $\lL(\cH)$ and, by the universal property of $\cV_A(\lL(\cH))$, this implies that there is an algebra map $\cV_{\Sscript{A}}(\lL(\cH)) \to \Endk{V}^{\text{op}}$ which makes of $V$ a right $\cV_A(\lL(\cH))$-module. This defines the functor $\nabla$ on the objects. On arrows this functor acts as the identity, that is, $\nabla(f)=f$ for every $\cH$-colinear map $f: V \to V'$. The fact that $f$ is $\cV_A(\lL(\cH))$-linear may be proved by mimicking the argument of the proof of the second claim in Lemma \ref{lema:Htilda}: take $B\subseteq \cV_A(\lL(\cH))$ such that $f(vb)=f(v)b$ for all $b\in B$ and show that $B=\cV_A(\lL(\cH))$.  Monoidality of $\nabla$  comes as follows:  both tensor products are modelled on $\tensor{A}$ and the subsequent computation
\begin{gather*}
\lambda^{V\tensor{A}W}(\delta)(v\tensor{A}w) = -(\vi\tensor{A}\wi)\, \delta(\vii\wii) = -(\vi\tensor{A}\wi) \, \Big(  \varepsilon(\vii)\delta(\wii) + \delta(\vii) \varepsilon(\wii) \Big)
\\ =  v\tensor{A}\lambda^{\Sscript{V}}(\delta)(w) + \lambda^{\Sscript{W}}(\delta)(v) \tensor{A} w  =  v\tensor{A} w\, \iota_{\Sscript{\lL(\cH)}}(\delta)  + v\, \iota_{\Sscript{\lL(\cH)}}(\delta) \tensor{A} w \, =\, (v\tensor{A}w) \, \Delta(\iota_{\Sscript{\lL(\cH)}}(\delta)).
\end{gather*}
shows that the action on $\nabla(V\tensor{A}W)$ coincides with the diagonal one. The identity object $A$ has the action $\lambda^{\Sscript{A}}: \lL(\cH) \to \Endk{A}^{\text{op}}$ given by
$\lambda(\delta)(a) = -\ao\, \delta(\aoo) = -\delta(t(a))$. Therefore $\lambda^{\Sscript{A}} = -\omega$, the anchor described in Proposition \ref{prop:LR}. The rest of the proof of the first statement follows from the construction performed in \S\ref{ssec:Tannaka}.

Lastly, the naturality of $\Omega$ is proved as follows. Given a morphism $\phi:\cH \to \cH'$ of Galois Hopf algebroids, then  on the one hand we have   a commutative diagram
$$
\xymatrix@R=18pt{   & \Acom{\cH'} \ar@{->}^-{\nabla'}[rrr] \ar@/_1pc/@{->}^-{\oO}[rddd]  & & & \Amod{\cV_{\Sscript{A}}(\lL(\cH'))} \ar@/^1.5pc/@{->}^-{\oO'}[llddd]  \\  \Acom{\cH} \ar@{->}^-{\nabla}[rrr] \ar@{->}^-{\phi_{*}}[ru] \ar@/_1pc/@{->}^-{\fomega}[rrdd] & & &  \Amod{\cV_{\Sscript{A}}(\lL(\cH))}  \ar@{->}^-{ \cV_A(\lL(\phi))_*}[ru] \ar@/^1pc/@{->}_-{\fomega'}[ldd] & \\ & & & &   \\ & & \proj{A} & & }
$$
which leads to a commutative diagram
$$
\xymatrix@C=60pt{    \ar@{->}^{\rR(\nabla)}[r]  \ar@{->}_-{\rR(\phi_{*})}[d]  \infcomatrix{\Acom{\cH}}{\Sigma}  &    \ar@{->}^{\iI\lL(\phi)_{*}=\rR(\cV_A(\lL(\phi))_*)}[d]   \iI\lL(\cH) \\      \infcomatrix{\Acom{\cH'}}{\Sigma}  \ar@{->}^{\rR(\nabla')}[r] &   \iI\lL(\cH'). }
$$
On the other hand we have, by definition of the functor $\rR$, a commutative diagram
$$
\xymatrix@C=60pt{    \ar@{->}^{\can{\Sscript{\cH}}}[r]  \ar@{->}_-{\rR(\phi_{*})}[d]  \infcomatrix{\Acom{\cH}}{\Sigma}  &    \ar@{->}^{\phi}[d]   \cH \\      \infcomatrix{\Acom{\cH'}}{\Sigma}  \ar@{->}^{\can{\Sscript{\cH'}}}[r] &   \cH'.}
$$
Putting together the two diagrams leads to the naturality of $\Omega$ and it finishes the proof.
\end{proof}


\subsection{The  Lie-Rinehart algebra of a Lie groupoid. Revisited}\label{ssec:LA-LG}
Here we provide an algebraic approach  to the construction of a Lie algebroid, or Lie-Rinehart algebra, from a given Lie groupoid. This approach unifies in fact the definition given in  \cite[\S 3.5]{Mackenzie} and the one  in \cite{Cartier:2008}. We also discuss the injectivity of the unit of the adjunction between integration and differentiation functors, see Appendix \ref{ssec:UAdj}.

We will employ the following notations. Consider a diagram of commutative $\Rr$-algebras $\xymatrix @C=10pt{ B \ar@{->}^{x}[r] & C \ar@{->}^{y}[r] & D,}$ where $\Rr$ denotes the field of real numbers. As usual, we denote
\begin{equation*}
{\rm Der}^{\Sscript{x}}_{\Sscript{\Rr}}(C, D_{y}):=\Big\{  \gamma \in \Der{\Rr}{C}{D_{y}} | \; \gamma \circ x = 0\Big\}
\end{equation*}

Given a connected smooth real manifold $\cM$,  for each point $x \in\cM$, we denote by $x$ itself the algebra map $\Cinfty{\cM} \to \Rr$ sending $p \mapsto p(x)$.  The global smooth sections of the  tangent  vector bundle $T\cM = \cup_{x\, \in \, \cM} \Der{\Rr}{\Cinfty{\cM}}{\Rr_{x}} $ of $\cM$ are identified with the $\Cinfty{\cM}$-module of derivations of $\Cinfty{\cM}$ as follows: Take a section $\delta \in \Gamma(T\cM)$, we have a derivation
$$
\Cinfty{\cM} \longrightarrow \Cinfty{\cM}, \qquad \Big(  p  \longmapsto [ x \mapsto \delta_{x}(p) ] \Big),
$$
see \cite[\S9.38]{nestruev}. Let us consider a Lie groupoid
$$
\cG: \, \, \xymatrix@C=45pt{ \Ga \ar@<1ex>@{->}|(.4){\scriptstyle{s}}[r] \ar@<-1ex>@{->}|(.4){\scriptstyle{t}}[r] & \ar@{->}|(.4){ \scriptstyle{\iota}}[l]  \Go, }
$$
where $\Ga$ is assumed to be a connected smooth real manifold and $s, t$ are surjective submersions. This leads to a diagram of (geometric \cite[Definition 3.7]{nestruev}) smooth real algebras
$$
 \xymatrix@C=45pt{ \Cinfty{\Go} \ar@<1ex>@{->}|(.4){\scriptstyle{s^{*}}}[r] \ar@<-1ex>@{->}|(.4){\scriptstyle{t^{*}}}[r] & \ar@{->}|(.4){ \scriptstyle{\iota^{*}}}[l]  \Cinfty{\Ga}. }
$$
The \emph{left star} of a point $x \in \Go$ is by definition the (sub)-manifold $\cG_{x}=\{g \in \Ga |\, s(g)=x\} $ of $\Ga$, and denote by $\tau_{x}: \cG_{x} \hookrightarrow \Ga$ the corresponding embedding. Notice that we have a disjoint union $\Ga=\uplus_{x\, \in \, \Go} \cG_{x}$.  For an object $x \in \Go$, we have the following surjective $\Rr$-linear map: $T_{x}s: T_{\iota(x)}\Ga \to T_{x}\Go$, so we can set $\cE_{x}:=\ker{T_{x}s}$ and then consider the  vector bundle  $\cE =\cup_{x\, \in \, \Go}\cE_{x}$. Each fiber $\cE_{x}$ is then identified with $\Rr$-vector space
$ {\rm Der}_{\Sscript{\Rr}}^{\Sscript{s^{*}}}(\Cinfty{\Ga}, \Rr_{\iota(x)})$, thus, $\cE_{x}\,  =\, {\rm Der}_{\Sscript{\Rr}}^{\Sscript{s^{*}}}(\Cinfty{\Ga}, \Rr_{\iota(x)})$.

There is  another vector bundle $\cF$ whose fibers at a point $x \in\Go$ is given by the $\Rr$-vector space
$$
\cF_{x}=T_{\iota(x)}(\cG_{x})=\Der{\Rr}{\Cinfty{\cG_{x}}}{\Rr_{\iota(x)}}.
$$

\begin{lemma}\label{lemma:IsoBundles}
We have an isomorphism of $\Rr$-vector spaces
\begin{equation}\label{Eq:Eta}
\eta_x:\Der{\Rr}{\Cinfty{\cG_{x}}}{\Rr_{\iota(x)}}   \to  {\rm Der}_{\Sscript{\Rr}}^{\Sscript{s^{*}}}(\Cinfty{\Ga}, \Rr_{\iota(x)}), \qquad \Big( \gamma_{x} \longmapsto \Big[  p \mapsto \gamma_{x}(p  \tau_{x})  \Big]  \Big)
\end{equation}
induced by $\tau_x^*:\Cinfty{\cG_{1}}\to\Cinfty{\cG_{x}}, p\mapsto p\tau_x$.
\end{lemma}
\begin{proof}
Recall that, by hypothesis, $s:\cG_1\to\cG_0$ is a surjective submersion. In particular, in light of \cite[Corollary 5.14]{Lee} for example, $\cG_x=s^{-1}(x)$ is a closed embedded submanifold of $\cG_1$ with local (in fact, global) defining map $s$ itself. Thus, as a consequence of \cite[Proposition 5.38]{Lee}, for any $h\in\cG_x$ we have that $T_h\cG_x=\ker{T_{h}s:T_h\cG_1\to T_{s(h)}\cG_0}$. In particular, $\Der{\Rr}{\Cinfty{\cG_{x}}}{\Rr_{\iota(x)}}=T_{\iota(x)}\cG_x = \ker{T_{\iota(x)}s:T_{\iota(x)}\cG_1\to T_{s\iota(x)}\cG_0}$, where the second identification is given through the inclusion $T_{\iota(x)}\tau_x:T_{\iota(x)}\cG_x\to T_{\iota(x)}\cG_1$ induced by $\tau_x$.
\end{proof}

As a consequence we get an isomorphism of vector bundles $\eta:\cF\to\cE$ and hence an isomorphism of $\Cinfty{\Go}$-module $\etaup:=\Gamma(\eta): \Gamma(\cF) \to \Gamma(\cE)$.  There are two morphisms of $\Cinfty{\Go}$-modules:
\begin{equation}\label{Eq:Omegas-Calorrr}
\begin{gathered}
\omega^{\Sscript{\cE}}: \Gamma(\cE) \longrightarrow {\rm Der}_{\Sscript{\Rr}}(\Cinfty{\Go}), \qquad \Big(  \delta \longmapsto [ a \mapsto \delta_{-}(at) ] \Big)
\end{gathered}
\end{equation}
and  $\omega^{\Sscript{\cF}} := \omega^{\Sscript{\cE}} \circ \etaup $. Recall that, by the foregoing, we can identify $\Gamma\left(\cT\cG_0\right)$ with ${\rm Der}_{\Sscript{\Rr}}(\Cinfty{\Go})$. By means of this identification one can check that the morphism of vector bundles $Tt:\cE\to \cT\cG_0$ induced by the $\R$-linear maps $T_xt:T_{\iota(x)}\cG_1\to T_x\cG_0$ is such that $\omega^{\Sscript{\cE}} = \Gamma(Tt)$. Clearly, $\omega^{\Sscript{\cF}} = \Gamma(Tt\circ\eta)$.

Given an arrow $g \in \Ga$, we have the right multiplication action $R_{g}: \cG_{t(g)} \to \cG_{s(g)}$, $h \mapsto hg$ (this, by the Lie groupoid structure of $\cG$, is a diffeomorphism).  Now, fix an object $x \in\Go$, a function $q \in \Cinfty{\cG_{x}}$ and a global section $\gamma \in \Gamma(\cF)$, then we  have a smooth function $\flecha{\gamma}_{-}(q) \in \Cinfty{\cG_{x}}$ given by
$$
\flecha{\gamma}_{-}(q): \cG_{x} \longrightarrow \Rr, \qquad \Big( h \longmapsto \flecha{\gamma}_{h}(q) := \gamma_{t(h)}(q R_{h}) \Big)
$$
see \cite[Corollary 3.5.4]{Mackenzie}.  The derivation $\flecha{\gamma}_{-}: \Cinfty{\cG_{x}} \to  \Cinfty{\cG_{x}}$ satisfies the following equalities:
\begin{equation}\label{Eq:Ga}
\left(\flecha{a\gamma}\right)_{-} \, \, =\, \, \tau_x^*\left(t^{*}(a)\right) \flecha{\gamma}_{-}, \quad \flecha{\gamma}_{\iota{-}}\,\, =\,\, \gamma_{-}, \quad \text{ for all } a \in \Cinfty{\Go},
\end{equation}
where $(a\gamma)_x=a(x)\gamma_x$ and $(b\flecha{\gamma}_{-})_{h} = b(h)\flecha{\gamma}_{h}$ for every $x\in\Go,a\in \Cinfty{\Go},b\in\Cinfty{\cG_{x}},h\in \cG_{x}$.
In this way, for a given pair of sections $(\gamma, \gamma') \in \Gamma(\cF) \times \Gamma(\cF)$, we have the following smooth global section
$$
[\gamma, \gamma']_{x}: \Cinfty{\cG_{x}} \longrightarrow \Rr_{\iota(x)}, \qquad \Big(  q \longmapsto \gamma_{x}(\flecha{\gamma'}_{-}(q)) - \gamma'_{x}(\flecha{\gamma}_{-}(q)) \Big).
$$
Namely, since $\iota(x)\in\cG_{x}$, for two functions $p, q \in \Cinfty{\cG_{x}}$ we may compute
\begin{eqnarray*}
[\gamma, \gamma']_{x}(p q) &=&   \gamma_{x}(\flecha{\gamma'}_{-}(pq)) - \gamma'_{x}(\flecha{\gamma}_{-}(pq))
\\ &=&  \gamma_{x}\Big(p\flecha{\gamma'}_{-}(q)+ \flecha{\gamma'}_{-}(p)q\Big) - \gamma'_{x}\Big(p\flecha{\gamma}_{-}(q) + \flecha{\gamma}_{-}(p)q\Big)
\\ &=& p(\iota(x))\gamma_{x}(\flecha{\gamma'}_{-}(q))+\gamma_{x}(p)\flecha{\gamma'}_{\iota(x)}(q)  + \flecha{\gamma'}_{\iota(x)}(p)\gamma_{x}(q) + \gamma_{x}(\flecha{\gamma'}_{-}(p))q(\iota(x))
\\ &\,\, &- \gamma'_{x}(p)\flecha{\gamma}_{\iota(x)}(q) - p(\iota(x))\gamma'_{x}(\flecha{\gamma}_{-}(q)) -\gamma'_{x}(\flecha{\gamma}_{-}(p))q(\iota(x)) - \flecha{\gamma}_{\iota(x)}(p)\gamma'_{x}(q)
\\ &=& p(\iota(x))[\gamma, \gamma']_{x}(q) + [\gamma, \gamma']_{x}(p)q(\iota(x)) \\ &\,\, & +\gamma_{x}(p)\flecha{\gamma'}_{\iota(x)}(q)  + \flecha{\gamma'}_{\iota(x)}(p)\gamma_{x}(q)- \gamma'_{x}(p)\flecha{\gamma}_{\iota(x)}(q)- \flecha{\gamma}_{\iota(x)}(p)\gamma'_{x}(q)
\\ &\overset{\eqref{Eq:Ga}}{=}& p(\iota(x))[\gamma, \gamma']_{x}(q) + [\gamma, \gamma']_{x}(p)q(\iota(x)),
\end{eqnarray*}
which shows that $([\gamma, \gamma']_{\Sscript{x}})_{\Sscript{ x \, \in \, \Go}} \in \Gamma(\cF)$. Furthermore, for a given $a \in \Cinfty{\Go}$, we have
\begin{eqnarray*}
[\gamma, a\gamma']_{x}(q) &=& \gamma_{x}(\flecha{(a\gamma')}_{-}(q)) - a(x)\gamma'_{x}(\flecha{\gamma}_{-}(q))
\\ &\overset{\eqref{Eq:Ga}}{=}& \gamma_{x}\Big(\tau_x^*(t^{*}(a))\flecha{\gamma'}_{-}(q)\Big) - a(x)\gamma'_{x}(\flecha{\gamma}_{-}(q))
\\ &=& \tau_x^*(t^{*}(a))(\iota(x))\gamma_{x}(\flecha{\gamma'}_{-}(q)) + \gamma_{x}(\tau_x^*(t^{*}(a)))\flecha{\gamma'}_{\iota(x)}(q) - a(x)\gamma'_{x}(\flecha{\gamma}_{-}(q))
\\ &\overset{\eqref{Eq:Ga}}{=}& a(x)[\gamma,\gamma']_{x}(q) + \omega^{\Sscript{\cF}}(\gamma)(a)(x)\gamma'_{x}(q),
\end{eqnarray*}
for every function $q \in \Cinfty{\cG_{x}}$. Thus, for every function $a \in \Cinfty{\Go}$, we have
$$
[\gamma, a\gamma']= a [\gamma, \gamma'] + \omega^{\Sscript{\cF}}(\gamma)(a) \gamma'
$$
as an equality in $\Gamma(\cF)$. This completes the structure of the Lie algebroid $(\cF, \Go)$, and the structure of Lie-Rinehart algebra of $(\Gamma(\cF), \Cinfty{\Go})$. This Lie algebroid is know in the literature as the \emph{Lie algebroid of the Lie groupoid $\cG$}.

Now, we come back to the vector bundle $(\cE, \Go)$. We can endow it with a Lie algebroid structure via the isomorphism $\eta$. The bracket on $\Gamma(\cE)$ is given by
$$
[\delta, \delta']_{x}: \Cinfty{\Ga} \longrightarrow \Rr_{\iota(x)}, \qquad \Big( b \longmapsto \delta_{x}(\flecha{\delta'}_{-}(b))- \delta'_{x}(\flecha{\delta}_{-}(b)) \Big)
$$
and the anchor is the map $\omega^{\Sscript{\cE}}$ of equation \eqref{Eq:Omegas-Calorrr}. In fact, concerning the bracket we can compute
\begin{align*}
\eta_x([\gamma,\gamma']_x)(p) & = [\gamma,\gamma']_x(p\tau_x)= \gamma_{x}(\flecha{\gamma'}_{-}(p\tau_x)) - \gamma'_{x}(\flecha{\gamma}_{-}(p\tau_x)) \\
  & \stackrel{(*)}{=} \gamma_x(\flecha{\etaup(\gamma')}_{-}(p)\tau_x)-\gamma'_x(\flecha{\etaup(\gamma)}_{-}(p)\tau_x) \\
  & = \etaup(\gamma)_x(\flecha{\etaup(\gamma')}_{-}(p))- \etaup(\gamma')_x(\flecha{\etaup(\gamma)}_{-}(p)) \\
  & = [\etaup(\gamma),\etaup(\gamma')]_x(p)
\end{align*}
where $(*)$ follows from the equality $\flecha{\etaup(\gamma)}_{-}(p)\tau_x = \flecha{\gamma}_{-}(p\circ \tau_x)$ which descends from
\begin{align*}
\left(\flecha{\etaup(\gamma)}_{-}(p)\tau_x\right)(h) & = \flecha{\etaup(\gamma)}_{-}(p)(\tau_x(h)) = \etaup(\gamma)_{t(\tau_x(h))}(p\circ R_{\tau_x(h)}) \\
 & = \etaup(\gamma)_{t(h)}(p\circ R_{\tau_x(h)}) = \eta_{t(h)}(\gamma_{t(h)})(p\circ R_{\tau_x(h)}) \\
  & = \gamma_{t(h)}(p\circ R_{\tau_x(h)}\circ\tau_{t(h)}) = \gamma_{t(h)}(p\circ \tau_x\circ R_{h}) \\
   & = \left(\flecha{\gamma}_{-}(p\circ \tau_x)\right)(h).
\end{align*}

With these structures,  we get that $\etaup: \Gamma(\cF) \to \Gamma(\cE)$ is an isomorphism of Lie-Rinehart algebras, where $\etaup=\Gamma(\eta)$ and $\eta$ is fiberwise given by equation \eqref{Eq:Eta}.

Lastly, applying the differentiation functor of \S\ref{ssec:lL} and using the natural transformation of equation \eqref{eq:Theta} together with the commutative diagram of equation \eqref{Eq:Triangle},  we obtain a commutative diagram of Lie-Rinehart algebras over $A:=\Cinfty{\Go}$
\begin{equation}\label{Eq:diagram}
\begin{gathered}
\xymatrix @C=0pt@R=13pt{   &  {\rm Der}^{\Sscript{s}}_{\Sscript{\Rr}}\Big(\cV_{\Sscript{A}}( \Gamma(\cF) )^{\circ}, A_{\varepsilon} \Big)  \ar@{->}^-{\lL\iI(\etaup)}[rr] &       &     {\rm Der}^{\Sscript{s}}_{\Sscript{\Rr}} \Big( {\cV_{\Sscript{A}}( \Gamma(\cE) )^{\circ}}, {A_{\varepsilon}} \Big)   & \\  &        &   & & \\  \Gamma(\cF) \ar@{->}_-{\etaup}[rr] \ar@{->}^-{\Theta_{\Sscript{\Gamma(\cF)}}}[ruu]  \ar@{->}_-{\Theta'_{\Sscript{\Gamma(\cF)}}}[dd]   &     &  \Gamma(\cE) \ar@{->}^-{\Theta_{\Sscript{\Gamma(\cE)}}}[ruu]  \ar@{->}_-{\Theta'_{\Sscript{\Gamma(\cE)}}}[dd]  & & \\ & &       & & \\
{\rm Der}^{\Sscript{s}}_{\Sscript{\Rr}} \Big( {\rsaft{\cV_{\Sscript{A}}( \Gamma(\cF) )}}, {A_{\varepsilon}} \Big) \ar@/_1pc/@{->}_(0.4){\lL(\hat{\zeta})}|(0.575){\phantom{X}}[uuuur]  \ar@{->}^-{\lL\iI'(\etaup)}[rr]  &    &  {\rm Der}^{\Sscript{s}}_{\Sscript{\Rr}} \Big( {\rsaft{\cV_{\Sscript{A}}( \Gamma(\cE) )}}, {A_{\varepsilon}} \Big)   , \ar@/_1pc/@{->}_-{\lL(\hat{\zeta})}[uuuur]    & &   }
\end{gathered}
\end{equation}
whose horizontal arrows are isomorphisms of Lie-Rinehart algebras.

\begin{remark}\label{rem:Smooth}
In the case of Lie groups, the map $\Theta$ of the  diagram \eqref{Eq:diagram} is  injective. In fact this map is injective for any finite-dimensional Lie algebra. Namely,  taking a  finite dimensional Lie $\Bbbk$-algebra $L$,  we have, as in Proposition \ref{prop:Theta}, the map $\Theta_L: L \to \Derk{}{U_{\Sscript{\K}}(L)^{\circ}}{\Bbbk_{\Sscript{\varepsilon}}}$ given by the evaluation $X \mapsto[f \mapsto f(X)]]$, where $U_{\Sscript{\K}}(L)^{\circ}$ is the finite dual Hopf algebra of the universal enveloping algebra of $L$.
Since, in light of \cite[p. 157]{Montgomery:1993}, $U_{\Sscript{\K}}(L)^{\circ}$ is a dense in $U_{\Sscript{\K}}(L)^*$ (here the topology is the linear one), $U_{\Sscript{\K}}(L)$ is a proper algebra in the sense of \cite[page 78]{Abe}, so $\Theta_L$ is injective.  Furthermore, in light of \cite[Theorem 6.1]{Hochschild:1970}, for $\K$ an algebraically closed field of characteristic zero $\Theta_L$ is bijective if and only if $L = [L,L]$.

Now if $G$ is a compact Lie group, then $G \cong {\rm CAlg}_{\Rr}(\rR_{\Sscript{\Rr}}(G), \Rr)$, the character group of    the commutative Hopf real algebra $\rR_{\Sscript{\Rr}}(G)$ of all  representative smooth  functions on $G$. The Lie algebra $Lie(G) = \lL(\rR_{\Sscript{\Rr}}(G))= {\rm Der}_{\Sscript{\Rr}}(\rR_{\Sscript{\Rr}}(G), \Rr_{\Sscript{\varepsilon}} )$ of $G$ is then identified with the Lie algebra of the primitive elements $Lie(G) \cong \prim(\rR_{\Sscript{\Rr}}(G)^{\circ})$ \cite[\S 4, Section 3]{Abe} of the finite dual $\rR_{\Sscript{\Rr}}(G)^{\circ}$.
Denote by $\tauup: U_{\Sscript{\Rr}}(Lie(G)) \hookrightarrow \rR_{\Sscript{\Rr}}(G)^{\circ}$ the canonical monomorphism of co-commutative Hopf algebras. Then, we know \cite{Michaelis-primitive}  that the map
$$
\hat{(-)}: \rR_{\Sscript{\Rr}}(G) \longrightarrow (\rR_{\Sscript{\Rr}}(G)^{\circ} )^{*},\qquad \Big( \varrho \longmapsto [f \mapsto f(\varrho)  \Big)
$$
factors through the inclusion $ \rR_{\Sscript{\Rr}}(G)^{\circ}{}^{\circ} \, \subseteq \,  (\rR_{\Sscript{\Rr}}(G)^{\circ} )^{*}$.
Therefore, the map $\Theta_{\Sscript{Lie(G)}}$ is a split monomorphism of Lie algebras,  namely,  with splitting map $\lL(\tauup^{\circ}\hat{(-)})$.

In the case of compact Lie groupoid (i.e., $\Go$ is a compact smooth manifold and each of the isotropy Lie groups of $\cG$ is compact), it would be interesting to study the injectivity of the map $\Theta$ either in the left hand or right hand  triangle in diagram \eqref{Eq:diagram}.
\end{remark}


\section{The factorization of the anchor map of the Lie-Rinehart algebra of a split Hopf algebroid}\label{sec:A1}

In this last appendix we show how the anti-homomorphism of Lie algebras $\cL ie(\cG)(\K) \to  \Ders{\K}{\mathscr{O}_\K(\cX)}$ of \cite[II, \S4, n\textsuperscript{o}4, Proposition 4.4, page 212]{DemazureGabriel} becomes the map of equation \eqref{Eq:DG}. We also give some specific cases  of  Example  \ref{Exam:Main}.

Recall that we have a commutative Hopf algebra $H$ such that $\cG=\Calg\left(H,-\right)$ and a left $H$-comodule commutative algebra $A$ such that $\cX=\Calg\left(A,-\right)$.
The coaction $\rho:A\to H\otimes A$ induces on $\cX$ a $\cG$-operation $\Calg(H,R)\times \Calg(A,R) \to \Calg(A,R)$, in the sense  of \cite[II, \S1, n\textsuperscript{o}3, D\'efinition 3.1, page 160]{DemazureGabriel}, which is an instance of
\begin{equation*}
\mu_B:\Calg(H,R)\times \Calg(A,B) \to \Calg(A,B):\, (f,g)\mapsto \left[a\mapsto f(a_{-1})g(a_{0})\right]
\end{equation*}
for $R\in \Calg$ and every $R$-algebra $B$. Define
\[
\mho_R:\Calg(H,R)\to\mathsf{Aut}_R(\cXR):f\mapsto \mho_R(f)
 \]
 where $\mho_R(f)_B:\Calg(A,B)\to \Calg(A,B):g\mapsto \mu_B(f,g)$ and where the functor $\cXR:\mathrm{CAlg}_R\to \mathsf{Sets}$ is simply the restriction of $\cX$ to $\mathrm{CAlg}_R$. The group $\mathsf{Aut}_R\left(\cXR\right)$ is the group of natural isomorphisms in $\mathsf{Nat}(\cXR,\cXR)$.

 Define the functor $\cA ut\left(\cX\right):\Calg\to \mathsf{Sets}$ by setting $\cA ut\left(\cX\right)(R):=\mathsf{Aut}_R\left(\cX\otimes R\right)$ and for every morphism $\phi:S\to R$ set $\cA ut\left(\cX\right)(\phi):\cA ut\left(\cX\right)(S)\to\cA ut\left(\cX\right)(R)$ sending every natural transformation $(\tau_B)_{ B\in \rmod{S}}$ to the natural transformation $(\tau_B)_{B\in \rmod{R}}$. Note that for every $f\in \Calg(H,S)$ and every $g\in\Calg(A,B)$ with $B\in \rmod{R}$, for all $a\in A$ we have that
 \[
 \left(\mho_S(f)_B(g)\right)(a)= f(a_{-1})\cdot g(a_0) = \xi(f(a_{-1}))g(a_0) = \left(\mho_R(\xi\circ f)_B(g)\right)(a).
 \]
 As a consequence, $\mho_S(f)_B= \mho_R(\xi\circ f)_B$ for all $B\in \rmod{R}$, which means that $\mho_R$ is natural in $R$ and hence we can write $\mho:\cG\to \cA ut\left(\cX\right)$.

 Now, recall that for every $R$-algebra $B$ we have an isomorphism $\Calg(A,B)\cong \mathrm{CAlg}_R(A\otimes R,B)$. As a consequence, $\mathsf{Aut}_R\left(\cXR\right)\cong \mathsf{Aut}_R\big(\mathrm{CAlg}_R(A\otimes R,-)\big)$ and, in view of the Yoneda isomorphism, $\mathsf{Aut}_R\big(\mathrm{CAlg}_R(A\otimes R,-)\big)\cong \mathsf{Aut}_R(A\otimes R)^{\mathsf{op}}$. Summing up, we have a group isomorphism
 \[
 \cY_R:\mathsf{Aut}_R(\cXR)\to \mathsf{Aut}_R(A\otimes R)^{\mathsf{op}}.
 \]
The composition of $\mho_R$ with $\cY_R$ yields a natural transformation
\begin{equation}\label{eq:reprfunctor}
\Calg(H,R)\stackrel{\mho_R}{\longrightarrow} \cA ut\left(\cX\right)(R)=\mathsf{Aut}_R\left(\cX\otimes R\right)\stackrel{\cY_R}{\longrightarrow}\mathsf{Aut}_R(A\otimes R)^{\mathsf{op}}
\end{equation}
acting, from the left-most member to the right-most, as
\begin{equation*}
f\longmapsto \left[\left(a\otimes r\right)\mapsto \left(a_0\otimes f(a_{-1})r\right)\right]
\end{equation*}
(see \cite[II, \S1, n\textsuperscript{o}2, 2.7, page 153]{DemazureGabriel}).
Set $\cO_\K:=\Calg(\K[T],-):\Calg\to \mathsf{Sets}$ and $\mathscr{O}_\K(\cX):=\mathsf{Nat}\left(\cX,\cO_\K\right)$ as in
\cite[I, \S1, n\textsuperscript{o}6, 6.1, page 26]{DemazureGabriel}. In view of Yoneda Lemma again, $\mathscr{O}_\K(\cX)\cong A$. Therefore, $\Ders{\K}{\mathscr{O}_\K(\cX)}\cong \Ders{\K}{A}$.

If we write $\K(\epsilon):=\K[T]/\langle T^2\rangle$, where $\epsilon:=T+\langle T^2\rangle$, for the $\K$-algebra of dual numbers, then $\cL ie(\cG)(\K)\subseteq \Calg(H,\K(\epsilon))$ as defined in \cite[II, \S4, n\textsuperscript{o}1, 1.2, page 200]{DemazureGabriel} is the kernel of the group homomorphism $\Calg(H,\K(\epsilon))\to\Calg(H,\K)$ given by composition with $p_1:\K(\epsilon)\to \K;\, \left[\left(a+b\epsilon\right)\mapsto a\right]$, i.e.,
\begin{equation*}
\cL ie(\cG)(\K)=\Big\{f:H\to \K(\epsilon)\mid p_1\circ f = \varepsilon\Big\}.
\end{equation*}
Set $p_2:\K(\epsilon)\to \K;\, \left[\left(a+b\epsilon\right)\mapsto b\right]$. Clearly,
\begin{equation*}
\xymatrix @!0 @R=12pt @C=130pt {
\Der{\K}{H}{\K_\varepsilon} \ar@{<->}[r]^-{\cong} & \cL ie(\cG)(\K) \\
\delta \ar@{|->}[r] & \left[\left(\varepsilon+\delta\epsilon\right):x\mapsto \left(\varepsilon(x)+\delta(x)\epsilon\right)\right] \\
p_2\circ f & f \ar@{|->}[l] }
\end{equation*}
The functor $\mho:\cG\to \cA ut(\cX)$ gives  $\cL ie(\mho)(\K):\cL ie(\cG)(\K)\to \cL ie(\cA ut(\cX))(\K)$ by restriction of the morphism $\mho_{\K(\epsilon)}:\cG(\K(\epsilon))\to \cA ut(\cX)(\K(\epsilon))$. Note that
\begin{equation*}
\cY_{\K(\epsilon)}\circ \mho_{\K(\epsilon)} : \Calg\left(H,\K(\epsilon)\right) \longrightarrow \mathsf{Aut}_{\K(\epsilon)}\left(A(\epsilon)\right)^{\mathrm{op}}; \;  \Big( f\longmapsto \left[(a+b\epsilon)\mapsto \left(a_0f(a_{-1})+b_0f(b_{-1})\epsilon\right)\right] \Big).
\end{equation*}
Now, for every $\phi \in \cL ie\left(\cA ut(\cX)\right)(\K)$, $\cF\in\mathscr{O}_\K(\cX)$ and $S\in \Calg$ one may consider
\begin{equation}\label{eq:DChiPhi}
\cD_\phi^\cX(\cF)_S:=\left(\cX(S) \stackrel{\cX(i_1)}{\longrightarrow} \cX(S(\epsilon)) \stackrel{\phi_{S(\epsilon)}}{\longrightarrow} \cX(S(\epsilon)) \stackrel{\cF_{S(\epsilon)}}{\longrightarrow} \cO_{\K}(S(\epsilon)) \stackrel{\cO_{\K}(p_2)}{\longrightarrow} \cO_{\K}(S) \right)
\end{equation}
where $i_1:S\to S(\epsilon);\,\left[s\mapsto s\right]$. Here we may apply $\phi_{S(\epsilon)}$ because $\phi\in \cL ie\left(\cA ut(\cX)\right)(\K)\subseteq \cA ut(\cX)(\K(\epsilon))=\mathsf{Aut}_{\K(\epsilon)}(\cX\otimes \K(\epsilon))$ and $S(\epsilon)$ is a $\K(\epsilon)$-algebra. This defines a map
\begin{equation}\label{eq:DXphi}
\cD^\cX_\phi:\mathscr{O}_\K(\cX)\to \mathscr{O}_\K(\cX)
\end{equation}
which turns out to be a $\K$-derivation of the algebra $\mathscr{O}_\K(\cX)$ (cf.~\cite[II, \S4, n\textsuperscript{o}2, 2.4, page 203]{DemazureGabriel}). By considering the composition
\begin{equation}\label{eq:composition}
\varpi:=\left(\Der{\K}{H}{\K_\varepsilon} \stackrel{\cong}{\to} \cL ie(\cG)(\K) \stackrel{\cL ie(\mho)(\K)}{\longrightarrow} \cL ie\left(\cA ut(\cX)\right)(\K) \stackrel{\cD_{-}^{\cX}}{\longrightarrow} \Ders{\K}{\mathscr{O}_\K(\cX)} \stackrel{\cong}{\to} \Ders{\K}{A}\right)
\end{equation}
one gets the canonical morphism claimed at the beginning of this subsection. Let us compute explicitly how this composition acts on a $\delta\in \Der{\K}{H}{\K_\varepsilon}$. The first isomorphism associates to $\delta$ the map $\varepsilon+\delta\epsilon\in\cL ie(\cG)(\K)$. Set $\phi:=\cL ie(\mho)(\K)\left(\varepsilon+\delta\epsilon\right)=\mho_{\K(\epsilon)}\left(\varepsilon+\delta\epsilon\right)\in \cL ie\left(\cA ut(\cX)\right)(\K)$. Then $\cD_{-}^{\cX}$ maps $\phi$ to $\cD_{\phi}^{\cX}\in \Ders{\K}{\mathscr{O}_\K(\cX)}$. The last isomorphism in \eqref{eq:composition} sends $\cD_{\phi}^{\cX}$ to the composition
\[
\xymatrix @R=0pt{
A \ar[r]^-{\cong} & \mathscr{O}_\K(\cX) \ar[r]^-{\cD_{\phi}^{\cX}} & \mathscr{O}_\K(\cX)\ar[r]^-{\cong} & A \\
a \ar@{|->}[r] & \cF^a= \Calg\left(\mathsf{ev}_a,-\right)  \ar@{|->}[r] & \cD_{\phi}^{\cX}(\cF^a) \ar@{|->}[r] & \left(\cD_{\phi}^{\cX}(\cF^a)_A(\id{A})\right)(T).
}
\]
Here $\mathsf{ev}_a:\K[T]\to A$ is the unique algebra map sending $T$ to $a$. Now, let us compute explicitly
\begin{align*}
& \left(\cD_{\phi}^{\cX}(\cF^a)_A(\id{A})\right)(T) \stackrel{\eqref{eq:DChiPhi}}{=} \left[\left( \cO_{\K}(p_2)\circ \cF^a_{A(\epsilon)} \circ \phi_{A(\epsilon)} \circ \cX(i_1) \right)(\id{A})\right](T) \\
 & = \left[\left( \cO_{\K}(p_2)\circ \cF^a_{A(\epsilon)} \circ \phi_{A(\epsilon)} \right)(i_1)\right](T) = \left[\left( \cO_{\K}(p_2)\circ \cF^a_{A(\epsilon)}\right)\left(\mho_{\K(\epsilon)}\left(\varepsilon+\delta\epsilon\right)_{A(\epsilon)}(i_1)\right)\right](T) \\
 & = \left[\left( \cO_{\K}(p_2)\circ \cF^a_{A(\epsilon)}\right)\left(\mu_{A(\epsilon)}(\varepsilon+\delta\epsilon,i_1)\right)\right](T) = \left[\cO_{\K}(p_2) \left(\mu_{A(\epsilon)}(\varepsilon+\delta\epsilon,i_1)\circ \mathsf{ev}_a\right)\right](T) \\
 & = \left(p_2\circ \mu_{A(\epsilon)}(\varepsilon+\delta\epsilon,i_1)\circ \mathsf{ev}_a\right)(T) = p_2\left(\mu_{A(\epsilon)}(\varepsilon+\delta\epsilon,i_1)(a)\right) = p_2\left((\varepsilon+\delta\epsilon)(a_{-1})i_1(a_0)\right) \\
 & = p_2\left(\varepsilon(a_{-1})a_0+\epsilon\delta(a_{-1})a_0\right) = \delta(a_{-1})a_0
\end{align*}

Summing up, the canonical morphism is given by
\begin{equation*}
\varpi:  \Der{\K}{H}{\K_\varepsilon} \longrightarrow  \Ders{\K}{A}: \delta\longmapsto \left[a\mapsto \delta(a_{-1})a_0\right].
\end{equation*}
Thus $\varpi = \omega \circ \tau$ as in \eqref{Eq:DG}.

Now, let us give some examples of the factorization introduced in Example \ref{Exam:Main}.

\begin{example}
If $A=\K$, then $\Derk{t}{\cH}{A_{\varepsilon_{\cH}}}=\Der{\K}{H}{\K_\varepsilon}$ and $\Ders{\K}{A}=0$, whence $\omega=0=\varpi$ and $\tau$ is the identity.
\end{example}

\begin{example}
Take $A$ to be the Hopf algebra $H$ itself with comodule structure given by $\Delta$ (this would correspond to the action of $\cG$ on itself by left multiplication). In this case, $\cH = H\otimes H$ with
\begin{gather*}
\eta_{\cH}(x\otimes y) = x_1\otimes x_2y, \qquad \Delta_{\cH}(x\otimes y) = (x_1\otimes 1)\tensor{H}(x_2\otimes y), \\
\varepsilon_{\cH}(x\otimes y) = \varepsilon(x)y, \qquad \cS(x\otimes y) = S(x)y_1\otimes y_2
\end{gather*}
and $\varpi$ satisfies $\varpi(\delta):x\mapsto \delta(x_1)x_2$ for every $\delta \in \Der{\K}{H}{\K},x\in H$. Notice that the anchor map
$$
\xymatrix @R=0pt @C=10pt {\omega:\Derk{t}{\cH}{H_{\varepsilon_{\cH}}} \ar[r] & \Ders{\K}{H}, \\
\delta \ar@{|->}[r] & \big[x\mapsto \delta(x_1\otimes x_2)=\delta(x_1\otimes 1)x_2\big]
}
$$
admits an inverse, explicitly given by
$$
\xymatrix @R=0pt @C=10pt {
\omega^{-1}:\Ders{\K}{H} \ar[r] & \Derk{t}{\cH}{H_{\varepsilon_{\cH}}} \\
d \ar@{|->}[r] & \left[x\otimes y\mapsto d(x_1)S(x_2)y\right],
}
$$
whence the factorization of the morphism $\varpi$ is trivial.
\begin{invisible}
The abstract reason why $\omega$ admits an inverse is that $\eta_\cH$ is in fact invertible. Indeed, $\omega(\delta)=\delta\circ \eta_{\cH}\circ(H\otimes u_{H})$, whence it is not difficult to deduce that an inverse for $\omega$ has to be of the form $\omega^{-1}(d) = m\circ (d\otimes H)\circ \eta_{\cH}^{-1}$ (and in fact this is the case).

For the sake of clearness, one may check directly that the inverse $\omega^{-1}$ is well-defined by means of the following computations
\begin{align*}
\omega^{-1}(d)(xy\otimes zw) & = d(x_1y_1)S(x_2y_2)zw = d(x_1)y_1S(x_2)S(y_2)zw + x_1d(y_1)S(x_2)S(y_2)zw \\
 & = \big(d(x_1)S(x_2)z\big)\varepsilon_{\cH}(y\otimes w) + \varepsilon(x\otimes z)\big(d(y_1)S(y_2)w\big), \\
\omega^{-1}(d)(t(y)) & = d(1)y = 0, \\
\omega^{-1}\omega(\delta)(x\otimes y) & = \omega(\delta)(x_1)S(x_2)y = \delta(x_1\otimes 1)x_2S(x_3)y \\
 & = \delta(x\otimes 1)y = \delta(x\otimes y), \\
\omega\omega^{-1}(d)(x) & = \omega^{-1}(d)(x_1\otimes 1)x_2 = d(x_1)S(x_2)x_3 = d(x).
\end{align*}
\end{invisible}

We recall\footnote{See \cite[II, \S4, n\textsuperscript{o}4, Proposition 4.6, page 214]{DemazureGabriel} and, for example, \cite[Corollary 4.3.2]{Abe} for the left-hand analogue.} that in this case $\varpi$ induces an anti-isomorphism of Lie algebras between $\Der{\K}{H}{\K_{\varepsilon}}$ and the Lie subalgebra of $\Ders{\K}{H}$ formed by the right-invariant derivations, where a linear operator $T:H\to H$ is said to be \emph{right-invariant} if it satisfies $\Delta\circ T=(T\otimes H)\circ \Delta$ (from a geometric point of view, \eg when $H$ is the Hopf algebra of an affine algebraic group $G$, this encodes the fact that $T$ commutes with all the right-translation operators $T_g:H\to H$ given by $\left(T_g(f)\right)(h) = f(hg)$ for all $g,h\in G$. See \eg \cite[\S12.1]{Waterhouse}).

It is easy to check that for every $\delta\in \Der{\K}{H}{\K_{\varepsilon}}$, $\varpi(\delta)$ is right-invariant. Conversely, if $d\in \Ders{\K}{H}$ is right invariant then $d(x)_1\otimes d(x)_2 = d(x_1)\otimes x_2$ and hence $d(x)=\varepsilon d(x_1)x_2=\varpi(\varepsilon d)(x)$ for every $x\in H$.
\end{example}

\begin{example}
Consider the obvious action of $\mathsf{GL}_2(\C)$ on $\C^2$. This makes of the coordinate ring $A:=\C[X_1,X_2]$ of $\C^2$ a left comodule algebra over the coordinate ring $H:=\C[Z_{i,j},\mathsf{det}(Z)^{-1}]$ of $\mathsf{GL}_2(\C)$, where $\mathsf{det}(Z)=Z_{1,1}Z_{2,2}-Z_{1,2}Z_{2,1}$. Explicitly, the Hopf algebra structure on $H$ is given by
\begin{gather*}
\Delta(Z_{1,1}) = Z_{1,1}\otimes Z_{1,1} + Z_{1,2}\otimes Z_{2,1}, \hspace{5pt} \Delta(Z_{1,2}) = Z_{1,1}\otimes Z_{1,2} + Z_{1,2}\otimes Z_{2,2}, \hspace{5pt} S(Z_{1,1}) = \frac{Z_{2,2}}{\mathsf{det}(Z)}, \hspace{5pt} S(Z_{1,2}) = -\frac{Z_{1,2}}{\mathsf{det}(Z)},\\
\Delta(Z_{2,1}) = Z_{2,1}\otimes Z_{1,1} + Z_{2,2}\otimes Z_{2,1}, \hspace{5pt} \Delta(Z_{2,2}) = Z_{2,1}\otimes Z_{1,2} + Z_{2,2}\otimes Z_{2,2}, \hspace{5pt} S(Z_{2,1}) = -\frac{Z_{2,1}}{\mathsf{det}(Z)}, \hspace{5pt} S(Z_{2,2}) = \frac{Z_{1,1}}{\mathsf{det}(Z)},\\
\end{gather*}
and $\varepsilon(Z_{i,j})=\delta_{i,j}$ for every $i,j\in\{1,2\}$, while the comodule structure on $A$ is given by
$$
\rho(X_1) = Z_{1,1}\otimes X_1 + Z_{1,2}\otimes X_2, \qquad \rho(X_2) = Z_{2,1}\otimes X_1 + Z_{2,2}\otimes X_2.
$$
For every $\delta\in\Der{\C}{H}{\C_{\varepsilon}}$, the morphism $\varpi$ satisfies
\begin{equation}\label{eq:varpi}
\varpi(\delta)(X_1) = \delta(Z_{1,1})X_1 + \delta(Z_{1,2})X_2, \qquad \varpi(\delta)(X_2) = \delta(Z_{2,1})X_1 + \delta(Z_{2,2})X_2
\end{equation}
and it factors through $\tau(\delta):Z_{i,j}\otimes X_k\mapsto \delta(Z_{i,j})X_k\in\Derk{t}{H\otimes A}{A_{\varepsilon_{H\otimes A}}}$, for $k=1,2$.
\begin{invisible}
For the sake of the unaccustomed reader, we recall that the explicit expressions for $\Delta,S,\varepsilon,\rho$ have been computed via the following argument. By definition of $\Delta$, $\Delta(Z_{1,1})=\sum \alpha_i\otimes \beta_i$ if and only if $Z_{1,1}(A\cdot B)= \sum \alpha_i(A)\beta_i(B)$ for every $A,B\in\mathsf{GL}_2(\C)$. Now, if $A=(a_{i,j})$ and $B=(b_{i,j})$, then
$$\sum \alpha_i(A)\beta_i(B) = Z_{1,1}(A\cdot B) = a_{1,1}b_{1,1}+a_{1,2}b_{2,1} = Z_{1,1}(A)Z_{1,1}(B) + Z_{1,2}(A)Z_{2,1}(B)$$
and so $\Delta(Z_{1,1}) = Z_{1,1}\otimes Z_{2,2} + Z_{1,2}\otimes Z_{2,1}$ as claimed. The other are computed analogously.
\end{invisible}

Notice that from Equation \eqref{eq:varpi} we deduce that $\varpi$ is injective and that $\varpi(\delta)$ is uniquely determined by the $2 \times 2$ complex matrix $M(\delta):=\left(m_{i,j}\right)$ with $m_{i,j}=\delta(Z_{i,j})$ for all $i,j\in\{1,2\}$. Note that the latter assignment yields a bijective correspondence
$$M : \Der{\C}{H}{\C_{\varepsilon}} \to \mathrm{Mat}_{2}(\C)$$
which satisfies
$$M([\delta,\delta']) = ([\delta,\delta'](Z_{i,j})) = (\delta(Z_{i,j}))\cdot(\delta'(Z_{i,j}))-(\delta'(Z_{i,j}))\cdot(\delta(Z_{i,j})) = [M(\delta),M(\delta')].$$
Thus $M$ is the well-known identification between the Lie algebra of the algebraic group $\mathsf{GL}_2(\C)$ and the general linear algebra $\mathfrak{gl}_2(\C)=\mathrm{Mat}_2(\C)$.
\begin{invisible}
Indeed, every $\delta$ is uniquely determined by its value on the $Z_{i,j}$ and the definition of equality between polynomials does the job. Moreover, the freeness in choosing the images of $Z_{i,j}$ gives the bijective correspondence with $\mathrm{Mat}_2(\C)$.
\end{invisible}
\end{example}



\bigskip

\textbf{Acknowledgements}
L. El Kaoutit would like to thank the members of department of Mathematics ``Giuseppe Peano'' of University of Turin for  warm hospitality and for providing a very fruitful working ambience. His stay was supported by grant PRX16/00108.

\end{document}